\documentclass[12pt, reqno]{amsart}

\usepackage[margin=1in]{geometry}
\usepackage{hyperref}
\usepackage{tikz}
\usepackage{soul}
\usepackage{comment}

\usetikzlibrary{math}
\usetikzlibrary{arrows.meta}
\usetikzlibrary{decorations.pathmorphing}

\newtheorem{theo}{Theorem}[section]
\newtheorem{prop}[theo]{Proposition}
\newtheorem{lemma}[theo]{Lemma}
\newtheorem{introtheo}{Theorem}
\theoremstyle{definition}
\newtheorem{defi}[theo]{Definition}
\newtheorem{rema}[theo]{Remark}
\newtheorem{exam}[theo]{Example}

\renewcommand{\det}{\mathrm{det}}
\renewcommand{\hat}{\widehat}
\renewcommand{\tilde}{\widetilde}

\newcommand{\be}{\begin{equation}}
\newcommand{\ee}{\end{equation}}
\newcommand{\Zqc}{Z_q}
\newcommand{\Zq}{\mathbf{Z}_q}
\newcommand{\R}{\mathbb{R}}
\newcommand{\C}{\mathbb{C}}
\newcommand{\E}{\mathbb{E}}
\newcommand{\Er}{\mathrm{E}^*}
\newcommand{\Z}{\mathbb{Z}}
\newcommand{\Tr}{\mathrm{Tr}}
\newcommand{\Det}{\mathrm{Det}}
\newcommand{\T}{\mathbb{T}}
\newcommand{\trqc}{\mathrm{tr}_q}
\newcommand{\trq}{\mathbf{tr}_q}
\newcommand{\htrq}{\hat{\mathbf{tr}}_q}
\newcommand{\Kdetq}{\mathrm{Kdet}_q}
\newcommand{\GL}{\mathrm{GL}}
\newcommand{\SL}{\mathrm{SL}}
\newcommand{\Mat}{\mathrm{Mat}}

\title{A quantum N-dimer model}

\author[D. C. Douglas]{Daniel C. Douglas}
\address{Department of Mathematics, Virginia Tech, 225 Stanger Street, Blacksburg, VA 24061}
\email{dcdouglas@vt.edu}
\author[R. Kenyon]{Richard Kenyon}
\address{Department of Mathematics, Yale University, New Haven CT 06511}
\email{richard.kenyon@yale.edu}
\author[N. Ovenhouse]{Nicholas Ovenhouse}
\address{Department of Mathematics, Michigan State University, 619 Red Cedar Road, East Lansing, MI 48824}
\email{ovenhou3@math.msu.edu}
\author[S. Panitch]{Samuel Panitch}
\address{Department of Mathematics, Yale University, New Haven CT 06511}
\email{sam.panitch@yale.edu}
\author[S. Tata]{Sri Tata}
\address{Department of Mathematics, Yale University, New Haven CT 06511}
\email{sri.tata@yale.edu}

\date{\today}

\thanks{R.K. was supported by the Simons Foundation grant 327929.}

\begin{document}

\begin{abstract}
We study a quantum version of the $n$-dimer model from statistical mechanics, based on the formalism from quantum topology   developed by Reshetikhin and Turaev (the latter which, in particular, can be used to construct the Jones polynomial of a knot in $\mathbb{R}^3$).  We apply this machinery to construct an isotopy invariant polynomial for knotted bipartite   ribbon graphs in $\mathbb{R}^3$, giving, in the planar setting, a quantum \hbox{$n$-dimer} partition function.    As one application, we compute the expected number of loops in the (classical) double dimer model for planar bipartite graphs.
\end{abstract}

\maketitle

\section{Introduction}\label{sec:intro}

In this paper we study a model blending ideas from low dimensional topology and representation theory, specifically quantum topology,   with ideas from combinatorics and probability, specifically statistical mechanics.  Our construction can be thought of as   a Jones-polynomial-like quantum invariant for bipartite graphs in two and three dimensions having strong connections to the dimer model.    While here we deal with the topologically trivial settings of $\R^2$ and $\R^3$, we hope this will be a first step in studying similar   models for 2- and 3-dimensional manifolds, where it is natural to incorporate the language of skein theory.  

\subsection{The Jones polynomial}

The birth of quantum topology largely coincided with Jones' discovery \cite{Jones} of a Laurent polynomial $J_q(K)$ in a single variable $q^\frac{1}2$ associated to each knot $K$ in $\R^3$ and independent of isotopy.  Witten \cite{Witten} showed that the Jones polynomial can be understood from the point of view of quantum field theory by associating to any knot $K$ in a closed three manifold $M$ colored by a representation of $\mathrm{SU}_n$ a quantum invariant defined by a path integral, which essentially recovers the Jones polynomial when $M=S^3$ and the color is the defining representation of $\mathrm{SU}_2$.  From the physics perspective, $q=e^{2\pi i\hbar}$ where $\hbar$ is Planck's constant, and taking the limit $\hbar\to0$, that is $q\to1$, recovers the `classical theory'.  Reshetikhin and Turaev \cite{reshetikhin1} put Witten's construction on a more solid mathematical footing by rigorously constructing the Jones polynomial and more general quantum invariants, now called Reshetikhin--Turaev invariants, associated to colored ribbon graphs in $\R^3$ in terms of the representation theory of quantum groups, such as $\mathrm{U}_q(\mathfrak{sl}_2)$ or $\mathrm{U}_q(\mathfrak{sl}_n)$.  

A powerful property of the Jones polynomial is that it is completely determined by local relations called skein relations, which describe interactions associated to crossings.  The original skein relations applied to knots, secretly encapsulating the representation theory of $\SL_2$ or $\mathrm{U}_q(\mathfrak{sl}_2)$.  A knot, being an embedded circle, is homeomorphic to any finite connected $2$-valent graph.  Kuperberg \cite{kuperberg} discovered skein relations among $3$-valent graphs, called $3$-webs, underlying the quantum invariants for $\mathrm{U}_q(\mathfrak{sl}_3)$.  Sikora \cite{sikora1}, among others \cite{ckm, moy}, generalized these skein relations to $\mathrm{U}_q(\mathfrak{sl}_n)$ based on $n$-webs, which in Sikora's description are $n$-valent graphs, providing a diagrammatic framework for $\SL_n$ quantum invariants.  

\subsection{The dimer model}

On a graph $G$ a dimer cover or perfect matching is a pairing of the vertices into neighboring pairs; equivalently,   it is a set of edges such that each vertex is the endpoint of a single edge in the set. The dimer model is the study of   natural probability measures (such as the uniform measure) on the set of dimer covers $\Omega_1(G)$.  

A celebrated theorem of Kasteleyn \cite{Kasteleyn} says that one can count dimer covers of planar graphs using the determinant of a signed adjacency matrix,   the Kasteleyn matrix.  More generally, many probabilistic quantities of interest can be computed using determinantal methods.   From this starting point, the dimer model has been shown to have deep connections with many different areas of mathematics such as   combinatorics, probability, complex analysis \cite{Kenyon.conformal}, algebraic geometry \cite{KO2, KOS}, partial differential equations   \cite{CKP, KO2, KenyonPrause}, and integrable systems \cite{GoncharovKenyon}.

Recent works \cite{cohenetal, DKS, fll19, KenyonShi, KenyonOvenhouse, lam15, tata2024rankndimermodelssurfaces} study the $n$-dimer model   (an $n$-dimer cover is an overlay of $n$ dimer covers) and connect this dimer model with webs and representation theory.   In particular, the notion of trace of an $n$-web is very closely connected with the number of proper edge $n$-colorings of the web,   as defined and discussed in the next section. This setting provides for a natural quantization, which is the main goal of this paper. 

\subsection{The classical (\texorpdfstring{$q=1$}{q=1}) case}

(Before beginning, a friendly reminder to the reader that all of the terminology of the paper is properly defined in Sections \ref{sec:preliminaries} onward.  For the sake of being self-contained, some definitions are duplicated in part in this introduction.  Lastly, for recent related work, see \cite{GPPSS1, inoue2025, russell2026}.)

Fix a natural number $n\geq1$.  For a planar bipartite (ciliated) graph $G=(V=B\cup W,E)$, an \emph{$n$-multiweb} $m$ is a function $m:E\to\{0,1,2,\dots,n\}$ which sums to $n$ at each vertex: $\sum_{v\sim v'}m(vv')=n$ for all $v\in V$.  (Note the notion of an $n$-multiweb $m$ makes sense in any abstract graph $G$.  In \cite{fll19, lam15} $n$-multiwebs are called `weblike subgraphs'.)  By a standard result, see \cite{LovaszPlummer}, every $n$-multiweb on a bipartite graph can be obtained by overlaying $n$ single dimer covers, so $n$-multiwebs are also called \emph{$n$-dimer covers}. The set of $n$-multiwebs is denoted $\Omega_n$.  

An \emph{edge $n$-coloring} $c$ of $m$ is an assignment to each edge $e$ of a subset $c(e)\subset\{1,2,\dots,n\}$ such that   $|c(e)|=m(e)$ for all $e\in E$ and $\bigcup_{v\sim v'}c(vv')=\{1,2,\dots,n\}$ for all $v\in V$.  

Let $\Mat_n(\C)$ denote the set of $n\times n$ matrices. A \emph{$\Mat_n$-connection} $\Phi$ on $G$ is a function $\Phi:E\to\Mat_n(\mathbb{C})$.    (This is not the usual usage of the word `connection', which would require the matrices to be invertible.)

In \cite{DKS}, by a standard tensor network construction \cite{penrose69}     the trace $\mathrm{tr}(\Phi,m)$ of an $n$-multiweb $m$ with respect to a $\Mat_n$-connection $\Phi$ was defined, and it was shown that when $\Phi=I$ is the identity connection,   assigning the identity matrix to every edge, then (for positive ciliations) $\mathrm{tr}(I,m)$ equals the number of edge $n$-colorings of $m$.  

Let $N:=|B|=|W|$.  Associated to every $\Mat_n$-connection $\Phi$ is an $Nn\times Nn$ matrix $K(\Phi)$ called the Kasteleyn matrix,   which is a signed adjacency matrix of $G$ weighted by the connection $\Phi$.  Generalizing Kasteleyn's theorem in the case $n=1$,   as well as an analogous result due to Kenyon for $n=2$ \cite{Kenyon.double},  the main result of \cite{DKS} says that (for positive ciliations) up to a   global sign the determinant of $K(\Phi)$ equals the $n$-dimer partition function $Z(\Phi):=\sum_{m\in\Omega_n}\mathrm{tr}(\Phi,m)$ with respect to $\Phi$.    This was applied to study the probabilities of multiwebs on some simple surfaces, such as the annulus in the case $n=3$.    In \cite{KenyonOvenhouse}, this higher rank version of Kasteleyn's theorem was further generalized to the setting of mixed $n$-dimer covers,   and was applied to other related stat mech models such as square ice.   See also \cite{AEKOW} for further work on the higher rank dimer model.

\subsection{The \texorpdfstring{$q=q$}{q=q} case:  three dimensions}

(Throughout the paper, $n=n$ means $n\geq1$, and $q=q$ means $q\neq0$ excluding certain roots of unity depending on $n$.)  A bipartite ribbon graph $\mathbf{G}$ in $\R^3$ (we use boldface to denote objects in $\R^3$) is a compact oriented surface obtained by gluing   together disks and rectangles, where the disks, either black or white, are the `vertices' and the rectangles are the `edges'.    An $n$-web $\mathbf{W}$ in $\R^3$ is an $n$-valent bipartite ribbon graph (possibly with additional vertex-free loop components).    Sikora \cite{sikora2, sikora1} assigned to every (ciliated) $n$-web $\mathbf{W}$ a Laurent polynomial $\trq(\mathbf{W})$   in $q^{1/n}$ called the  \emph{quantum trace}, constructed as a specific Reshetikhin--Turaev invariant.  An $n$-multiweb $m$ in $\mathbf{G}$ determines, by splitting edges of weight $m(e)$ into $m(e)$ parallel edges, an $n$-web $\mathbf{W}_m$.  Slightly generalizing Sikora's construction, we define the quantum trace $\trq(m)$ of $m$ to be the quantum trace $\trq(\mathbf{W}_m)$ divided by $\prod_{e\in E}[m(e)]!$ where $[m(e)]!$ denotes the quantum factorial.  A priori, this is a rational expression in $q^\frac1n$.  

\begin{introtheo}
The quantum trace $\trq(m)$ of an $n$-multiweb $m$ in a bipartite (ciliated) ribbon graph $\mathbf{G}$ in $\R^3$ is a Laurent polynomial in $q^\frac{1}{n}$.
\end{introtheo}

We define the quantum partition function $\Zq=\mathbf{Z}_{q,n}$ of $\mathbf{G}$ to be 
\be\label{Zqdef}
\Zq := q^{-N\binom{n}{2}}\sum_{m\in\Omega_n}\trq(m)
\ee
where the sum is over all $n$-multiwebs $m$ in $\mathbf{G}$.  (Recall that $N$ is half the number of vertices.)  When $n=1$ then $\trq(m)=+1$ for all $m$ and $q$, so $\Zq=Z_\mathrm{dimer}=\#\{\text{dimer covers of }\mathbf{G}\}$ is the classical dimer partition function.  It follows that the quantum model reduces to the classical model in the $n=1$ case.  While the topology of single dimer covers is not terribly interesting (but the statistics is!), the topology becomes much more interesting when $n\geq2$ since laying down multiple dimer covers yields web-like objects; here, the quantum invariant truly deforms the classical one.  

This definition of $\Zq$ is a special case of a more general definition valid for $\mathbf{G}$ in any oriented three manifold $M$, which we briefly describe now.    Sikora's $\SL_n$ skein relations allow one to define the \emph{$n$-skein space}, where an element is a formal linear combination of $n$-webs $\mathbf{W}$ modulo   the skein relations (similar to a homology theory where an element of homology is a formal linear combination of cycles modulo relations).    Then $\Zq$ can be defined in exactly the same way as in (\ref{Zqdef}) except that $\trq(m)$ is replaced by the skein   $\frac{1}{\prod_{e\in E}[m(e)]!} \left<\mathbf{W}_m\right>$ of the $n$-web $\mathbf{W}_m$ in the $n$-skein space   (and the sum is possibly a signed sum, depending on the topology of the manifold $M$).  The connection to the topologically trivial setting when $M=\R^3$ is that in   this setting the $n$-skein space is isomorphic to the complex numbers $\mathbb{C}$, and through this isomorphism the   skein $\frac{1}{\prod_{e\in E}[m(e)]!} \left<\mathbf{W}_m\right>$ becomes exactly the quantum trace $\trq(m)$.
  
\subsection{The \texorpdfstring{$q=q$}{q=q} case:  two dimensions}

For a planar bipartite (embedded ciliated) graph $G$, we provide quantum versions of the $q=1$ results of \cite{DKS, KenyonOvenhouse}.  We first define the notion of an \emph{(edge-commuting) $n$-quantum connection} $\Phi_q$, where an $n\times n$ quantum matrix of $q$-commuting variables is assigned to each edge of the graph.  For the constructions of the paper to be well-defined, in particular to deal with the ordering of the noncommuting variables, we assume the fairly restrictive condition that the quantum matrix entries from different edges commute.  Fix $\mathbb{A}$ to be the algebra generated by all these quantum matrix entries.  For every $n$-multiweb $m$ we define the quantum trace $\trqc(\Phi_q,m)$ of $m$ with respect to the $n$-quantum connection $\Phi_q$.  The associated quantum partition function is 
$$
\Zqc(\Phi_q):=q^{-N\binom{n}{2}}\sum_{m\in\Omega_n}\trqc(\Phi_q,m).
$$

Note by inserting $\R^2$ linearly into $\R^3$ that $G$ can naturally be considered as a ribbon graph $\mathbf{G}$.  Building on results of Sikora, we show that there exists an $n$-quantum connection $\Phi_q=I_q$ called the \emph{quantum identity connection}, which is unique up to diagonal gauge transformations, such that the two versions $\trq(m)=\trqc(I_q,m)$ coincide.  It follows that $\Zq=\Zqc(I_q)$ as well.  

Call a Laurent polynomial $L(q)$ \emph{symmetric} if $L(q)=L(q^{-1})$, and call a polynomial $P(q)$ \emph{palindromic} if there exists a nonnegative integer or half-integer $\alpha$ such that $q^{-\alpha} P(q)$ is symmetric.  By using this identification of the three dimensional and two dimensional quantum traces, together with an analysis of the ciliation data, we prove the following result.

\begin{introtheo}
The quantum trace $\trq(m)$ of an $n$-multiweb $m$ in a planar bipartite (embedded ciliated) graph $G$ in $\R^2$ is a palindromic polynomial in $q$.    The shifting exponent $\alpha$ is the same for all $m\in\Omega_n$ and is equal to $N\binom{n}{2}$.    In particular, $\Zq$ is a symmetric Laurent polynomial in $q$.    Moreover, for positive ciliations $\Zq^+:=\Zq$ is nonzero with nonnegative integer coefficients, and is independent of the choice of positive ciliation and planar embedding (so only depends on the combinatorial structure of the graph).
\end{introtheo}

When $q=1$ and for general $n$, then $\mathbf{Z}_1^+=(Z_\mathrm{dimer})^n$ recovers the classical $n$-dimer partition function: see \cite{DKS}.

Proceeding to the last main result, assume for the moment that $G$ is simple: $G$ has no parallel edges.  Given Kasteleyn signs $\epsilon$ for $G$ (depending on the ciliation when $n$ is even) and an $n$-quantum connection $\Phi_q$, we define the \emph{$q$-Kasteleyn matrix} $K(\Phi_q)\in \Mat_N(\Mat_n(\mathbb{A}))$ by $K(\Phi_q)_{wb}:=\epsilon(bw)\Phi_q(bw)$.   Let  $\tilde{K}=\tilde{K}(\Phi_q)\in \Mat_{Nn}(\mathbb{A})$ be formed from $K(\Phi_q)$ in the obvious way by thinking of each entry $K(\Phi_q)_{bw}$ as an $n\times n$ block of elements of $\mathbb{A}$. This $\tilde{K}$ is not a quantum matrix despite the fact that the $\Phi_q(bw)$ are, so the quantum determinant of $\tilde K$ is not defined.   We define however the \emph{$q$-Kasteleyn determinant}  $\Kdetq(\Phi_q)\in\mathbb{A}$ by
$$
\Kdetq(\Phi_q):=q^{-N\binom{n}{2}}\sum_{\tilde{\sigma}\in\mathfrak{S}_{Nn}}(-1)^{\tilde{\sigma}}\left(\prod_{v\in V}q^{\ell(\sigma_v)}\right)\left(\prod_{e\in E}q^{\binom{m_e}{2}}q^{\ell(\sigma_e)}\right)\tilde{K}_{1\tilde{\sigma}(1)}\tilde{K}_{2\tilde{\sigma}(2)}\dots\tilde{K}_{Nn\tilde{\sigma}(Nn)}
$$
where the local permutations $\sigma_v$ and $\sigma_e$ are defined in Section \ref{sec:quantumkasteleynconnection}.  When $G$ is not simple, a slightly more general formula for $\Kdetq(\Phi_q)$ holds.  Note that $\Kdetq(\Phi_q)$ reduces to the classical Kasteleyn determinant formula when $q=1$.  

\begin{introtheo}\label{mainthm}
For a planar bipartite (embedded ciliated) graph $G$ in $\R^2$ one has  
\be\label{detK=Z}
\Zqc(\Phi_q)=\pm\Kdetq(\Phi_q).  
\ee
\end{introtheo}

As an example, listed below are the values of $\Zq^+=\mathbf{Z}_{q,n}^+(G)$ for $n=1,2,3,4$ for a square grid graph $G$.  (Computed, for instance, using the code linked in Appendix \ref{app:code}.)  Note that $Z_\mathrm{dimer}=36$ and $\mathbf{Z}_{1,n}^+(G)=36^n$ for each $n$ as required.  For more examples, see Section \ref{sec:examples}.  In parting, it is tempting to seek a conceptual (e.g. categorical \cite{sikoracateg, fendleykrushkal, khovanovcateg, kronheimermrowkacateg, laudaqueffelecrosecateg}) meaning for the natural number coefficients appearing in $\Zq^+$.  

\begin{center}
\includegraphics[height=1.33in]{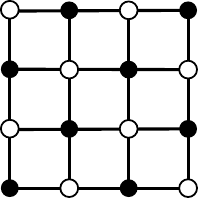}

\bigskip

$\mathbf{Z}_{q,1}^+(G)=36$, 

\bigskip

$\mathbf{Z}_{q,2}^+(G)=q^4+24q^3+145q^2+316q+324+316q^{-1}+145q^{-2}+24q^{-3}+q^{-4}$, 

\bigskip

$\mathbf{Z}_{q,3}^+(G)=34q^8+96q^7+526q^6+896q^5+2784q^4+3136q^3+6450q^2+5280q+8252+5280q^{-1}+6450q^{-2}+3136q^{-3}+2784q^{-4}+896q^{-5}+526q^{-6}+96q^{-7}+34q^{-8}$,

\bigskip

$\mathbf{Z}_{q,4}^+(G)=q^{16}+24q^{15}+144q^{14}+368q^{13}+1373q^{12}+2628q^{11}+7145q^{10}+11228q^9+24756q^8+33980q^7+60629q^6+73576q^5+110412q^4+120068q^3+155022q^2+151408q+174092+151408q^{-1}+155022q^{-2}+120068q^{-3}+110412q^{-4}+73576q^{-5}+60629q^{-6}+33980q^{-7}+24756q^{-8}+11228q^{-9}+7145q^{-10}+2628q^{-11}+1373q^{-12}+368q^{-13}+144q^{-14}+24q^{-15}+q^{-16}$.  
\end{center}

\subsection{Probability}

Theorem \ref{mainthm} is useful because it gives a `state sum' formula for $\Zq=\Zqc(I_q)$. Although computing  either side of (\ref{detK=Z}) is computationally intractable for large planar graphs, the statement nonetheless has real applications.

Consider for example the case $n=2$. A $2$-multiweb $m$ is a collection of simple loops and doubled edges.   In this case   (for positive ciliations) the trace $\trq(m)$ is simply $\trq(m)=q^N[2]^{L(m)}$ where $[2]:=q+\frac1q$ and $L(m)$ is the  number of loops in $m$, see Section \ref{sssec:probabilitiesn=2}. Consequently the quantum partition function (\ref{Zqdef}) for $n=2$ has a particularly nice form
$$
\Zq^+=\sum_{m\in\Omega_2} [2]^{L(m)}.
$$

One unforeseen application of the $n=2$ version of Theorem \ref{mainthm} is a new calculation of a classical quantity (which, to the authors' understanding, is the first computation of this kind for the instances discussed below):  the distribution of the number of  loops in the (classical) double dimer model. Indeed, for a bipartite planar graph with quantum identity connection, $\Zq^+$ is the partition function for the double dimer model with weight $[2]=q+\frac1q$ per loop.  By expanding $q$ near $1$, the `local' computation of $\Zq^+$ provided by Theorem \ref{mainthm} allows us to  in principle compute, using determinantal methods,  all moments of the distribution of the number of loops for the classical double dimer model (that is, the case $q=1$). 

We carry out this computation for the expected density $\rho$ of double dimer loops per vertex for the infinite honeycomb graph, giving an exact infinite series formula (\ref{rhosum}), which numerically is indistinguishable from $\rho=\frac1{27}$.  See Section \ref{loopdensity}. We conjecture that $\rho=\frac1{27}$ is indeed the exact value. 

A similar computation for the density of loops in the classical double dimer model on the graph $\Z^2$ gives density $\rho$ numerically indistinguishable from $\frac{1}{16}$.  We do not have any rigorous explanation for the apparent simplicity of these densities.

In the $n=n$ case, we define (see Section \ref{ssec:randomvariable}) a $\mathbb{Q}_{\geq0}$-valued random variable $X_n=X_n(m)$ using the quantum trace of the multiweb $m\in\Omega_n$ such that $X_2(m)=L(m)$, and satisfying other natural properties.  We call $X_n(m)$ the \emph{circulation} of the multiweb $m$.  In particular, the variable $X_n$ enjoys the property 
$$
\E(X_n)=\frac{1}{\mathbf{Z}_1^+}\left.\frac{d^2}{dq^2}\right|_{q=1}\Zq^+
$$
computing the expectation value of $X_n$ with respect to the natural probability measure for the $n$-dimer model (note the first derivative of $\Zq^+$ at $q=1$ vanishes by symmetry).  In the case $n=2$, this is the essential property underlying our computation of the expected density $\rho$ of loops discussed above.  

\section*{Acknowledgements}

We thank Tommaso Cremaschi, Vladimir Fock, Alexander Goncharov, Vijay Higgins, Andrew Neitzke, Adam Sikora, and David Speyer for helpful discussions.

\section{Preliminaries}
\label{sec:preliminaries}

\subsection{Graphs}

\begin{defi}\label{defi:graphassumptions}
By \emph{bipartite graph} (or just \emph{graph}) $G$ we mean a finite bipartite graph $G$ with vertex set $V=B\cup W$ and edge set $E$.  The graph $G$ is not assumed to be simple: multiple edges between vertices are allowed.  It is assumed that the vertices have been colored black and white, where the black vertices are denoted $b\in B$ and the white vertices $w\in W$.  The edges are oriented from the black vertices to the white vertices.  It is assumed that there are $N$ black vertices and $N$ white vertices, each set of which is labeled from $1$ to $N$ arbitrarily, so $B=\{b_1,b_2,\dots,b_N\}$ and $W=\{w_1,w_2,\dots,w_N\}$.  Edges are denoted $e\in E$ or by $e=bw$ when $G$ is assumed to be simple.  

A \emph{dimer cover} (or \emph{perfect matching}) of $G$ is subset of edges of $G$ such that each vertex of $G$ is the endpoint of exactly one of the edges in the subset.  Denote by $\Omega_1$ the set of dimer covers of $G$.  The \emph{dimer partition function} $Z_\mathrm{dimer}=|\Omega_1|$ counts the number of dimer covers of $G$.  Only graphs $G$ for which $\Omega_1$ is nonempty are considered.  

Without loss of generality, it may also be assumed that every edge of $G$ appears in some dimer cover.  Indeed, if an edge does not appear in any dimer cover, then it may be removed from $G$ without affecting the model (which is defined only in terms of dimer covers of $G$ and their constituent edges).  A particular consequence of this is that (the components of) $G$ may also be assumed to be (vertex) $2$-connected.  This is because if a vertex $v$ is a cut point of $G$, then at least one of the edges incident to $v$ will not appear in any dimer cover of $G$.  

A \emph{ciliation} $L$ of $G$ is the choice, for every vertex of $G$, of a linear ordering of the incident half-edges at that vertex.  A graph $G$ is \emph{ciliated} if it is equipped with a ciliation $L$.  
\end{defi}

\subsection{Ribbon graphs}

\begin{defi}

A \emph{bipartite ribbon graph} (or just \emph{ribbon graph}) $\mathbf{G}$ in $\R^3$ is, informally, an oriented surface-with-boundary obtained by gluing rectangles to black and white disks.  More precisely, $\mathbf{G}$ is a bipartite graph $G$ embedded in $\R^3$ and equipped with a \emph{framing}, namely a vector field orthogonal to edges away from vertices that extends continuously to the vertices.  The graph $G$ is the \emph{core} of the ribbon graph $\mathbf{G}$.  It is assumed that in a neighborhood of each vertex the incident half-edges are coplanar.  It follows that the framing is orthogonal to this plane at each vertex, so it makes sense to talk about the counterclockwise (CCW) or clockwise (CW) cyclic ordering of the half-edges at a vertex according to the right hand or left hand rule.  Choose the CCW (resp. CW) cyclic order at black (resp. white) vertices.  Additionally choosing, for every vertex, a preferred half-edge determines, in combination with the cyclic order for half-edges at that vertex, a corresponding linear order for these half-edges, and consequently determines a ciliation $L$ of $G$.  This is what is meant by a ciliation $L$ of $\mathbf{G}$.  (To be specific, say the preferred half-edge at each vertex comes first in the linear order.)  A ribbon graph $\mathbf{G}$ has \emph{blackboard framing} if the framing vector is constant in the upward vertical direction, that is, is the constant vector $(0,0,1)$.  Ribbon graphs $\mathbf{G}$ are considered up to \emph{isotopy} through the family of ribbon graphs.  Similarly, it makes sense to talk about ciliated ribbon graphs up to isotopy.  As usual, $\mathbf{G}$ is often conflated with its isotopy class.
\end{defi}

\begin{rema}
\begin{enumerate}
\item
One can go back and forth between the oriented surface interpretation of a ribbon graph and the graph-with-framing interpretation by, in the backward direction, taking a thin compact two dimensional neighborhood of the graph oriented such that the orientation normal vector points in the direction of the framing, and conversely.
\item
In practice, a ciliation $L$ of $\mathbf{G}$ is chosen by indicating a little hair-like cilia on the boundary of each vertex disk disjoint from the attaching half-edges (the preferred half-edge being the first half-edge appearing after the cilia in the cyclic order).
\item
A ribbon graph can always be isotoped, by introducing kinks, to have blackboard framing.
\end{enumerate}
\end{rema}

\begin{exam}
A perspective of a ciliated ribbon graph $\mathbf{G}_1$ is shown in Figure \ref{fig:ribbongraph}, left, where the positive $z$-direction points out of the page toward the eye of the reader.  The CCW linear order around the black vertex indicates that, outside of the right handed twist on the rightmost edge, $\mathbf{G}_1$ has the blackboard framing.  As shown in Figure \ref{fig:ribbongraph}, middle, $\mathbf{G}_1$ can be isotoped to a ciliated ribbon graph $\mathbf{G}_2$ with blackboard framing such that the right handed twist of $\mathbf{G}_1$ has been replaced by a positive kink.  
\end{exam}

\begin{defi}\label{defi:diagram}
The projection to $\R^2\cong\R^2\times\{0\}\subset\R^3$ by forgetting the third coordinate of a ribbon graph $\mathbf{G}$ with blackboard framing and in generic position, where the genericity can be achieved by arbitrarily small isotopy, determines a \emph{diagram} of $\mathbf{G}$ including the additional over/under information at each crossing.  See Figure \ref{fig:ribbongraph}, right.  Similarly, ciliated ribbon graphs have ciliated diagrams, where the half-edges at black and white vertices are linearly ordered CCW and CW.  Conversely, identifying the $x$- and $y$-axes of $\R^2$ with those of $\R^3$, a ciliated diagram determines a ciliated ribbon graph with blackboard framing up to isotopy.  (As usual, $\mathbf{G}$ is often conflated with its diagram.)
\end{defi}

\begin{rema}\label{rema:ribbongraphequalsciliatedribbongraph}
\begin{enumerate}
\item
Reidemeister's theorem for ciliated ribbon graphs \cite{fs22, sikora1} says that two ciliated ribbon graphs $\mathbf{G}_1$ and $\mathbf{G}_2$ are isotopic if and only if their diagrams are related by a sequence of ambient planar isotopies together with the ciliated framed Reidemeister moves, displayed in Figure \ref{fig:framedReidemeistermoves}.  In the first move, the vertex can be either black or white, and can go either over or under the strand, the over case being shown.
\item
Ciliated ribbon graphs will henceforth be displayed through their diagrams.  The edge orientations, always from black to white, might also be indicated on the diagrams for visual clarity.
\item\label{item:rema:ribbongraphequalsciliatedribbongraph}
Since all ribbon graphs appearing in this paper are ciliated, from now on a `ribbon graph' $\mathbf{G}$ means a `ciliated ribbon graph' $(\mathbf{G},L)$.  
\end{enumerate}
\end{rema}

\begin{figure}[t]
\centering
\includegraphics[height=.25\textheight]{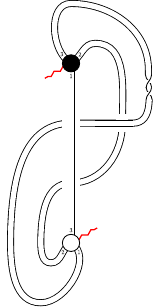}\hskip2cm\includegraphics[height=.25\textheight]{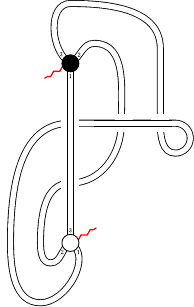}\hskip2cm\includegraphics[height=.25\textheight]{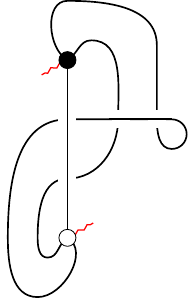}
\caption{Left: Ciliated ribbon graph with a right handed twist. Middle: Ciliated ribbon graph with a positive kink and the blackboard framing. Right: Diagram of a ciliated ribbon graph.}
\label{fig:ribbongraph}
\end{figure}

\begin{figure}[t]
\centering
\begin{minipage}{0.25\textheight}
\centering
\includegraphics[width=\linewidth]{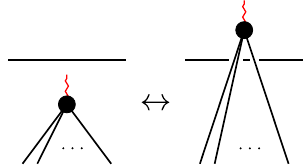}
\end{minipage}
\hspace{2cm}
\begin{minipage}{0.25\textheight}
\centering
\includegraphics[width=\linewidth]{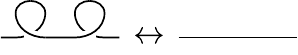}
\end{minipage}
\vspace{0.5cm}
\begin{minipage}{0.25\textheight}
\centering
\includegraphics[width=\linewidth]{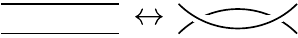}
\end{minipage}
\hspace{2cm}
\begin{minipage}{0.25\textheight}
\centering
\includegraphics[width=\linewidth]{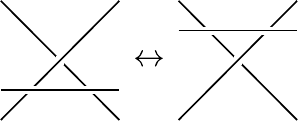}
\end{minipage}
\caption{Ciliated framed Reidemeister moves.}
\label{fig:framedReidemeistermoves}
\end{figure}

\subsection{Webs and traces}

\begin{defi}\label{defi:quantuminteger}
Throughout the paper, fix a nonzero complex number $q$.  It will be required that $q$ is not a $(2m)$-root of unity for $m=2,3,\dots,n$ but $q=\pm1$ is allowed.    For a natural number $k$ define the \emph{quantum natural number} (or \emph{quantum integer}) $[k]=\sum_{i=1}^k q^{-k-1+2i}=\frac{q^k-q^{-k}}{q-q^{-1}}$,   noting $[k]=k$ if $q=1$ and $[k]=(-1)^{k+1}k$ if $q=-1$.  Also put $[0]=0$.  Fix as well an $n$th root $q^\frac{1}{n}$.  

Later on, the following notions will also be required.  For a natural number $k$ the \emph{quantum factorial} is defined to be $[k]!=[k][k-1]\dots[2][1]$.  Also $[0]!=1$.  Similarly, the \emph{quantum binomial coefficient} is $\begin{bmatrix}m\\k\end{bmatrix}=\frac{[m]!}{[k]![m-k]!}$.

Throughout the paper, fix a natural number $n\in\{1,2,3,\dots\}$.  A \emph{proper $n$-web} (or just \emph{web}) $\mathbf{W}$ in $\R^3$ is the disjoint union of an $n$-valent ribbon graph (Remark \ref{rema:ribbongraphequalsciliatedribbongraph} (\ref{item:rema:ribbongraphequalsciliatedribbongraph})) and a (possibly empty) collection of embedded vertex-free framed loops.  This is in contrast to a ribbon graph $\mathbf{G}$, all of whose components have vertices.  There is also the \emph{empty web} $\emptyset$.
\end{defi}

\begin{rema}
\begin{enumerate}
\item
Here, the adjective `proper' refers to the fact that the core graph of the web $\mathbf{W}$ is $n$-valent.  This is as opposed to the combinatorial notion of a multiweb $m$, discussed in the next subsection.  To emphasize: in contrast to webs $\mathbf{W}$, which are $n$-valent, ribbon graphs $\mathbf{G}$ may have arbitrary valency, independent of $n$.
\item
The notions of isotopy, blackboard framing, and diagrams are defined for webs just as for ribbon graphs, and the Reidemeister moves for webs are also the same.  The empty web $\emptyset$ is the unique representative of its isotopy class.
\end{enumerate}
\end{rema}

\begin{exam}
There are ribbon graphs that are not webs, and webs that are not ribbon graphs.  But the ribbon graph depicted in Figure \ref{fig:ribbongraph} happens to be a $3$-web.  
\end{exam}

\begin{defi}
The \emph{quantum trace} $\trq(\mathbf{W})\in\mathbb{C}$ of a web $\mathbf{W}$ is defined by the following theorem of Sikora.
\end{defi}

\begin{theo}[{\cite{sikora1}}]\label{theo:qtraceofweb}
There exists a unique function $\trq$, which is a Laurent polynomial in $q^\frac{1}{n}$, from the set of webs to the complex numbers satisfying the following properties:
\begin{enumerate}
\item
$\trq(\mathbf{W})=\trq(\mathbf{W}^\prime)$ if $\mathbf{W}$ and $\mathbf{W}^\prime$ are isotopic.
\item
$\trq(\emptyset)=1$.
\item
$\trq(\mathbf{W}^\prime)=[n]\trq(\mathbf{W})$ when $\mathbf{W}^\prime$ is the disjoint union of $\mathbf{W}$ with a trivially framed unknot that is unlinked with $\mathbf{W}$.    
\item
$\trq$ satisfies the three local `skein relations' depicted in Figures \ref{fig:homflypt}, \ref{fig:framing}, and \ref{fig:determinant}.  Note that these relations involve the values of multiple webs.  Here, the webs agree outside a small neighborhood, inside which they differ as shown in the figures.  (And $\trq$ is being applied to each web in the linear combination.)
\end{enumerate}
\qed\end{theo}

\begin{rema}
\begin{enumerate}
\item
Note, in particular, that the properties described in the theorem determine the value of $\trq$ on the trivially framed unknot to be $[n]$. 
\item
In Figure \ref{fig:determinant}, the half-edges at the ciliated vertices are linearly ordered according to the cilia conventions, CCW at black vertices/sources and CW at white vertices/sinks. The permutations $\sigma\in \mathfrak{S}_n$ in the figure correspond to these linear orders.  Here $\ell(\sigma)$ is the length of the permutation (namely, the number of pairs $(i,j)$ with $i<j$ such that $\sigma(i)>\sigma(j)$, equivalently, the number of crossings in a minimal crossing diagram for $\sigma$).  Lastly, $\tilde{\sigma}_+$ indicates the positive braid lifting $\sigma$: for each crossing of strands, the strand coming up from the left crosses over the strand coming up from the right. 
\end{enumerate}
\end{rema}

\begin{figure}[t]
\centering
\includegraphics[width=.75\textwidth]{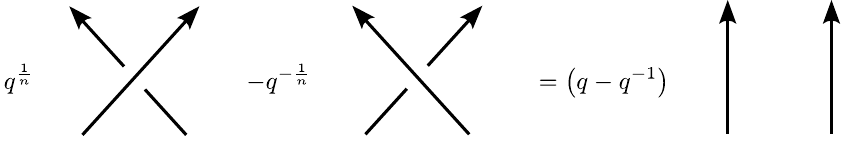}
\caption{Crossing change relation.}
\label{fig:homflypt}
\end{figure}

\begin{figure}[t]
\centering
\includegraphics[width=.25\textwidth]{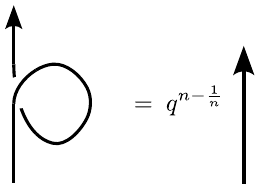}
\caption{Positive kink removing relation.}
\label{fig:framing}
\end{figure}

\begin{figure}[t]
\centering
\includegraphics[width=.75\textwidth]{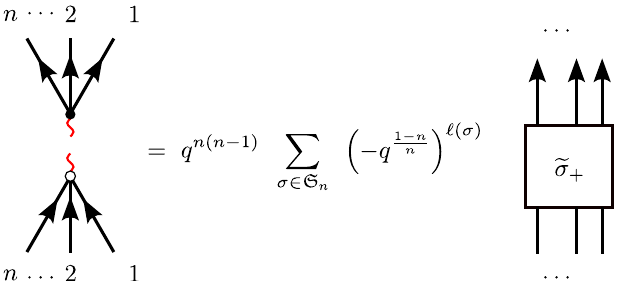}
\caption{Source-sink removing relation.}
\label{fig:determinant}
\end{figure}

\subsection{Multiwebs}

\begin{defi}
An \emph{$n$-multiweb} (or just \emph{multiweb}) $m$ in a graph $G$ is a function $m:E\to\{0,1,2,\dots,n\}$ such that for every vertex $v\in V$ the sum of $m(e)$ varying over edges $e\in E$ incident to $v$ equals $n$.  The notation $m_e$ instead of $m(e)$ for the edge multiplicity of a multiweb $m$ will be used.  An \emph{edge} of $m$ is an edge $e$ with $m_e>0$.  The set of multiwebs in $G$ is denoted $\Omega_n=\Omega_n(G)$.  

A multiweb $m$ is \emph{proper} if $m_e=0,1$ for all edges $e\in E$.  

Note that a dimer cover of $G$ is the same thing as a $1$-multiweb $m\in\Omega_1$.  The set $\Omega_n$ is nonempty exactly when $\Omega_1$ is nonempty \cite{LovaszPlummer}.

A multiweb $m$ in a ribbon graph $\mathbf{G}$ (Remark \ref{rema:ribbongraphequalsciliatedribbongraph} (\ref{item:rema:ribbongraphequalsciliatedribbongraph})) means a multiweb in its core $G$, and $\Omega_n$ is the set of such multiwebs.  A multiweb $m\in\Omega_n$ in $\mathbf{G}$ determines a corresponding \emph{split web} $\mathbf{W}_m$, which is the web in $\R^3$ obtained in the obvious way by deleting the edges $e$ of $\mathbf{G}$ where $m_e=0$, and splitting the edges of multiplicity $m_e>1$ into $m_e$ parallel copies.  Here, the ciliation $L$ of $\mathbf{G}$ restricts to a ciliation, also called $L$, of the split web $\mathbf{W}_m$ in the obvious way.  
\end{defi}

\begin{rema}
Imagining the ribbon graph $\mathbf{G}$ as comprised of actual ribbons and disks, the procedure to form the split web $\mathbf{W}_m$ from a multiweb $m$ would be to rip off the edges where $m_e=0$ and to make $m_e-1$ cuts along the edges of multiplicity $m_e>0$ to form $m_e$ parallel edges.  Note, in particular, that the cilia, attached to the disk boundaries away from the edges attachments, are untouched during this construction. 
\end{rema}

\begin{defi}\label{defi:firstquantumtraceofmultiweb}
The \emph{quantum trace} $\trq(m)\in\mathbb{C}$, actually depending on $q^\frac{1}{n}$, of a multiweb $m\in\Omega_n$ in a ribbon graph $\mathbf{G}$ is defined by
$$
\trq(m)=\frac{\trq(\mathbf{W}_m)}{\prod_{e\in E}[m_e]!}.
$$
Note, in particular, that the assumptions for $q$, see Definition \ref{defi:quantuminteger}, imply the denominator in the formula for $\trq(m)$ is nonzero.
\end{defi}

\begin{rema}\label{rema:afterdefof3dquantumtrace}
\begin{enumerate}
\item
Since the numerator $\trq(\mathbf{W}_m)$ is a Laurent polynomial in $q^\frac{1}{n}$ and the denominator $\prod_{e\in E}[m_e]!$ is a Laurent polynomial in $q$, the quantum trace $\trq(m)$ of a multiweb is, a priori, a rational expression in $q^\frac{1}{n}$.  Of course, when the multiweb is proper then the denominator is equal to $1$, so the quantum trace is Laurent in $q^\frac{1}{n}$.  
\item\label{item:rema:afterdefof3dquantumtrace}
There are other possibilities for the denominator normalization, such as $\prod_{e\in E}q^{\binom{m_e}{2}}[m_e]!$.  As it happens, our chosen denominator normalization is more suited to studying the probabilities of random multiwebs $m$ in the graph $\mathbf{G}$ in the planar setting.  
\end{enumerate}
\end{rema}

\begin{theo}\label{theo:polynomial}
For any ribbon graph $\mathbf{G}$ and multiweb $m\in\Omega_n$, the quantum trace $\trq(m)$ is a Laurent polynomial in $q^\frac{1}{n}$.
\end{theo}

\begin{proof}
This is by a slight generalization of the construction of \cite{sikora1} valid for `stated webs'.  Details are provided in Section \ref{sssec:proofofLaurentproperty}.  
\end{proof}

\begin{prop}[\cite{sikora1}]\label{prop:ciliadependence}
For $n$ odd, the quantum trace $\trq(\mathbf{W})$ of a web $\mathbf{W}$ is independent of the ciliation $L$.  The same is true for traces $\trq(m)$ of multiwebs $m$.

For $n$ even, if webs $\mathbf{W}$ and $\mathbf{W}^\prime$ differ only in that a single cilium has been rotated by one `click', then $\trq(\mathbf{W})=-\trq(\mathbf{W}^\prime)$.  If multiwebs $m,m'$ differ only in that a single cilium has been rotated past one edge $e$, then $\trq(m)=(-1)^{m_e}\trq(m')$.  
\qed\end{prop}

The following statement can be readily verified from either the skein relations (Theorem \ref{theo:qtraceofweb}) or the alternative state-sum construction (see e.g. Section \ref{sssec:proofofLaurentproperty}).

\begin{prop}\label{prop:multiplicative}
The quantum trace is multiplicative, in the following sense.  Given a multiweb $m$, assume $\mathbf{W}_m=\mathbf{W}_1\cup\mathbf{W}_2$ is a disjoint union of (groupings of connected) components $\mathbf{W}_1$ and $\mathbf{W}_2$ that are separable in $\R^3$, meaning they can be isotoped into separate balls.  Then $\trq(m)=\trq(m_1)\trq(m_2)$, where the multiwebs $m_i$ corresponding to the webs $\mathbf{W}_i$ are not multiwebs in $\mathbf{G}$ but in corresponding subgraphs $\mathbf{G}_i$.
\qed\end{prop}

\begin{defi}
For any ribbon graph $\mathbf{G}$, the  \emph{quantum partition function} $\Zq\in\mathbb{C}$ is the Laurent polynomial in $q^\frac{1}{n}$ defined by
$$
\Zq=q^{-N\binom{n}{2}}\sum_{m\in\Omega_n} \trq(m).
$$
\end{defi}

\begin{rema}\label{rema:afterfirstdefofqpartitionfunction}
\begin{enumerate}
\item
Note $\Zq$ is independent of isotopy of $\mathbf{G}$.  It does however depend on the ciliation $L$ of $\mathbf{G}$ when $n$ is even, by Proposition \ref{prop:ciliadependence}.  
\item
The normalization factor $q^{-N\binom{n}{2}}$, depending on $n$ and half the number of vertices $N$, is chosen to simplify calculations later on.  (See also \cite[Section 11.2]{LESIKORA}).
\item
The proof of Theorem \ref{theo:qtraceofweb} constructs the quantum trace $\trq(\mathbf{W})$ of a web $\mathbf{W}$ as a certain Reshetikhin--Turaev invariant \cite{reshetikhin1}, as briefly discussed in Section \ref{sec:intro}. In particular, this is the case for $\trq(m)$ as well when $m$ is a proper multiweb.  Motivated by Theorem \ref{theo:polynomial}, the authors expect that it should not be too difficult to show that the quantum partition function $\Zq$ can also be formulated as a Reshetikhin--Turaev invariant.  
\item\label{item:rema:afterfirstdefofqpartitionfunction}
When $n=1$, it is not hard to show from Theorem \ref{theo:qtraceofweb} that $\trq(m)=+1$ for every dimer cover $m\in\Omega_1$ and for all $q$.  Consequently, $\Zq=Z_\mathrm{dimer}$ is the classical dimer partition function of $G$, namely, the number of dimer covers of $G$.  In particular, the quantum model reduces to the classical model in this case.  Strong connections will be made to the $n$-dimer model for $n>1$, in the planar setting, in Section \ref{sec:planarsetting}.
\item  
For any oriented $3$-manifold $M$, the notions of ribbon graphs $\mathbf{G}$ and webs $\mathbf{W}$ in $M$ are defined in the same way as in $\R^3$.  Sikora \cite{sikora1} defined the `$n$-skein space' as the quotient of the free complex vector space on the set of isotopy classes of webs in $M$ by the local skein relations depicted in Theorem \ref{theo:qtraceofweb}.  The $n$-skein space is a quantum deformation of the $\SL_n$ character variety of $M$, see for instance \cite[Corollary 20]{sikora1}.  Generalizing the above construction, a quantum invariant $\Zq$ can be assigned to any ribbon graph $\mathbf{G}$ in $M$ by replacing each appearance of $\trq(\mathbf{W}_m)$ within $\trq(m)$ in the above formula for $\Zq$ with the class $\left<\mathbf{W}_m\right>$ of $\mathbf{W}_m$ in the skein space  (and possibly including signs in the sum, depending on the topology of the manifold $M$).    Note here that the coefficients of the terms $\left<\mathbf{W}_m\right>$ appearing in $\Zq$ remain rational in $q$.  (See \cite{cremaschi1} for a study of the $n$-skein space of the annulus, connecting to the quantum cluster geometry of Fock--Goncharov \cite{FG}.)
\end{enumerate}
\end{rema}

\subsubsection{Proof of Theorem \ref{theo:polynomial}}\label{sssec:proofofLaurentproperty}

First, the construction of $\trq(\mathbf{W})$ as a Reshetikhin--Turaev invariant is explained.  Actually, a slight generalization, valid for `stated webs' $(\mathbf{W},s)$, is required for the following argument.

A \emph{slit collection} $H\subset\R^2$ is a (possibly empty) finite collection of disjoint horizontal segments, thought of as `cuts' in the plane.  A \emph{web (with boundary) $\mathbf{W}$ relative to the slit collection $H$} is defined exactly as before, except $\mathbf{W}$ is allowed to have monovalent vertices (possibly none, and with distinct $xy$-coordinates) ending on $H\times\R$, where the blackboard framing is required.  The monovalent vertices are `pinned' in the sense that they are not allowed to move during isotopies.  Web diagrams $D$ drawn in $(\R^2,H)$ are defined as before.  

A \emph{strip $R\subset\R^2$ relative to $H$} is an infinite horizontal strip in $\R^2$ of the form $R=\R\times[y_1,y_2]$ with $y_1<y_2$ whose interior contains no slits of $H$, but whose boundary is allowed to contain slits (possibly none).  Given a web diagram $D$, a \emph{strip system $\{R_i\}_{i=1,2,\dots,t}$ relative to $H$ (and $D$)} consists of finitely many strips $R_i=\R\times[y_i,y_{i+1}]$ such that $D\subset\cup_{i=1}^t R_i$.  It is always assumed that diagrams $D$ are in generic position with respect to the strip system (that is, they intersect the boundaries of each $R_i$ transversely).

A component of the restriction $D|_{R_i}$ of the diagram to the $i$-th strip $R_i$ is a \emph{building block} if it is one of the local diagrams displayed in Figure \ref{fig:xmastreechandelier}.  The diagram $D$ is in \emph{good position relative to the strip system $\{R_i\}_{i=1,2,\dots,t}$} if for each $i$ the components of the restriction $D|_{R_i}$ are either (1) building blocks, at most one of which is not a vertical strand (the first two diagrams in Figure \ref{fig:xmastreechandelier}); or, (2) isolated points lying on the strip boundary intersect $H$ (these are, in particular, monovalent vertices from boundary-ending strands in adjacent strips).  It is always assumed that diagrams $D$ are in good position with respect to the strip system.  

A point in $D\cap H$ is a \emph{slit point}.  Every slit point is in some $\R\times\{y_i\}$.  If a slit point is an isolated point of $D|_{R_i}$ it is \emph{excludable relative to $R_i$}, else it is \emph{includable relative to $R_i$}.   Note a slit point is always includable in one of its two adjacent strips (in which case it is part of a building block) and excludable in the other.  Points of $D\cap\partial R_i$ that are not slit points are also considered \emph{includable relative to $R_i$}.  

Let $V=\mathbb{C}^n$ with standard basis $\{x_j\}$ and $V^*=\{V\to\mathbb{C}\}$ its linear dual space with standard basis $\{x^*_j\}$.  The intersection points of the diagram $D$ with the strip boundary components $\R\times\{y_i\}$ consist of finitely many strands oriented either down or up (with respect to the second coordinate).  For each $i=1,2,\dots,t$, to every downward directed includable (relative to $R_i$) intersection point in $\R\times\{y_i\}$ or $\R\times\{y_{i+1}\}$ attach a copy of $V$, and to every upward directed includable point attach $V^*$. If $V_{i,1}, V_{i,2}, \dots, V_{i,r_i}$ are the spaces so-attached to the strip boundary component $\R\times\{y_i\}$ from left to right, and $V^\prime_{i,1}, V^\prime_{i,2}, \dots, V^\prime_{i,r^\prime_i}$ the spaces for $\R\times\{y_{i+1}\}$, put $V_i:=V_{i,1}\otimes V_{i,2}\otimes\dots\otimes V_{i,r_i}$ and $V^\prime_i:=V^\prime_{i,1}\otimes V^\prime_{i,2}\otimes\dots\otimes V^\prime_{i,r^\prime_i}$.  (If $D$ does not intersect a strip boundary component, the space $\mathbb{C}$ is attached.)

Each diagram restriction $D|_{R_i}$ defines a linear map $\varphi_i:V_i\to V^\prime_{i}$ as follows.  The linear maps for building blocks are displayed in Figure \ref{fig:xmastreechandelier}.  Here $I_V$ and $I_{V^*}$ are identity maps, 
$$
R(x_i\otimes x_j) = q^{-\frac{1}{n}}
\begin{cases}
x_j \otimes x_i, & i>j \\
q  x_i \otimes x_j, & i=j \\
x_j \otimes x_i +(q-q^{-1})x_i\otimes x_j, & i<j
\end{cases},
$$
$$
T_+(1) = \sum_{\sigma \in \mathfrak{S}_n} (-q)^{\ell(\sigma)}x_{\sigma(1)}\otimes x_{\sigma(2)}\otimes \dots \otimes x_{\sigma(n)},
$$
$$
T_-(x_{i_1}\otimes x_{i_2}\otimes\dots \otimes x_{i_n}) = 
\begin{cases}
(-q)^{\ell(\sigma)}, & (1,2,\ldots,n) \mapsto (i_1,i_2,\ldots,i_n) \text{ defines a permutation}\\
0, & \text{otherwise}
\end{cases}.
$$ 
If $D|_{R_i}$ is not a building block, then it is the disjoint union of a building block with some oriented vertical strands (plus isolated points that are being systematically ignored), and the linear map $\varphi_i$ is defined in the obvious way by tensoring together the linear map for the building block with the identity maps for the vertical strands.  

A \emph{stated web $(\mathbf{W},s)$ relative to $H$} is a web together with the assignment of numbers $s(p)\in\{1,2,\dots,n\}$, called \emph{states}, to the monovalent vertices $p$ of $\mathbf{W}$ lying on $H\times\R$.  Here $s:\partial \mathbf{W}\to\{1,2,\dots,n\}$ is called the \emph{state function}.  Notions of isotopy and diagrams of stated webs are the same as for webs.  In particular, $(D,s)$ is the associated stated diagram.  

For $i=1,2,\dots,t$, a \emph{state function $s_i$ for the restriction $D|_{R_i}$} assigns to each includable boundary point $p$ a state $s_i(p)\in\{1,2,\dots,n\}$.  The pair $(D|_{R_i},s_i)$ is called a \emph{stated diagram restriction}.  Also write $s_i=(s_i^H,s_i^{H^c})$ where $s_i^H$ is the restriction of the state function $s_i$ to the includable slit points, and $s_i^{H^c}$ is the restriction to the nonslit includable points.  

Given a stated restriction $(D|_{R_i},s_i)$ consider the linear map $\varphi_i:V_i\to V^\prime_i$ for $D|_{R_i}$ defined above.  Let $v(s_i)\in V_i$ be the standard basis tensor whose $j$-th factor, corresponding to the $j$-th includable boundary point $p_j$ on $\R\times\{y_i\}$ measured from left to right, is the standard basis element $x_{s_i(p_j)}$ or $x^*_{s_i(p_j)}$ (depending on whether $V_{i,j}$ is $V$ or $V^*$).  And let $v^\prime(s_i)\in V^\prime_i$ be the tensor determined in an identical way for the boundary component $\R\times\{y_{i+1}\}$.  Define $\varphi_i(s_i)\in\mathbb{C}$ to be the coefficient of the basis element $v^\prime(s_i)$ in the image $\varphi_i(v(s_i))\in V^\prime_i$.  

Note there is a natural one-to-one correspondence between the nonslit includable points $p^\prime_{j^\prime}$ of $\R\times\{y_{i+1}\}$ relative to $D|_{R_i}$ and those $p_{j(j^\prime)}$ of $\R\times\{y_{i+1}\}$ relative to $D|_{R_{i+1}}$ for determined $j=j(j^\prime)$.  Call stated restrictions $(D|_{R_i},s_i)$ and $(D|_{R_{i+1}},s_{i+1})$ \emph{compatible} if $s_i^{H^c}(p^\prime_{j^\prime})=s_{i+1}^{H^c}(p_{j(j^\prime)})$ for all such $j^\prime$.  A collection of stated diagrams $\{(D|_{R_i},s_i)\}_{i=1,2,\dots,t}$ is called \emph{compatible} if all pairs of adjacent stated diagrams are compatible.  

The stated diagram $(D,s)$ for $(\mathbf{W},s)$ determines states $s_i^H$ for the includable slit points for the restrictions $D|_{R_i}$ for all $i=1,2,\dots,t$ (but does not determine states $s_i^{H^c}$).  Define 
$$
\trq(\mathbf{W},s):=\sum_{\text{compatible }s_1,s_2,\dots,s_t}\prod_{i=1}^t \varphi_i(s_i)\in\mathbb{C}
$$ 
summed over all compatible states $s_i=(s_i^H,s_i^{H^c})$ of the restrictions $D|_{R_i}$ such that the fixed states $s_i^H$ for includable slit points are those determined by the stated diagram $(D,s)$.  

When $H,s=\emptyset$, this is the definition of $\trq(\mathbf{W}):=\trq(\mathbf{W},\emptyset)$ given in \cite{sikora1}.  The isotopy invariance of $\trq(\mathbf{W})$ in the closed case (Theorem \ref{theo:qtraceofweb}) is equivalent to $\trq(\mathbf{W})$ satisfying the Reidemeister moves (Figure \ref{fig:framedReidemeistermoves}), which are local.  It immediately follows in the more general open setting that $\trq(\mathbf{W},s)$ is isotopy invariant with respect to $H$, as isotopies of $\trq(\mathbf{W},s)$ are not allowed to cross the slits of $H$.  (It is not hard to see that $\trq(\mathbf{W},s)$, so defined, is not in general invariant upon isotoping across slits.)  Equivalently, this isotopy invariance means that $\trq(\mathbf{W},s)$ is independent of the choice of good position relative to a strip system.  

Moving ahead, let $h$ be a slit not in $H$.  There is an operation that takes a web $\mathbf{W}$ relative $H$ and `cuts' it, by introducing the slit $h$, to give a new web $\mathbf{W}^h$ relative to $H\cup h$.  (To avoid overlapping monovalent vertices, the part of $\mathbf{W}^h$ just below (resp. above) the slit $h$ can be perturbed slightly to the left (resp. right), say.)  The resulting web $\mathbf{W}^h$ is then considered up to isotopy with respect to $H\cup h$.  Note that this operation is only defined on webs $\mathbf{W}$, not on their isotopy classes relative to $H$ (for such an isotopy could inadvertently cross $h$).    

Note also this cutting operation is not defined on stated webs $(\mathbf{W},s)$.  However, if additional states $s^h=(s_l^h, s_u^h)$ are chosen, assigned to the lower and upper boundary points of $\mathbf{W}^h$ created after cutting $\mathbf{W}$ along the intersection $\mathbf{W}\cap (h\times\R)$, then one obtains a stated web $(\mathbf{W}^h,s\cup s^h)$.  Call the stating $s^h$ of $(\mathbf{W}^h,s\cup s^h)$ \emph{compatible} if the states $s_l^h$ and $s_u^h$ agree on pairs of corresponding boundary points coming from cutting along $h$. Then, essentially by definition (say, by refining the strip system to include the new slit $h$), 
$$
\trq(\mathbf{W},s)=\sum_{\text{compatible }s^h}\trq(\mathbf{W}^h,s\cup s^h).
$$  

As the last preparatory item, let $h$ lie along $\R\times\{y_i\}$ relative to a strip system $\{R_i\}_i$ for $H\cup h$.  Assume that the building block of $D^h|_{R_i}$ (resp. $D^h|_{R_{i-1}}$) is the source black (resp. sink white) vertex from Figure \ref{fig:xmastreechandelier} intersecting $h$ in $m_h$ (adjacent) endpoints, constituting all the upper (resp. lower) boundary points of $D^h$ along $h$.  In this case, say the strip system $\{R_i\}_i$ is \emph{well-adapted to $\mathbf{W}^h$}.  Then, from the definitions of the linear maps $T_-$ and $T_+$,  
\begin{gather*}
\sum_{\text{compatible }s^h}\trq(\mathbf{W}^h,s\cup s^h)
=\left( \sum_{\sigma\in\mathfrak{S}_{m_h}}q^{2\ell(\sigma)} \right)\sum_{I(m_h)}\trq(\mathbf{W}^h,s\cup s^h_{I(m_h)})
\\=q^{\binom{m_h}{2}}[m_h]!\sum_{I(m_h)}\trq(\mathbf{W}^h, s\cup s^h_{I(m_h)})
\end{gather*} 
where in the second and third expressions the sum is over all subsets $I(m_h)=\{i_1,i_2,\dots,i_{m_h}\}$ of $\{1,2,\dots,n\}$ of size $m_h$ with $i_1<i_2<\dots<i_{m_h}$, and where $s^h_{I(m_h)}=(s^h_{lI(m_h)},s^h_{uI(m_h)})$ assigns the states $i_1,i_2,\dots,i_{m_h}$ from left to right on both the lower and upper boundary points along $h$.  (For the last equality above, see, e.g., \cite{sikora1}.)

Now, to begin the proof proper, let $\mathbf{W}_m$ be the split web corresponding to the multiweb $m\in\Omega_n$.  If $m$ is proper, then $\trq(m)=\trq(\mathbf{W}_m)$ is Laurent by Theorem \ref{theo:qtraceofweb}.  Else, let $e_1, e_2, \dots, e_r$ be the edges of $m$ of nontrivial multiplicity $m_{e_j}>1$.  (Note that since the conclusion of the theorem is independent of the sign of the trace $\trq(m)$, cilia considerations are irrelevant, by Proposition \ref{prop:ciliadependence}, so will be systematically ignored throughout the proof.)

For $j=0,1,2,\dots,r$ and subsets $I(m_{e_1}), I(m_{e_2}), \dots, I(m_{e_j})$ as above, inductively define the slit $h_j$, slit collection $H_j=\cup_{j'=0}^j h_{j'}$, and stated web $(\mathbf{W}_j,s^j_{I(m_{e_1}), I(m_{e_2}), \dots, I(m_{e_j})})$ with respect to $H_j$ as follows.  For the base case, put $\mathbf{W}_0=\mathbf{W}_m$ and $s^0, h_0=\emptyset$.  Then, it is not hard to see that one can isotope $(\mathbf{W}_{j-1},s^{j-1}_{I(m_{e_1}), I(m_{e_2}), \dots, I(m_{e_{j-1}})})$ relative to $H_{j-1}$, where the isotopy is supported around the edge $e_j$ by `shrinking' $e_j$ into a small neighborhood, such that for the stated web $(\mathbf{W}^\prime_{j-1},s^{j-1}_{I(m_{e_1}), I(m_{e_2}), \dots, I(m_{e_{j-1}})})$ relative to $H_{j-1}$ resulting from the isotopy:  (1)  $e_j$ does not cross any other part of the web $\mathbf{W}^\prime_{j-1}$;  (2)  $e_j$ is vertical with the black (resp. white) vertex at the top (resp. bottom);  and, (3)  for a small slit $h_j$ inserted midway across $e_j$, putting $H_j=H_{j-1}\cup h_j$ and $\mathbf{W}_j=\mathbf{W}_{j-1}^{\prime h_j}$, there exists a strip system $\{R^j_i\}_{i=1,2,\dots,t_j}$ for $H_j$ and well-adapted to $\mathbf{W}_j$ in the sense above (with the black and white vertices of $e_j$ in the two strips adjacent to $h_j$).  Define the state function $s^j_{I(m_{e_1}), I(m_{e_2}), \dots, I(m_{e_j})}=s^{j-1}_{I(m_{e_1}), I(m_{e_2}), \dots, I(m_{e_{j-1}})}\cup s^{h_j}_{I(m_{e_j})}$ where $s^{h_j}_{I(m_{e_j})}$ is defined as above.  

Iterating the above calculation, 
\begin{gather*}
\trq(\mathbf{W}_0)
=\trq(\mathbf{W}^\prime_0,s^0)
=q^{\binom{m_{e_1}}{2}}[m_{e_1}]!\sum_{I(m_{e_1})}\trq(\mathbf{W}_1,s^1_{I(m_{e_1})})
=q^{\binom{m_{e_1}}{2}}[m_{e_1}]!\sum_{I(m_{e_1})}\trq(\mathbf{W}^\prime_1,s^1_{I(m_{e_1})})
\\=q^{\binom{m_{e_1}}{2}}[m_{e_1}]!q^{\binom{m_{e_2}}{2}}[m_{e_2}]!\sum_{I(m_{e_1}), I(m_{e_2})}\trq(\mathbf{W}_2,s^2_{I(m_{e_1}), I(m_{e_2})})
=\dots
\\=
\left(\prod_{i=1}^r q^{\binom{m_{e_i}}{2}}[m_{e_i}]!\right)\sum_{I(m_{e_1}), I(m_{e_2}), \dots, I(m_{e_r})}\trq(\mathbf{W}_r,s^r_{I(m_{e_1}), I(m_{e_2}), \dots, I(m_{e_r})})
\end{gather*}
where the first, third, etc., equalities are by the isotopy invariance of the quantum trace for stated webs.  In particular, $\trq(\mathbf{W}_m)=\trq(\mathbf{W}_0)$ is divisible by $\prod_{i=1}^r [m_{e_i}]!$ (even by $\prod_{i=1}^r q^{\binom{m_{e_i}}{2}}[m_{e_i}]!$, compare Remark \ref{rema:afterdefof3dquantumtrace} (\ref{item:rema:afterdefof3dquantumtrace})).  This completes the proof.  

\begin{figure}[t]
\centering
\begin{gather*}
    \vcenter{\hbox{
        \begin{tikzpicture}
            \draw[line width=1pt, double distance=3pt,
             arrows = {Stealth[length=0pt 2.5 0,fill=white]}-] (0,0) -- (0,1.5);
        \end{tikzpicture}
    }}
    =
    I_V, 
    \hspace{5mm}
    \vcenter{\hbox{
        \begin{tikzpicture}
            \draw[line width=1pt, double distance=3pt,
             arrows = {Stealth[length=0pt 2.5 0,fill=white]}-] (0,1.5) -- (0,0);
        \end{tikzpicture}
    }}
    =
    I_{V^*},
    \hspace{5mm}
    \vcenter{\hbox{
        \begin{tikzpicture}[scale=.9]
            \draw[line width=1pt, double distance=3pt,
             arrows = {Stealth[length=0pt 2.5 0,fill=white]}-] ({sqrt(2)/2},{-sqrt(2)/2}) -- ({-sqrt(2)/2},{sqrt(2)/2});
            \draw[white, line width=12] ({-sqrt(2)/2},{-sqrt(2)/2}) -- ({sqrt(2)/2},{sqrt(2)/2});
            \draw[line width=1pt, double distance=3pt,
             arrows = {Stealth[length=0pt 2.5 0,fill=white]}-] ({-sqrt(2)/2},{-sqrt(2)/2}) -- ({sqrt(2)/2},{sqrt(2)/2});
        \end{tikzpicture}
    }}
    =
    R,
    \vcenter{\hbox{
        \begin{tikzpicture}[scale=.9]
            \draw[line width=1pt, double distance=3pt,
             arrows = {Stealth[length=0pt 2.5 0,fill=white]}-] ({-sqrt(2)/2},{-sqrt(2)/2}) -- ({sqrt(2)/2},{sqrt(2)/2});
            \draw[white, line width=12] ({sqrt(2)/2},{-sqrt(2)/2}) -- ({-sqrt(2)/2},{sqrt(2)/2});
            \draw [line width=1pt, double distance=3pt,
             arrows = {Stealth[length=0pt 2.5 0,fill=white]}-] ({sqrt(2)/2},{-sqrt(2)/2}) -- ({-sqrt(2)/2},{sqrt(2)/2});
        \end{tikzpicture}
    }}
    =
    R^{-1}, 
    \\
    \vcenter{\hbox{
        \begin{tikzpicture}[]
            \draw[line width=1pt, double distance=3pt,
             arrows = -{Stealth[length=0pt 2.5 0,fill=white]}] (0,0) ..controls (.5,-1.25) and (1,-1.25) .. (1.5,0);
        \end{tikzpicture}
    }}
    =
    1 \mapsto \sum_{i=1}^n x_i \otimes x_i^*,
    \hspace{5mm}
    \vcenter{\hbox{
        \begin{tikzpicture}[]
            \draw[line width=1pt, double distance=3pt,
             arrows = {Stealth[length=0pt 2.5 0,fill=white]}-] (0,0) ..controls (.5,-1.25) and (1,-1.25) .. (1.5,0);
        \end{tikzpicture}
    }}
    =
    1 \mapsto \sum_{i=1}^n q^{n+1-2i}x_i^* \otimes x_i,   
    \\
    \vcenter{\hbox{
        \begin{tikzpicture}[]
            \draw[line width=1pt, double distance=3pt,
             arrows = -{Stealth[length=0pt 2.5 0,fill=white]}] (0,0) ..controls (.5,1.25) and (1,1.25) .. (1.5,0);
        \end{tikzpicture}
    }}
    =
    x_i^*\otimes x_j \mapsto \delta_{ij},
    \hspace{5mm}
    \vcenter{\hbox{
        \begin{tikzpicture}[]
            \draw[line width=1pt, double distance=3pt,
             arrows = {Stealth[length=0pt 2.5 0,fill=white]}-] (0,0) ..controls (.5,1.25) and (1,1.25) .. (1.5,0);
        \end{tikzpicture}
    }}
    =
    x_i\otimes x_j^* \mapsto q^{2i-n-1}\delta_{ij},
    \\
    \vcenter{\hbox{
        \begin{tikzpicture}[]
\draw[red, line width=1pt, decoration = {zigzag,segment length = 3mm, amplitude = 1mm},decorate] (0,0)--(0,.85);
            \draw[line width=1pt, double distance=3pt,
             arrows = -{Stealth[length=0pt 2.5 0,fill=white]}] (0,0) -- (-1,-.5);
            \draw[line width=1pt, double distance=3pt,
             arrows = -{Stealth[length=0pt 2.5 0,fill=white]}] (0,0) -- (-.5,-1);
            \draw[line width=1pt, double distance=3pt,
             arrows = -{Stealth[length=0pt 2.5 0,fill=white]}] (0,0) -- (1,-.5);
            \node[] at (.25,-.5) {...};
            \node[draw,circle,fill=black,scale=2] (a) at (0,0) {};  
        \end{tikzpicture}
    }}
    =
    T_-,
    \hspace{5mm}
    \vcenter{\hbox{
        \begin{tikzpicture}[] 
\draw[red, line width=1pt, decoration = {zigzag,segment length = 3mm, amplitude = 1mm},decorate] (0,0)--(0,-.85);
            \draw[line width=1pt, double distance=3pt,
            arrows = {Stealth[length=0pt 2.5 0,fill=white,reversed]}-]
              (-1,.5) -- (0,0);
            \draw[line width=1pt, double distance=3pt,
             arrows = {Stealth[length=0pt 2.5 0,fill=white,reversed]}-] (-.5,1)-- (a);
            \draw[line width=1pt, double distance=3pt,
             arrows = {Stealth[length=0pt 2.5 0,fill=white,reversed]}-] (1,.5) -- (a);
            \node[] at (.25,.5) {...};
            \node[line width = 1pt,draw,black,circle,fill=white,scale=2] (a) at (0,0) {}; 
        \end{tikzpicture}
    }}
    =
    T_+
\end{gather*}
\caption{Building blocks for the Reshetikhin--Turaev invariant.}
\label{fig:xmastreechandelier}
\end{figure}
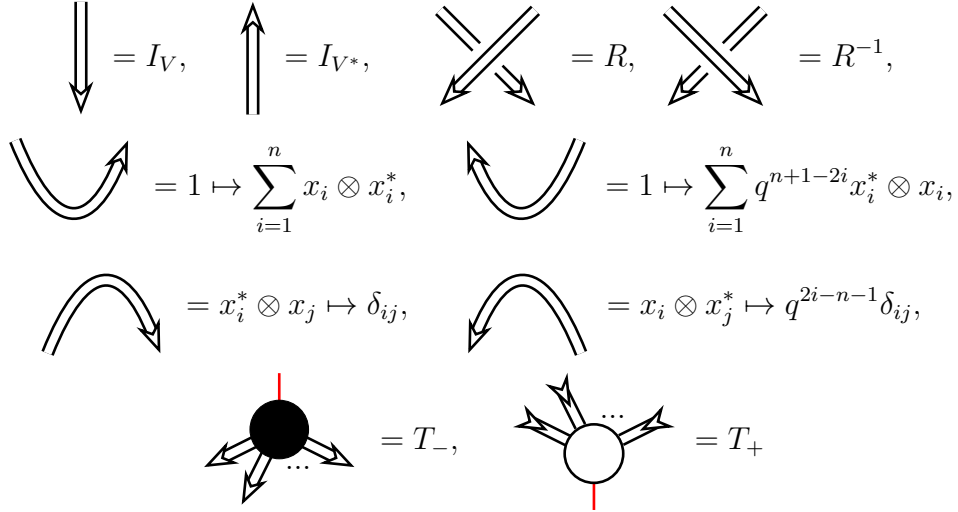

\section{Planar setting}\label{sec:planarsetting}

\subsection{Planar graphs}

\begin{defi}
A diagram in $\R^2$ of a ribbon graph $\mathbf{G}$ is \emph{planar} if it has no crossings.  To say $\mathbf{G}$ is planar means, possibly after isotopy, it admits a planar diagram.  From such a planar diagram is obtained a planar embedding of the core graph $G$, and in this way the diagram is identified with $G$.  In particular, $G$ can be thought of as a ciliated planar graph embedded in $\R^2$, where the cilia orientation conventions are the same as those described for diagrams in Definition \ref{defi:diagram}. 

A \emph{face} of $G$ means an internal face.  The \emph{length} of a face, namely the number of edges, is denoted $l$.  
\end{defi}

\begin{rema}\label{rema:sikoratraceofmultiweb}
For the remainder of the paper the main objects of study will be ciliated planar graphs $G$ planarly embedded in $\R^2$.     Note that, by equipping   $G\subset\R^2$ with the blackboard framing and thinking of it as a ribbon graph   $\mathbf{G}\subset\R^2\times\{0\}\subset\R^3$, it makes sense to talk about the quantum trace $\trq(m)$ of a multiweb $m\in\Omega_n$ in $G$.  (This quantum trace is independent of ambient planar isotopy of $G$.)
\end{rema}

\subsection{Quantum connections and local definition of trace}

\begin{defi}\label{defi:quantummatrix}
A \emph{$2\times2$ quantum matrix} is a $2\times2$ matrix $M_q=\begin{pmatrix}a&b\\c&d\end{pmatrix}$ in formal variables $a,b,c,d$ subject to the relations 
$$
ab=qba, \quad cd=qdc, \quad ac=qca, \quad bd=qdb, \quad ad-da=(q-q^{-1})bc, \quad bc=cb.
$$  
We also require that $q$ commutes with all other variables. 

What is meant by an $n\times n$ quantum matrix is an $n\times n$ matrix $M_q=(M_{ij})$ in formal variables $M_{ij}$ such that every $2\times2$ submatrix $\begin{pmatrix}M_{ij}&M_{il}\\M_{kj}&M_{kl}\end{pmatrix}$ for $i<k$ and $j<l$ is a $2\times2$ quantum matrix.  In the same way, it makes sense to talk about a quantum matrix with coefficients in a specified (complex) algebra, so long as the matrix elements $M_{ij}$ satisfy the desired $q$-commutation relations.  
\end{defi}

\begin{exam}\label{exam:quantummatrixexamples}
\begin{enumerate}
\item
For $1\leq m\leq n$, every $m\times m$ submatrix of a quantum matrix is also a quantum matrix.  
\item\label{item2:exam:quantummatrixexamples}
When $q=1$, any matrix over a commutative algebra is a quantum matrix.  
\item\label{item:exam:quantummatrixexamples}
Any diagonal matrix whose entries mutually commute is a quantum matrix.  
\end{enumerate}
\end{exam}

\begin{defi}\label{defi:planarquantumtrace}
An \emph{edge-commuting $n\times n$ quantum connection} (or just \emph{quantum connection}) $\Phi_q$ is the assignment to each edge $e\in E$ of $G$ a quantum matrix $\Phi_q(e)=(\Phi_q(e)_{ij})$ satisfying the additional property that $\Phi_q(e)_{ij}$ commutes with $\Phi_q(e^\prime)_{i^\prime j^\prime}$ for all $e\neq e^\prime$ and $i,j,i^\prime,j^\prime$.  

Let $V=\mathbb{C}^n$ with standard basis $\{x_i\}$ and $V^*=\{V\to\mathbb{C}\}$ its linear dual space with standard basis  $\{x^*_i\}$, and let $\mathbb{A}=\mathbb{A}(\Phi_q)$ denote the algebra generated by the variables $\Phi_q(e)_{ij}$ varying over all $e,i,j$ (subject to the $q$-commutation and commutation relations specified just above).  The \emph{quantum codeterminant} is the element $\mathrm{codet}_q\in V^{\otimes n}$   defined by 
$$
\mathrm{codet}_q=\sum_{\sigma\in \mathfrak{S}_n}(-q)^{\ell(\sigma)}x_{\sigma(1)}\otimes x_{\sigma(2)}\otimes\dots\otimes x_{\sigma(n)}.
$$  
The \emph{quantum dual codeterminant} is the element $\mathrm{codet}_q^*\in(V^*)^{\otimes n}$ defined by 
$$
\mathrm{codet}^*_q=\sum_{\sigma\in \mathfrak{S}_n}(-q)^{\ell(\sigma)}x^*_{\sigma(1)}\otimes x^*_{\sigma(2)}\otimes\dots\otimes x^*_{\sigma(n)}.
$$

Let $m\in\Omega_n(G)$ be a proper multiweb.  At a black vertex $b$ attach a copy of $V$ to every half-edge from $b$ on the edges $e$ of $m$, that is, the edges with multiplicity $m_e=1$. Likewise at a white vertex $w$ attach a copy of $V^*$ to each half-edge of $m$.  At each vertex $v$ of $G$, the linear order of half-edges of $G$ incident to $v$ coming from the cilia at $v$ induces a linear order of the half-edges of $m$ at $v$.  Attach a copy of the quantum codeterminant $\mathrm{codet}_q\in V^{\otimes n}$ (resp. quantum dual codeterminant $\mathrm{codet}_q^*\in(V^*)^{\otimes n}$) to each black vertex $b$ (resp. white vertex $w$).  Here, in either vertex case one imagines the $i$-th factor in the tensor product corresponding to the $i$-th half-edge of $m$ in the linear order at that vertex.  The quantum matrix $\Phi_q(e)$ on the edge $e$ of $m$ \emph{(co)acts on the left} on a vector $v\in V$ located at the half-edge of $e$ incident to the black vertex $b$ to give the element $\Phi_q(e)(v)\in\mathbb{A}\otimes V$ defined by
$$
\Phi_q(e)(v)=\sum_{i,j=1}^n \alpha_j \Phi_q(e)_{ij} \otimes x_i \quad \left(v=\sum_{j=1}^n \alpha_j x_j\right).
$$  
Note also in general that an element $v^*\in V^*$ pairs with an element $\sum_i a_i\otimes v_i\in\mathbb{A}\otimes V$ to give the element $\sum_i v^*(v_i) a_i\in\mathbb{A}$.  Now, taking the quantum codeterminants at all the black vertices $b$ yields a tensor $\mathrm{codet}_q^{\otimes N}\in(V^{\otimes n})^{\otimes N}$.  Letting the quantum matrices $\Phi_q(e)$ act on this tensor for every edge $e$ of $m$ yields a tensor $(\bigotimes_e\Phi_q(e))(\mathrm{codet}_q^{\otimes N})\in((\mathbb{A}\otimes V)^{\otimes n})^{\otimes N}$ where the ordering of each inner tensor factor $(\mathbb{A}\otimes V)^{\otimes n}$ still comes from the cilia at black vertices.  Reorder the entire tensor according to the cilia at the white vertices $w$ yielding a tensor in $(\mathbb{A}^{\otimes n}\otimes V^{\otimes n})^{\otimes N}$ which by slight abuse of notation is also denoted $(\bigotimes_e\Phi_q(e))(\mathrm{codet}_q^{\otimes N})$.  Note in particular that both the $\mathbb{A}$ and the $V$ factors at each white vertex $w$ are ordered according to the cilia at $w$.  Pairing this tensor with the quantum dual codeterminants varying over all the white vertices yields an element $\mathrm{codet}_q^{*\otimes N}((\bigotimes_e\Phi_q(e))(\mathrm{codet}_q^{\otimes N}))\in(\mathbb{A}^{\otimes n})^{\otimes N}$.  This is an element of the form $\sum_i \alpha_i \bigotimes_{j=1}^N(a_{ij1}\otimes a_{ij2}\otimes\dots\otimes a_{ijn})$.  Multiplying together in $\mathbb{A}$ the factors of each term in this sum yields an element in $\mathbb{A}$ of the form $\sum_i \alpha_i \prod_{j=1}^N(a_{ij1}a_{ij2}\dots a_{ijn})$, called the \emph{quantum trace of the proper multiweb $m$ with respect to $\Phi_q$} and denoted $\trqc(\Phi_q,m)\in\mathbb{A}$.  Note here that the order of multiplying the $N$ different $n$-length factors $a_{ij1}a_{ij2}\dots a_{ijn}$ in the product over $j$ is immaterial, as the different $n$-length factors, corresponding to different white vertices, commute by definition of the edge-commuting quantum connection $\Phi_q$ (there is a natural surjection from the set of edges to the set of white vertices).  Actually, for $m$ proper, as currently, the order of the $n$ factors in each product $a_{ij1}a_{ij2}\dots a_{ijn}$ is also immaterial for the same reason, as there is a one-to-one correspondence between variables and edges of $m$ at each vertex. This order will matter when $m$ is not proper, discussed now.  

Let $m\in\Omega_n$ be any multiweb.  Define an auxiliary ciliated planar graph $G_m$, whose planar embedding is uniquely determined up to arbitrarily small isotopy, such that the vertices of $G_m$ are the vertices of $G$ and such that there are $m_e$ edges of $G_m$ for every edge $e$ of $G$ obtained by splitting $e$ into $m_e$ parallel copies.  The ciliation of $G$ determines a canonical ciliation of $G_m$, in particular, the cilia of $G_m$ never go in between the copied edges when $m_e>1$.  The edge-commuting quantum connection $\Phi_q$ on $G$ determines a quantum connection $\Phi_q(G_m)$ on $G_m$, not necessarily edge-commuting, by putting $\Phi_q(G_m)(e^\prime)=\Phi_q(e)$ on each of the $m_e$ parallel edges $e^\prime$ of $G_m$ coming from splitting the edge $e$ of $G$.  In particular, the variables on the $m_e$ parallel edges do not in general commute.  Note $G_m$ can be thought of as a proper multiweb in itself.  As discussed in the previous paragraph, the quantum trace $\trqc(\Phi_q(G_m),G_m)\in\mathbb{A}$ thus makes sense despite the fact that $\Phi_q(G_m)$ is not edge-commuting, as the cilia around white vertices control the ordering of the noncommuting variables on parallel split edges.  Finally, analogous to Definition \ref{defi:firstquantumtraceofmultiweb}, define the \emph{quantum trace of the multiweb $m$ with respect to $\Phi_q$}, also denoted $\trqc(\Phi_q,m)\in\mathbb{A}$, by
$$
\trqc(\Phi_q,m)=\frac{\trqc(\Phi_q(G_m),G_m)}{\prod_{e\in E}[m_e]!}.
$$
\end{defi}

\begin{rema}\label{rema:definitionofquantumtraceofmultiwebs}
\begin{enumerate}
\item\label{item2:rema:definitionofquantumtraceofmultiwebs}
Of course, the definition of $\trqc(\Phi_q,m)$ depends on the ciliation $L$ of $G$.  This dependence will be suppressed in the notation.  
\item\label{item:rema:definitionofquantumtraceofmultiwebs}
Just as for Definition \ref{defi:firstquantumtraceofmultiweb}, a priori the quantum trace $\trqc(\Phi_q,m)$ of a general multiweb $m\in\Omega_n$ is defined by a rational expression in $q$ (over $\mathbb{A}$).  It is shown later  (Remark \ref{rema:qtraceispolyexpressioninqconnection}) 
(\ref{item:qtraceispolyexpressioninqconnection}) 
that it is in fact a polynomial expression in $q$ (over $\mathbb{A}$).
\item  
The commuting condition for edge-commuting quantum connections $\Phi_q$ is quite restrictive.  Other cases are relevant, especially in settings that are topologically nontrivial, see e.g. \cite{bc21, ChekhovShapiro, douglas1, GoncharovKenyon, pp24}.
\end{enumerate}
\end{rema}

\begin{defi}
For any planar graph $G$ equipped with a quantum connection $\Phi_q$, the \emph{quantum partition function with respect to $\Phi_q$}, denoted $Z(\Phi_q)\in\mathbb{A}$, is defined by
$$
Z(\Phi_q)=q^{-N\binom{n}{2}}\sum_{m\in\Omega_n} \trqc(\Phi_q,m).
$$
\end{defi}

\begin{rema}
Note $Z(\Phi_q)$ is independent of planar isotopy of $G$. 
\end{rema}

\begin{exam}
Consider the $3$-multiweb of Figure \ref{small3web}, where $I$ denotes the identity matrix. Its quantum trace is computed as follows. The codeterminant at $b$ is  
$$
\mathrm{codet}_q = x_1\otimes x_2\otimes x_3-qx_1\otimes x_3\otimes x_2-qx_2\otimes x_1\otimes x_3+q^2x_2\otimes x_3\otimes x_1+ q^2x_3\otimes x_1\otimes x_2-q^3x_3\otimes x_2\otimes x_1
$$ 
where the order of tensor factors corresponds to the order of edges counterclockwise around the black vertex (taking into account the multiplicity). The dual codeterminant at $w$ has a similar formula. When computing the pairing along the lower edge (of multiplicity $2$) the indices in the first and second position must match. This implies that those in the third position also match. Thus each term in the codeterminant at $b$ is paired with the corresponding term in the dual codeterminant at $w$, giving $6$ terms in all:
$$
\trqc(\Phi_q,m)=\frac{\phi_{33}+q^2\phi_{22}+q^2\phi_{33}+q^4\phi_{11}+q^4\phi_{22}+q^6\phi_{11}}{[2]}=q^5\phi_{11}+q^3\phi_{22}+q\phi_{33}.
$$
If, instead, $\phi$ and $I$ are swapped, then a similar (while slightly more involved, using the quantum matrix relations) calculation gives the quantum trace to be 
$$
\trqc(\Phi_q,m)=q^5(\phi_{22}\phi_{33}-q\phi_{23}\phi_{32})+q^3(\phi_{11}\phi_{33}-q\phi_{13}\phi_{31})+q(\phi_{11}\phi_{22}-q\phi_{12}\phi_{21}).
$$

\begin{figure}[t]
\centering
\includegraphics[width=1.5in]{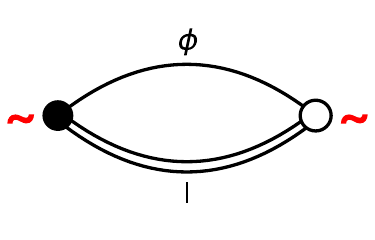}
\caption{\label{small3web}A ciliated $3$-multiweb with a quantum connection.}
\end{figure}
\end{exam}

\subsection{Alternative definition of quantum trace}\label{alt}

\begin{defi}\label{defi:qdeterminant}
The \emph{quantum determinant} $\det_q(M_q)$ of a quantum matrix $M_q=(M_{ij})$ is 
$$
\det_q(M_q)=\sum_{\sigma\in \mathfrak{S}_n}(-q)^{\ell(\sigma)}M_{1\sigma(1)}M_{2\sigma(2)}\dots M_{n\sigma(n)}.
$$
See \cite{parshall_wang} for information on quantum determinants.  

For subsets $S,T\subset\{1,2,\dots,n\}$ of size $1\leq m\leq n$ the \emph{$(S,T)$-quantum minor} $\det_{q,S,T}(M_q)$ is the quantum determinant of the $m\times m$ quantum submatrix $(M_{ij})_{i\in S,j\in T}$.  

A \emph{half-edge $n$-coloring} $c$ of a multiweb $m\in\Omega_n$ is the assignment of subsets $S_e, T_e\subset\{1,2,\dots,n\}$ of size $1\leq m_e\leq n$ to the edges $e$ of $m$, where $S_e$ (resp. $T_e$) is imagined to be attached to the half-edge of $e$ incident to the white (resp. black) vertex, and satisfying the property that for every vertex the union of the subsets attached to the half-edges incident to that vertex equals $\{1,2,\dots,n\}$.  In particular, the subsets around a given vertex are disjoint.  

If $c$ is a half-edge coloring of $m$ and $v$ is a vertex, then the associated \emph{vertex permutation} $\sigma_v\in \mathfrak{S}_n$ is defined as follows.  According to the linear order of half-edges at $v$, list the colors of the subsets of $\{1,2,\dots,n\}$ assigned by $c$ to these half-edges, where for a given subset the colors are listed in their natural order.  This determines the desired permutation $\sigma_v$ of $\{1,2,\dots,n\}$.   

Let $\Phi_q$ be a quantum connection on $G$ and let $m\in\Omega_n$ be a multiweb.  Note if $c$ is a half-edge coloring of $m$, then the quantum minor $\det_{q,S_e,T_e}(\Phi_q(e))$ of size $m_e$ is defined for all edges $e$ of $m$.  The \emph{alternative quantum trace} of $m$, denoted $\trqc^\prime(\Phi_q,m)\in\mathbb{A}$, is defined to be
\be\label{alttrace}
\trqc^\prime(\Phi_q,m)=\left(\prod_{e\in E}q^{\binom{m_e}{2}}\right)\sum_c\prod_{v\in V}(-q)^{\ell(\sigma_v)}\prod_{e\in E}\det_{q,S_e,T_e}(\Phi_q(e))  
\ee
where the sum is over all half-edge colorings $c$ of the multiweb $m$.  Note the order of the factors in $\mathbb{A}$ in the rightmost product over edges $e$ of $m$ is immaterial, as the quantum connection $\Phi_q$ is edge-commuting.  
\end{defi}

\begin{rema}\label{rema:alternativeqtraceispolynomial}
The alternative quantum trace $\trqc^\prime(\Phi_q,m)$ of a general multiweb $m\in\Omega_n$ is defined by a polynomial expression in $q$ (over $\mathbb{A}$).  
\end{rema}

\begin{prop}\label{prop:comparisonofqtraces}
For all quantum connections $\Phi_q$ and multiwebs $m\in\Omega_n$,
$$
\trqc(\Phi_q,m)=\trqc^\prime(\Phi_q,m)\in\mathbb{A}.
$$
\end{prop}

\begin{lemma}[\cite{parshall_wang}] \label{lem:qdet_orders}
Let $M_q=(M_{ij})$ be a quantum matrix.
\begin{enumerate}
\item 
$\displaystyle \det_q(M_q) = \sum_{\tau \in \mathfrak{S}_n} (-q)^{\ell(\tau)-\ell(\sigma)} M_{\sigma(1)\tau(1)}M_{\sigma(2)\tau(2)} \dots M_{\sigma(n)\tau(n)}$ for any $\sigma \in \mathfrak{S}_n$.
\item 
$\displaystyle \det_q(M_q) = \sum_{\sigma \in \mathfrak{S}_n} (-q)^{\ell(\sigma)-\ell(\tau)} M_{\sigma(1)\tau(1)}M_{\sigma(2)\tau(2)} \dots M_{\sigma(n)\tau(n)}$ for any $\tau \in \mathfrak{S}_n$.
\end{enumerate}
\qed\end{lemma}

\begin{proof}[Proof of Proposition \ref{prop:comparisonofqtraces}] 
Throughout the proof, notation as in Definitions \ref{defi:planarquantumtrace}, \ref{defi:qdeterminant} is used.  

Let $G_m$ be the split multiweb, with induced quantum connection $\Phi_q=\Phi_q(G_m)$.  The terms in the sum (indexed by $\mathfrak{S}_n$) expressing the codeterminant $\mathrm{codet}_q$ at each vertex $b$ are in one-to-one correspondence with half-edge colorings near $b$.  Similarly at each vertex $w$.  Since half-edge colorings can be chosen independently at each vertex, there is a bijection between half-edge colorings $c$ of $G_m$ and terms in the tensor $\mathrm{codet}_q^{*\otimes N}\otimes\mathrm{codet}_q^{\otimes N}$. 

One gathers
$$
\mathrm{codet}_q^{*\otimes N}\otimes\mathrm{codet}_q^{\otimes N}
=\sum_c \left(\prod_v (-q)^{\ell(\sigma_v)}\right)
\bigotimes_w x^*_{\sigma_w(1)}\otimes\dots\otimes x^*_{\sigma_w(n)} \otimes\bigotimes_b x_{\sigma_b(1)}\otimes\dots\otimes x_{\sigma_b(n)}
$$
where $\sigma_v=\sigma_b$ or $\sigma_w$ in $\mathfrak{S}_n$ are the vertex permutations for $c$.  Note in the contraction 
$$
\mathrm{codet}_q^{*\otimes N}\left(\left(\bigotimes_e\Phi_q(e)\right)(\mathrm{codet}_q^{\otimes N})\right)
$$ 
if an edge $e$ assigns color $j$ to the black end $b$ and color $i$ to the white end $w$, then one will get a factor of $\Phi_q(e)_{ij}$.  Also note, since $\det_{q,\{i\},\{j\}}(\Phi_q(e))=\Phi_q(e)_{ij}$, the result follows immediately if $m$ is proper.  

So let $m$ be a general multiweb.  By the definition, the convention is used that one multiplies factors in $\mathbb{A}$ from left to right according to the (clockwise) cyclic order of the edges around white vertices.  

Partition the sum by equivalent colorings, calling two half-edge colorings $c$ of  $G_m$ equivalent if they correspond to the same half-edge coloring of $m$, that is, differ only by permuting colors on half-edges which come from the same half-edge in $m$. For a given edge $e$ from $b$ to $w$ of multiplicity $m_e$, fix the sets $S_e = \{i_1,\dots,i_{m_e}\}$ and $T_e=\{j_1,\dots,j_{m_e}\}$,  and (first fixing half-edge colorings on all edges other than $e$) let $c_0$ be the half-edge coloring of $G_m$ which uses the colors of $S_e$ and $T_e$ in increasing order on the $m_e$ edges.  Let $\sigma_{0}:=\sigma(c_0)_w$ and $\tau_{0}:=\sigma(c_0)_b$ in $\mathfrak{S}_n$ be the corresponding vertex permutations for $c_0$.  

For any other equivalent half-edge coloring $c$ (agreeing with $c_0$ away from $e$),  the colors from the set $S_e$ will be permuted by some permutation $\sigma\in\mathfrak{S}_{m_e}$, and $T_e$ by some permutation $\tau \in \mathfrak{S}_{m_e}$.   Therefore the $(-q)^{\ell(\sigma_w) + \ell(\sigma_b)}$ factor in the $c$ term in $\trqc(\Phi_q,G_m)$  will differ from the corresponding factor $(-q)^{\ell(\sigma_0)+\ell(\tau_0)}$ in the $c_0$ term by $(-q)^{\ell(\sigma) + \ell(\tau)}$.  Gathering, there will be one coloring $c$ equivalent to $c_0$ for every pair $\sigma,\tau$, giving (a factor of) a term in the sum:
$$ 
(-q)^{\ell(\sigma_0)+\ell(\tau_0)} (-q)^{\ell(\sigma)+\ell(\tau)} \Phi_q(e)_{i_{\sigma(1)}j_{\tau(1)}}\Phi_q(e)_{i_{\sigma(2)}j_{\tau(2)}} \dots \Phi_q(e)_{i_{\sigma(m_e)}j_{\tau(m_e)}}. 
$$

By Lemma \ref{lem:qdet_orders}, summing over all $\sigma,\tau$ in $\mathfrak{S}_{m_e}$ gives  
$$
(-q)^{\ell(\sigma_0)+\ell(\tau_0)}  \det_{q,S_e,T_e}(\Phi_q(e)) \sum_{\tau\in\mathfrak{S}_{m_e}} q^{2\ell(\tau)}.
$$  
Noting that 
$$
\sum_{\tau\in m_e} q^{2\ell(\tau)} = q^{\binom{m_e}{2}}[m_e]!
$$ 
(see e.g. \cite{sikora1}), and that this calculation applies independently on every edge $e$, the proof is complete.	
\end{proof} 

\begin{rema}\label{rema:qtraceispolyexpressioninqconnection}
\begin{enumerate}
\item
For a different perspective on this proof, via a quantum version of Grassmann variables \cite{bigYellowCFT1997}, see Appendix \ref{sec:qgrassmann}.  In particular, there it is shown that the full strength of the quantum matrix relations are not required, rather, a weaker set of relations (\ref{eq:mainRelation}) derived from them.
\item\label{item:qtraceispolyexpressioninqconnection}
Since the alternative quantum trace $\trqc^\prime(\Phi_q,m)$ is a polynomial expression in $q$ (over $\mathbb{A}$), so is the quantum trace $\trqc(\Phi_q,m)$.  Compare Remark \ref{rema:definitionofquantumtraceofmultiwebs} (\ref{item:rema:definitionofquantumtraceofmultiwebs}).  
\end{enumerate}
\end{rema}

\begin{defi}\label{defi:edge-coloring}
A half-edge $n$-coloring $c$ of a multiweb $m\in\Omega_n$ is an \emph{edge $n$-coloring} if $S_e=T_e$ for all edges $e$ of $m$, in which case one imagines just a single subset of $\{1,2,\dots,n\}$, say $S_e$, assigned to each edge $e$.
\end{defi}

\subsection{\texorpdfstring{$\GL_n$}{GLn}-connections}

\begin{defi}
A \emph{$\GL_n$-connection} $\Phi$ on $G$ assigns a matrix $\Phi(e)\in\GL_n(\mathbb{C})$ to every edge of $G$.  Note that:  (1)  when $q=1$, a $\GL_n$-connection $\Phi$ is, in particular, a quantum connection, $\Phi_1=\Phi$; and, (2) for all $q$, a diagonal $\GL_n$-connection $\Phi$, i.e. where all edge matrices $\Phi(e)$ are diagonal, is, in particular, a quantum connection, $\Phi_q=\Phi$.  Consequently, the trace $\mathrm{tr}_1(\Phi,m)\in\mathbb{C}$ is defined, as is $\trqc(\Phi,m)\in\mathbb{C}$ for diagonal $\Phi$ for all $q$, according to Definition \ref{defi:planarquantumtrace}.

Two $\GL_n$-connections $\Phi$ and $\Phi^\prime$ are \emph{$\GL_n$-gauge equivalent} if there are matrices $A_v\in\GL_n$ such that for all vertices $b,w$ and edges $e$ connecting $b$ to $w$ then $\Phi^\prime(e)=A_w\Phi(e)A_b$.  And they are \emph{diagonally gauge equivalent} if the matrices $A_v\in\GL_n$ can be taken to be diagonal matrices.

The \emph{identity connection} $I$ assigns the $n\times n$ identity matrix to every edge of $G$.

If $G$ is equipped with a $\GL_n$-connection $\Phi$ and $\gamma=v_1v_2\dots v_p v_1$ is a based oriented loop in $G$, the \emph{monodromy} of $\gamma$ with respect to $\Phi$ is $\Phi^\prime(v_p v_1)\dots\Phi^\prime(v_2v_3)\Phi^\prime(v_1v_2)\in\GL_n$ where $\Phi^\prime(v_i v_{i+1})$ equals $\Phi(v_i v_{i+1})$ if $v_i$ is black and equals $\Phi(v_i v_{i+1})^{-1}$ if $v_i$ is white.  Note that changing the base point changes the monodromy up to conjugation (so, for diagonal connections, the monodromy is defined independently of a choice of base point).
\end{defi}

\begin{prop}\label{prop:diagonalgaugeequivalence}
Two (resp. diagonal) $\GL_n$-connections $\Phi$ and $\Phi^\prime$ on $G$ are (resp. diagonally) $\GL_n$-gauge equivalent if and only if their based CCW face monodromies are the same for all faces.  (And, if the connections are valued in $\SL_n$, the gauge transformations $A_v$ may be as well.)
\end{prop}

\begin{proof}
The forward direction is clear.  The backward direction is by a standard spanning tree argument, similar to that presented in Section \ref{sssec:combinatorialconstructionofquantumidentityconnection}.  Indeed, trivializing the connection on a spanning tree of $G$ determines a sequence of gauge transformations taking $\Phi$ to $\Phi^\prime$ (which can be taken to be diagonal when $\Phi$ and $\Phi^\prime$ are diagonal).  The last statement is clear.
\end{proof}

\begin{rema}\label{rema:diagonal-connections-and-edge-colorings}
When $\Phi$ is a diagonal $\GL_n$-connection, then in the formula for $\trqc^\prime(\Phi,m)$ from (\ref{alttrace}) (and hence for $\trqc(\Phi,m)$ as well by Proposition \ref{prop:comparisonofqtraces}), only edge colorings $c$ (Definition \ref{defi:edge-coloring}) contribute to the sum over half-edge colorings.
\end{rema}

\begin{prop}\label{prop:tracefordiagconn}
For $\Phi$ a diagonal $\GL_n$-connection on $G$, and $m\in\Omega_n$ a multiweb,
\be \label{eq:formulaoftracefordiagonalconnection}
\trqc(\Phi,m) = \left(\prod_{e\in E}q^{\binom{m_e}{2}}\right)\sum_{c} \prod_{v\in V} (-q)^{\ell(\sigma_v)} \prod_{e\in E}  \prod_{i\in S_e} \Phi(e)_{ii}\in\mathbb{C}
\ee 
where the sum is over all edge colorings $c$ of the multiweb $m$.
\end{prop}

\begin{proof}
This follows immediately from Remark \ref{rema:diagonal-connections-and-edge-colorings} together with the simple observation that the quantum minor $\det_{q,S_e,S_e}(\Phi(e))$ equals the classical minor $\Det_{S_e,S_e}(\Phi(e))$ for diagonal $\GL_n$-connections $\Phi$.
\end{proof}

\begin{prop}\label{coro:diagconjugation}
If $\Phi$ and $\Phi^\prime$ are diagonal $\GL_n$-connections on $G$ that are diagonally $\SL_n$-gauge equivalent, then for all multiwebs $m\in\Omega_n$
$$
\trqc(\Phi,m)=\trqc(\Phi^\prime,m)\in\mathbb{C}.
$$
\end{prop}

\begin{proof}
Since for any given edge coloring $c$ the colors on the edges incident to $v$ give some permutation of $\{1,2,\dots,n\}$, all entries of the diagonal matrices $A_v$ enter into the rightmost product in Proposition \ref{prop:tracefordiagconn} exactly once. Since the product of the diagonal entries of $A_v\in\SL_n$ is $1$, the overall product is unchanged.
\end{proof}

\subsection{Quantum identity connection}\label{ssec:quantumidentityconnection}

\begin{defi}
Define the \emph{quantum identity matrix} $Q=Q_n\in\SL_n$ by
$$ 
Q = \begin{pmatrix} 
q^{n-1} & 0          & \dots&  & 0 \\
0          & q^{n-3} &&&\vdots \\
\vdots    &      & \ddots  &&\\
         &      & &q^{-(n-3)} & 0 \\
0          & \dots     && 0       & q^{-(n-1)}
\end{pmatrix}.
$$
Note that its matrix trace is simply the quantum integer $\Tr(Q) = [n]$. 

\emph{Sikora's connection} \cite[equation (5)]{sikora1}, 
here denoted $\Psi$, is the (not unique) diagonal $\SL_n$-connection constructed as follows.    For this geometric construction, it is helpful to think of the cilia as having a specified  angle in $[0,2\pi)$ around each vertex.    In particular, an isotopy of $G$ can also rotate the cilia (not allowing the cilia to cross any edges upon this rotation).  We now isotope $G$ so that edges are smooth curves and at each black vertex edges start in the direction  $(0,-1)$ (this is their initial tangent vector direction) and at each white vertex edges end in the direction $(0,-1)$, while at black vertices the cilia point north and at white vertices the cilia point south.  (Compare Figure \ref{fig:xmastreechandelier}.)  Then each edge $e$, oriented from black to white, has integrated curvature $2\pi\omega_e$, where $\omega_e$ is the integer CCW `winding number' associated to that edge.  On this edge put the connection  $Q^{-\omega_e}$.
\end{defi}

The following result is essentially a reformulation of Sikora \cite[Theorem 9]{sikora1}.  Recall the notation $\trq(m)$ from Remark \ref{rema:sikoratraceofmultiweb}.

\begin{theo}\label{theo:sikoraconnection}
For Sikora's connection $\Psi$ one has that for all multiwebs $m\in\Omega_n$
$$
\trq(m)=\trqc(\Psi,m)\in\mathbb{C}.
$$
\end{theo}

\begin{proof}
By the proof of Sikora \cite[Theorem 9]{sikora1}, for $m\in\Omega_n$ a proper multiweb
$$
\trq(m)=\sum_{c} \prod_{v\in V} (-q)^{\ell(\sigma_v)} \prod_{e\in E}  \prod_{i\in S_e} \Psi(e)_{ii}
$$
where the sum is over all edge colorings $c$ of the proper multiweb $m$.  By Proposition \ref{prop:tracefordiagconn}, $\trq(m)=\trqc(\Psi,m)$ as desired.  For $m$ a general multiweb, the result follows immediately from the proper multiweb case together with the definitions of $\trq(m)$ and $\trqc(\Psi,m)$.  (Here it is used that the connection $\Psi(G_m)$ on $G_m$ induced by $\Psi$ remains a Sikora connection.)
\end{proof}

\begin{defi}\label{defi:qidentityconn}
Suppose $\Phi$ is a diagonal $\SL_n$-connection on $G$ such that   any face of length $l$ with $k$ inward-pointing cilia has CCW monodromy $Q^{\frac{l}2-1-k}$.    Such a $\SL_n$-connection $\Phi$ is called a \emph{quantum identity connection} and is denoted $I_q$.    (Note, by Proposition \ref{prop:diagonalgaugeequivalence}, the quantum identity connection is unique up to diagonal gauge transformations.)  
\end{defi}

\begin{theo}\label{theo:quantumidentityconn}
For a quantum identity connection $I_q$ one has that for all multiwebs $m\in\Omega_n$
$$
\trq(m)=\trqc(I_q,m)\in\mathbb{C}
$$
and so $\Zq=Z(I_q)$. 
\end{theo}

\begin{proof} 
By Proposition \ref{prop:diagonalgaugeequivalence} and Proposition \ref{coro:diagconjugation} combined with Theorem \ref{theo:sikoraconnection}, it suffices to show that $I_q$ has the same CCW face monodromies as Sikora's connection $\Psi$. Consider the geometric construction of $\Psi$ above.  For a given face, traversed CCW, the total curvature of the tangent vector is $+2\pi$.   However when traversing a face, at each vertex with external cilium the local curvature there is $-\pi$, and at each vertex with internal cilium the local curvature is $+\pi$. The total curvature contribution along the edges is therefore $2\pi+\pi l_{\mathrm{int}}-\pi l_{\mathrm{ext}}$, where $l_{\mathrm{int}}$ is the number of vertices along the face with internal cilia and $l_{\mathrm{ext}}$ is the number of vertices along the face with external cilia. Since $l=l_{\mathrm{int}}+l_{\mathrm{ext}}$ the curvature contribution along the edges is $2\pi+2\pi l_{\mathrm{int}}-\pi l.$ Dividing by $2\pi$, the CCW monodromy of $\Psi$ around the face is $Q^{-(1+k-\frac{l}{2})}=Q^{\frac{l}2-1-k}$ as desired.  The last statement follows immediately from the definitions.
\end{proof}

\begin{rema}\label{rema:moreremarksregardingquantumidentityconnection}
\begin{enumerate}
\item
In particular, from the proof of Theorem \ref{theo:quantumidentityconn} it follows that Sikora's connection $\Psi$ is an example of a quantum identity connection, $I_q=\Psi$.
\item
It was shown in Theorem \ref{theo:polynomial} that $\trq(m)$ is a Laurent polynomial in $q^\frac{1}{n}$ for any multiweb $m\in\Omega_n$.  It follows from Theorem \ref{theo:quantumidentityconn} that, in the planar setting, $\trq(m)$ is in fact a Laurent polynomial in $q$, rather than $q^\frac{1}{n}$, by definition of $\trqc(I_q,m)$ as the contraction of codeterminants and dual codeterminants along the quantum identity connection $I_q$.  This will be strengthened, yet again, in Theorem \ref{theo:symmetricZq}.
\item\label{item2:rema:qtraceispolyexpressioninqconnection}
Note that the obvious adaptation of Proposition \ref{prop:ciliadependence} to the setting of graph connections holds when $q=1$ (following from Proposition \ref{prop:comparisonofqtraces}) but does not hold for general $q$, for either parity of $n$, due to the appearance of the $q^{\ell(\sigma_v)}$ terms in the definition of $\trqc^\prime(\Phi_q,m)$.  In other words, modifying the ciliation $L$, without also modifying the quantum connection $\Phi_q$, will not preserve the quantum traces $\trqc(\Phi_q,m)$ even up to a sign.  

How this is reconciled with Proposition \ref{prop:ciliadependence}, which is valid for general $q$, and Theorem \ref{theo:quantumidentityconn} is that when the ciliation $L$ is modified, the quantum identity connection $I_q$ (constructed, say, as Sikora's connection) is also modified in concert, so that the quantum trace $\trqc(I_q,m)$ is preserved (up to a sign) for any multiweb $m\in\Omega_n$ (and, for $n$ odd, $\trqc(I_q,m)$ is unchanged).
\item\label{rema:classicaltrace}
It immediately follows from the definitions that when $q=1$ then $Q=I$ hence the identity connection $I$ is a quantum identity connection, $I_1=I$.  In particular, $\mathbf{tr}_1(m)=\mathrm{tr}_1(I,m)$ by Theorem \ref{theo:quantumidentityconn}.  Proposition \ref{id} discusses when the identity connection $I$ is a quantum identity connection, $I_q=I$, for all $q$.  

Lastly, note that if a diagonal $\Phi$  is a quantum identity connection for $q=1$, $I_1=\Phi$,  then $\Phi$ need not equal the identity connection $I$, but it is diagonally gauge equivalent to $I$.
\end{enumerate}
\end{rema}

\begin{defi}\label{defi:classicaltrace}
The quantity 
$$
\mathbf{tr}_1(m)=\mathrm{tr}_1(I,m)
$$
is called the \emph{classical trace} of a multiweb $m\in\Omega_n$, and by definition it is an integer. In fact, it counts edge $n$-colorings, see Proposition \ref{prop:positivecilia}.
\end{defi}

\subsubsection{Combinatorial construction of quantum identity connection}\label{sssec:combinatorialconstructionofquantumidentityconnection}

Choose a spanning tree of $G$ and put the identity matrix $I$ on all the edges of the spanning tree.  Progressively working outwards from the leaves of the dual tree,    for each edge $e$ not in the spanning tree choose $I_q(e)=Q^\alpha$ for the appropriate power $\alpha$ such that the defining monodromy condition (Definition \ref{defi:qidentityconn}) of a quantum identity connection is met (on all faces contained in the cycle in $G$ associated to $e$).

\subsection{Positivity}

\begin{defi}\label{defi:positiveciliation}
A ciliation $L$ of $G$ is \emph{positive}, denoted $L=L^+$, if there are an even number of cilia pointing into every face of $G$.  
\end{defi}

\begin{rema}
As explained in \cite{DKS}, one way to produce a positive ciliation $L^+$ of $G$ is to first choose a dimer cover of $G$ and then to define the cilia such that for each dimer, the cilia at its vertices go into the same adjacent face.
\end{rema}

\begin{prop}[\cite{DKS}]\label{prop:positivecilia}
For all multiwebs $m\in\Omega_n$, there exists at least one edge coloring of $m$, and the classical trace $\mathrm{tr}_1(I,m)\in\Z\setminus\{0\}$ equals $\epsilon$ times the number of edge colorings of $m$, where the sign $\epsilon\in\{\pm1\}$ depends on the ciliation $L$.  When $L=L^+$ is a positive ciliation, then $\epsilon=+1$.  
\qed\end{prop}

\begin{rema}\label{rema:consequencesofq=1positivecilia}
\begin{enumerate}
\item\label{item3:rema:consequencesofq=1positivecilia}
By Propositions \ref{prop:comparisonofqtraces} and \ref{prop:positivecilia}, it follows that for a positive ciliation $L^+$ the quantum trace $\trqc(\Phi_q,m)$ is a positive polynomial expression in $q$ (over $\mathbb{A}$).  Compare Remark \ref{rema:qtraceispolyexpressioninqconnection}
(\ref{item:qtraceispolyexpressioninqconnection}).
\item\label{item:rema:consequencesofq=1positivecilia}
It is not true that $\trqc(\Phi_q,m)$ is independent of the choice of positive ciliation $L^+$ for a general quantum connection $\Phi_q$.  
\item\label{item2:rema:consequencesofq=1positivecilia}
It is true that $\trq(m)$ is independent of the choice of positive ciliation $L^+$.  Indeed, $\trq(m)$ only changes up to sign when changing cilia by Proposition \ref{prop:ciliadependence}, while $\mathbf{tr}_1(m)=\mathrm{tr}_1(I,m)$ (Remark \ref{rema:moreremarksregardingquantumidentityconnection} (\ref{rema:classicaltrace})) is positive for $L^+$ by Proposition \ref{prop:positivecilia}.  
\item
Despite Remark \ref{rema:consequencesofq=1positivecilia} (\ref{item:rema:consequencesofq=1positivecilia}), it is true that $\trqc(\Phi_q,m)$ is independent of the choice of positive ciliation $L^+$ for quantum identity connections $\Phi_q=I_q$.  Indeed, then $\trqc(I_q,m)=\trq(m)$, and the latter is independent of the choice of positive cilia by Remark \ref{rema:consequencesofq=1positivecilia} (\ref{item2:rema:consequencesofq=1positivecilia}).  
\end{enumerate}
\end{rema}

The following statement is immediate from Remark \ref{rema:consequencesofq=1positivecilia}.

\begin{prop}
For any quantum identity connection $I_q$, the quantum partition function $\Zq=\Zqc(I_q)$ is independent of the choice of positive ciliation $L^+$.  
\qed\end{prop}

\begin{defi}\label{defi:positivequantumpartitionfunction}
The \emph{quantum $n$-dimer partition function} associated to an (unciliated) embedded planar bipartite graph $G$ is $\Zq=\Zqc(I_q)$ evaluated for any choice of positive ciliation $L^+$ and quantum identity connection $I_q$ with respect to $L^+$.  It is denoted $\Zq^+$. 
\end{defi}

\begin{rema}\label{rema:classicalndimerpartitionfunction}
The classical $n$-dimer partition function is, by definition, $Z(n)=(Z_\mathrm{dimer})^n$ (Definition \ref{defi:graphassumptions} and Remark \ref{rema:afterfirstdefofqpartitionfunction} (\ref{item:rema:afterfirstdefofqpartitionfunction})), which is defined for $G$ independent of any cilia considerations (nor does it depend on the particular planar embedding of $G$).  By the results of this subsection, when $q=1$ one has $\mathbf{Z}_1^+=Z(n)$, which is also the same thing as $Z_1(I)$ for the $n\times n$ identity connection $I$ evaluated for any positive ciliation $L^+$.  (See e.g. \cite{DKS}.)
\end{rema}

\begin{exam}
While the trace $\trq(m)$ depends on the cilia only up to a sign, this is not true for $\Zq$.    For example, when $n=2$ and $G$ is the graph with two vertices and two edges, then $\mathbf{Z}_q^+=1+1+[2]_q=2+q+\frac1q$,   while choosing one cilia in and one cilia out then $\mathbf{Z}_q=2-q-\frac1q$. 
\end{exam}

\subsection{Symmetry}

\begin{defi}\label{def:trivialciliation}
A ciliation $L$ of $G$ is \emph{trivial} if the identity connection $I$ is a quantum identity connection, $I_q=I$.
\end{defi}

\begin{prop}\label{id}
If a ciliation $L$ satisfies the property that every bounded face of $G$ has $k=\frac{l}{2}-1$ inward pointing cilia   (where $l$ is the length of the face), then $L$ is trivial. Such a trivial ciliation $L$ exists. 
\end{prop}

\begin{proof} 
The first statement follows  immediately from the definition of a quantum identity connection.

For the existence, (temporarily) subdivide every face of length $>4$ into quadrilaterals, by adding edges across faces (and no new vertices),   so that any face of length $l$ will be subdivided into $\frac{l}2-1$ quadrilaterals.  Recall that  $G$ is assumed $2$-connected; it remains $2$-connected after adding edges.   The subdivided graph has all faces of  length $2$ or $4$.  We now show that when all faces have length $2$ or $4$, the cilia can be chosen so that each quad face has exactly one inwardly-pointing cilium, and bigon faces contain no cilia. Then removing the added edges without changing the cilia locations will complete the proof. 

Assign to each edge $e$ of $G$ a generically chosen conductance $c_e\in(0,\infty)$. Let $v_0,v_1$ be two vertices on the outer face of $G$, and find a function $f:V\to\R$   satisfying $f(v_0)=0, f(v_1)=1$ with $f$ harmonic on all other vertices, that is, $f(v)$ is the weighted average of its neighboring values,   weighted by the conductances. By genericity (and $2$-connectivity) all $f$ values are distinct. 

Around the vertices of each face of $G$ we claim that there is exactly one local maximum of $f$ and one local minimum (including the external face).   Suppose not: suppose on a face with vertices $a,b,c,d$ in cyclic order we have $f(a)>f(b)<f(c)>f(d)<f(a)$.     Then we can find, by the mean value property for $f$,  paths from $a$ and $c$ to $v_1$ on which $f$ is increasing,   and likewise paths from $b$ and $d$ to $v_0$  on which $f$ is decreasing. By planarity these paths must cross,   violating the fact that the values on the increasing paths are all larger than the values on the decreasing paths. 

Therefore $f$ defines a `left' and `right' side of each face: when facing from the minimum to the maximum, the left side is the  interval on the left between the min and max.  We now put a cilium in each face for the interior (i.e. $\neq$ min/max) white vertices on the left side (if any) and interior (i.e. $\neq$ min/max) black vertices on the right side (if any).  This assignment of cilia to interior vertices is well-defined (independent of which face containing the vertex is referenced).  Indeed, orienting edges from lower to higher value of $f$, it follows by the harmonic property that the outgoing (resp. ingoing) arrows from a vertex are consecutive in cyclic order (in other words, no vertex is a `saddle point').

In this way, each quad face gets exactly one cilia, since there is either one white on the left or one black on the right (but not both).    Put cilia for any unassigned vertices of the external face, including $v_0,v_1$, to point into the external face.
\end{proof}

\begin{rema}\label{3conn} 
If $G$ is $3$-connected, possibly containing bigons, we can argue more simply as follows, beginning after the second paragraph of the previous proof.   By a theorem of Tutte \cite{Tutte}, after isotopy all faces are convex (where bigon faces degenerate to segments), all edges are straight lines, and no edges are perfectly horizontal.  Then choose the cilia to point left at white vertices and right at black vertices, say.  By these assumptions, each quadrilateral face will have one cilium inside.  
\end{rema}

\begin{defi}
A Laurent polynomial in $q$ is \emph{symmetric} if it is invariant under $q\mapsto q^{-1}$.  A polynomial in $q$ is \emph{palindromic} if there exists a nonnegative integer or half-integer $\alpha$ such that multiplying this polynomial by $q^{-\alpha}$ yields a symmetric Laurent polynomial in $q$.  
\end{defi}

\begin{theo}\label{theo:symmetricZq}
For any ciliation $L$ and multiweb $m\in\Omega_n$, the quantum trace $\trq(m)=\trqc(I_q,m)$ is a nonzero palindromic   polynomial in $q$ with integer coefficients (all with the same sign), and with nonnegative coefficients when $L=L^+$ is positive.   Moreover, multiplying $\trq(m)$ by $q^{-N\binom{n}{2}}$ yields a symmetric Laurent polynomial.    It follows that $\Zq=\Zqc(I_q)$ is a symmetric Laurent polynomial in $q$ with integer coefficients,   and for positive ciliations $L=L^+$ the quantum partition function $\Zq=\Zq^+$ is a nonzero symmetric Laurent polynomial in $q$ with nonnegative coefficients. 
\end{theo}

\begin{proof}  
First choose $L$ to be a trivial ciliation, which is possible by Proposition \ref{id},    so that $I_q$ may be taken to be $I$.  If $m\in\Omega_n$ is a proper multiweb, then, by (\ref{eq:formulaoftracefordiagonalconnection}),
\be\label{eq:trivialciliaformula}
\trq(m) =\trqc(I_q,m) =  \sum_c\prod_v (-q)^{\ell(\sigma_v)}
\ee 
where the sum is over edge colorings $c$. Reversing the order of the colors (thereby reindexing the sum) replaces $\ell(\sigma_v)$ with $\binom{n}{2}-\ell(\sigma_v)$ at each vertex $v$. Thus 
$$
\trq(m)=\sum_c \prod_v (-q)^{\binom{n}{2}-\ell(\sigma_v)} = q^{2N\binom{n}{2}}\mathbf{tr}_{q^{-1}}(m).
$$ 
This shows that  
$$
q^{-N\binom{n}{2}}\trq(m) = q^{N\binom{n}{2}}\mathbf{tr}_{q^{-1}}(m)
$$ 
is symmetric. If $m$ is not proper, divide both sides by the symmetric Laurent polynomial $\prod_e [m_e]!$.  The statement for general ciliations $L$ follows from the trivial case by Remark \ref{rema:moreremarksregardingquantumidentityconnection} (\ref{item2:rema:qtraceispolyexpressioninqconnection}), and for positive ciliations $L=L^+$ by Remark \ref{rema:consequencesofq=1positivecilia} (\ref{item3:rema:consequencesofq=1positivecilia}).
\end{proof}

\begin{theo}\label{theo:dependenceonplanarembedding}
The quantum traces $\trq(m)$ for positive cilia $L^+$ are independent of the particular planar embedding of $G$.  In other words, they depend only on the combinatorial structure of $G$ as captured by the multiwebs $m\in\Omega_n$.  It follows that the quantum $n$-dimer partition function $\Zq^+$ is assigned to any abstract planar (unciliated) bipartite graph $G$.  
\end{theo}

\begin{proof}
Note the result is immediate when $n=2$ and $q=q$ by (\ref{eq:n=2partitionfunction}) below, and also when $n=n$ and $q=1$ by Proposition \ref{prop:positivecilia}.      Going forward, it suffices to assume $m$ is proper.

To start, note if $G$ is $3$-connected, then there are two embeddings:  $G$ and its reflection \cite{Whitney}.    That $\trq(m)$ is invariant under full reflection of $G$ (for positive cilia) is equivalent to the proof of Theorem \ref{theo:symmetricZq},   since reflecting is equivalent to reversing the order of the colors.  

Otherwise, assume $G$ has two vertices whose removal disconnects $G$ into (at least) two embedded components $\mathrm{C}, \mathrm{E}$,   so that the task is to show $\trq(m)$ is invariant upon reflecting the component $\mathrm{E}$ across an axis to give a newly embedded component $\Er$   (while keeping the embedding of $\mathrm{C}$ unchanged).  Such operations connect the space of embeddings (\cite{Whitney}).

While a purely combinatorial argument, akin to that for the $3$-connected case, escaped the authors, the following quantum topological argument is a natural substitute.  (As usual, cilia/sign considerations do not affect the argument, so will be systematically ignored.)  

Since $G$ is actually embedded in $\R^2\subset\R^3$, one is tempted to think that the isotopy invariance of $\trq(m)=\trq(\mathbf{W}_m)$ (where $\mathbf{W}_m$ is the corresponding $n$-web in $\R^3$) might allow one to simply rotate the $\mathrm{E}$ component in $\mathrm{C}\cup \mathrm{E}$ in $\R^3$ by $180$ degrees about the reflection axis to achieve $\mathrm{C}\cup\Er$.  Unfortunately, the resulting ribbon structure is different.  Nevertheless, quantum topology does allow for this `rotation' to be implemented, preserving the ribbon structure, without ever leaving the plane.  (That is, purely diagrammatically.)  

Take a vertex $v$ of $\mathrm{E}$, not part of the reflection axis, and drag it over the component $\mathrm{E}$ by isotopy,   leaving the axis as well as $\mathrm{C}$ fixed throughout, so that the resulting image is $\mathrm{C}\cup\Er$ except for some number of `kinks'   at vertices that are created during the process.  There is a skein relation, easily derivable from the Reshetikhin--Turaev construction of $\trq(\mathbf{W}_m)$ presented   in Section \ref{sssec:proofofLaurentproperty}, saying that these vertex kinks can be removed at the cost of multiplying by a power $q^\alpha$ (with a sign)   where $\alpha$ is $-1-\frac{1}{n}$ (resp. $1+\frac{1}{n}$) for positive/right handed (resp. negative/left handed) kinks   (regardless of whether the vertex is black or white).  See Figure \ref{fig:positivevertexkink}.  

By the isotopy invariance of the trace combined with the kink removing skein relation, the result of this `$180$ degree rotation over the axis' is thus multiplication by $q^\beta$ for some $\beta$.  One could justify, by a standard index argument, that the same number of positive and negative kinks are created during the isotopy, hence $\beta=0$.  Here is another argument.  

Just as above, by isotopy and kink removals bring $\mathrm{C}\cup\Er$ back to $\mathrm{C}\cup \mathrm{E}$, except this time implementing a `$180$ degree rotation under the axis' (that is, dragging the vertex $v$ under strands whenever it went over during the first isotopy, and vice versa).  This multiplies the trace by another factor of $q^\beta$, so that in total $\trq(\mathbf{W}_m)=q^{2\beta}\trq(\mathbf{W}_m)$.  Since $\trq(\mathbf{W}_m)=\trq(m)$ is not identically zero (it is positive when $q=1$), it follows that $\beta$ must be zero.   
\end{proof}

\begin{figure}[t]
\centering
\includegraphics[scale=1.25]{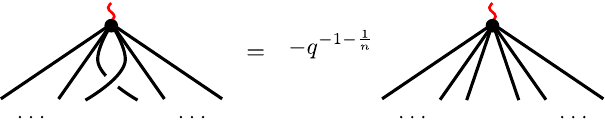}
\caption{Removing a positive vertex kink from an $n$-web $\mathbf{W}$.}
\label{fig:positivevertexkink}
\end{figure}

\section{Quantum Kasteleyn connection}\label{sec:quantumkasteleynconnection}

\subsection{\texorpdfstring{$q$}{a}-Kasteleyn matrix}

\begin{defi}
Let $G$ be a ciliated planar graph, as previously.  A \emph{choice of Kasteleyn signs} with respect to the ciliation $L$ is a function $\epsilon:E\to\{\pm1\}$ satisfying the property that, for every face of $G$ with $l$ sides and $k$ inward-pointing cilia, the product of the Kasteleyn signs $\epsilon(e)$ around the edges $e$ of the face equals $(-1)^{\frac{l}{2}-1}$ for $n$ odd and $(-1)^{\frac{l}{2}-1-k}$ for $n$ even.  

For a quantum connection $\Phi_q$ and a choice of Kasteleyn signs $\epsilon$, the \emph{$q$-Kasteleyn matrix with respect to $\Phi_q$}, denoted $K(\Phi_q)\in \Mat_N(\Mat_n(\mathbb{A}))$, is defined by putting $K(\Phi_q)_{ij}$ equal to the sum over the edges $e$ from the $j$-th black vertex $b_j$ to the $i$-th white vertex $w_i$ of the matrices $\epsilon(e)\Phi_q(e)\in \Mat_n(\mathbb{A})$.  (Note here there is only a true sum for some $i$ and $j$ if $G$ is not simple.)  It is emphasized once again that the Kasteleyn signs, hence $K(\Phi_q)$ as well, are defined relative to the ciliation $L$.  The $q$-Kasteleyn matrix can be considered as an $Nn\times Nn$ matrix $\tilde{K}(\Phi_q)=(\tilde{K}(\Phi_q)_{\tilde{i}\tilde{j}})$ over $\mathbb{A}$ in the obvious way by replacing each matrix $K(\Phi_q)_{ij}$ with its corresponding $n\times n$ block.  

When $q=1$, a quantum connection $\Phi_q=\Phi_1$ consisting of quantum matrices $\Phi_1(e)$ of the kind described in Example \ref{exam:quantummatrixexamples} (\ref{item2:exam:quantummatrixexamples}) is called a \emph{classical connection}.  Of course, the usual determinant $\Det(\tilde{K}(\Phi_1))$ of the $Nn\times Nn$ matrix $\tilde{K}(\Phi_1)$ over commuting variables makes sense for a classical connection.  
\end{defi}

\begin{theo}[\cite{DKS, KenyonOvenhouse}]\label{theo:classicalkasteleyntheorem}
For any classical connection $\Phi_1$ one has $\pm \Det(\tilde{K}(\Phi_1))=Z_1(\Phi_1)$.  
\qed\end{theo}

\begin{rema}
\begin{enumerate}
\item
This was proved for positive ciliations $L^+$ in \cite{DKS}, and for general ciliations $L$, in addition to more general connections corresponding to the `mixed dimer' setting, in \cite{KenyonOvenhouse}.  
\item
Note for positive ciliations $L^+$ (Definition \ref{defi:positiveciliation}) that when $n$ is even the sign monodromy $(-1)^{\frac{l}{2}-1-k}=(-1)^{\frac{l}{2}-1}$ appearing in the Kasteleyn condition is the usual one and the same as for $n$ odd.  In particular, when evaluated with positive cilia the Kasteleyn determinant $\Det(\tilde{K}(I))$ for the identity connection $I$ equals, up to a global sign, the $n$-dimer partition function $Z(n)$ (Remark \ref{rema:classicalndimerpartitionfunction}).  (In fact, the global sign is always $+$ for $n$ even, while for $n$ odd it depends on the arbitrary orderings of the black and white vertices from $1,2,\dots,N$ as well as the choice of Kasteleyn signs.)
\item
Note that, despite the terminology, neither $\tilde{K}(\Phi_q)$ nor the matrix entries $K(\Phi_q)_{ij}$ are typically quantum matrices, in the technical sense used previously.  In particular, the quantum determinant of $\tilde{K}(\Phi_q)$ is not defined.  
\end{enumerate}
\end{rema}

\subsection{\texorpdfstring{$q$}{q}-Kasteleyn determinant}\label{ssec:qkastdet}

\begin{defi}
The \emph{blow up graph} $\tilde{G}$ is the (nonplanar) bipartite graph obtained by replacing each vertex $v$ with $n$ copies and each edge $e$ of $G$ with a   copy of the complete bipartite graph $K_{n,n}$.  There is a natural projection from the vertices and edges of $\tilde{G}$ to those of $G$.   Over each white vertex $w_i$ of $G$ there are $n$ vertices $\tilde w_{i,1},\dots,\tilde w_{i,n}$ of $\tilde G$. The ordering of the white vertices   $w_1, w_2, \dots, w_N$ of $G$ determines a total ordering of the white vertices of $\tilde G$, the lexicographic ordering   ($w_{i,l}<w_{i',l'}$ if either $i<i'$ or $i=i'$ and $l<l'$).    When convenient $\tilde{w}_1, \tilde{w}_2, \dots, \tilde{w}_{Nn}$ may also be written for this total order, so that $\tilde w_{i,l}=\tilde w_{N(i-1)+l}$. Similarly for the black vertices.

The quantum connection $\Phi_q$ on $G$ lifts, in a sense, to a \emph{quantum scalar connection} $\tilde{\Phi}_q$ on $\tilde{G}$ such that for an edge $e$ of $G$ between vertices $b_j$ and $w_i$ the element $\Phi_q(e)_{lk}\in\mathbb{A}$ is assigned by $\tilde{\Phi}_q$ to the edge $\tilde{e}$ of $\tilde{G}$ lying over $e$ and connecting the vertices $\tilde{b}_{j,k}$ and $\tilde{w}_{i,l}$ lying over $b_j$ and $w_i$.  Then, for instance, for $\tilde{i}=N(i-1)+l$ and $\tilde{j}=N(j-1)+k$ the $q$-Kasteleyn matrix entry $\tilde{K}(\Phi_q)_{\tilde{i}\tilde{j}}\in\mathbb{A}$ is the same as the sum over edges $\tilde{e}$ of $\tilde{G}$ connecting vertices $\tilde{b}_{j,k}$ and $\tilde{w}_{i,l}$ of the quantities $\epsilon(e)\tilde{\Phi}_q(\tilde{e})\in\mathbb{A}$ where $e$ is the projection in $G$ of $\tilde{e}$.

Let $\tilde{\Omega}_1$ denote the set of dimer covers of $\tilde{G}$.  To a dimer cover $\tilde{m}\in\tilde{\Omega}_1$ is associated the following data:  
\begin{enumerate}
\item  
The edges $\tilde{e}$ of $\tilde{G}$ appearing in $\tilde{m}$, which project to edges $e$ of $G$.  
\item 
A permutation $\tilde{\sigma}\in\mathfrak{S}_{Nn}$.  This is defined by sending $\tilde{i}$ to the unique $\tilde{j}$ such that there is an edge $\tilde{e}$ in $\tilde{m}$ between $w_{\tilde{i}}$ and $b_{\tilde{j}}$.  Denote this edge by $\tilde{e}_{\tilde{i},\tilde{\sigma}(\tilde{i})}$ and its projection by $e_{\tilde{i},\tilde{\sigma}(\tilde{i})}$.  
\item  
A multiweb $m\in\Omega_n$, determined by projecting to $G$ the edges $\tilde{e}$ of $\tilde{m}$.  The multiplicities $m_e$ are precisely the number of edges $\tilde{e}$ of $\tilde{m}$ lying over $e$.  
\item 
A half-edge coloring $c$ of $m$.  Here if $e$ connects vertices $b_j$ and $w_i$ and if $\tilde{b}_{j,1}, \dots, \tilde{b}_{j,n}$ and $\tilde{w}_{i,1}, \dots, \tilde{w}_{i,n}$ are the corresponding vertices of $\tilde{G}$ lying above them, then the subset $T_e$ (resp. $S_e$) of $\{1,2,\dots,n\}$ assigned to the half-edge of $e$ connected to $b_j$ (resp. $w_i$) consists of all the indices $k$ (resp. $l$) such that there is a lift $\tilde{e}$ in $\tilde{m}$ attaching to $\tilde{b}_{j,k}$ (resp. $\tilde{w}_{i,l}$).
\item 
The vertex permutation $\sigma_v\in \mathfrak{S}_n$ associated to each vertex $v\in V$ of $G$ by the half-edge coloring $c$ of the multiweb $m$ of $G$ (defined in Section \ref{alt}).
\item\label{enum:item:definitionofh}
The \emph{edge permutation} $\sigma_e\in \mathfrak{S}_{m_e}$ associated to each edge $e\in E$, from $b_j$ to $w_i$ say, defined as follows.    First, identify $T_e$ (resp. $S_e$) with $\{1,2,\dots,m_e\}$ by the unique monotonically increasing bijection $g_e$ (resp. $f_e$).    Second, let $h_e$ be the unique bijection from $S_e$ to $T_e$ such that $\tilde b_{j,k}$ and  $\tilde w_{i,l}$ are connected by an edge $\tilde{e}$ of $\tilde{m}$, where $k=h_e(l)$.    Lastly, define $\sigma_e=g_e \circ h_e\circ f_e^{-1}$. See Figure \ref{fig:edgePerm} for an example.
\end{enumerate}

Finally, define the \emph{$q$-Kasteleyn determinant with respect to $\Phi_q$}, denoted $\Kdetq(\Phi_q)\in\mathbb{A}$, by
\begin{gather*}
\Kdetq(\Phi_q)=\\
q^{-N\binom{n}{2}}\sum_{\tilde{m}\in\tilde{\Omega}_1}(-1)^{\tilde{\sigma}}\left(\prod_{v\in V}q^{\ell(\sigma_v)}\right)\left(\prod_{e\in E}q^{\binom{m_e}{2}}q^{\ell(\sigma_e)}\right)\left(\prod_{\tilde{i}=1}^{Nn}\epsilon(e_{\tilde{i}\tilde{\sigma}(\tilde{i})})\right)
\tilde{\Phi}_q(\tilde{e}_{1\tilde{\sigma}(1)})\tilde{\Phi}_q(\tilde{e}_{2\tilde{\sigma}(2)})\dots\tilde{\Phi}_q(\tilde{e}_{Nn\tilde{\sigma}(Nn)}).
\end{gather*}
Note here that the order matters for the, not necessarily commuting, quantum scalar connection terms $\tilde{\Phi}_q(\tilde{e}_{\tilde{i}\tilde{\sigma}(\tilde{i})})\in\mathbb{A}$.    Note again that the ciliation $L$ is required to define both the local vertex permutations $\sigma_v$ as well as   the Kasteleyn signs $\epsilon(e_{\tilde{i}\tilde{\sigma}(\tilde{i})})$.  (Note that the signature of $\tilde\sigma$, $(-1)^{\tilde{\sigma}}$,   satisfies $(-1)^{\tilde{\sigma}}=(-1)^{\ell(\tilde{\sigma})}$.)
\end{defi}

\begin{rema}\label{rema:Kdetasdeterminant}
When $G$ is a simple graph,   then the function sending a dimer cover $\tilde{m}\in\tilde{\Omega}_1$ of $\tilde{G}$ to $\tilde{\sigma}\in\mathfrak{S}_{Nn}$ is injective, and the above formula for the $q$-Kasteleyn determinant $\Kdetq(\Phi_q)$ simplifies to the following (more determinant-like) formula in terms of the $q$-Kasteleyn matrix $K=K(\Phi_q)$, or rather $\tilde{K}=\tilde{K}(\Phi_q)$, 
$$
\Kdetq(\Phi_q)=q^{-N\binom{n}{2}}\sum_{\tilde{\sigma}\in\mathfrak{S}_{Nn}}(-1)^{\tilde{\sigma}}\left(\prod_{v\in V}q^{\ell(\sigma_v)}\right)\left(\prod_{e\in E}q^{\binom{m_e}{2}}q^{\ell(\sigma_e)}\right)\tilde{K}_{1\tilde{\sigma}(1)}\tilde{K}_{2\tilde{\sigma}(2)}\dots\tilde{K}_{Nn\tilde{\sigma}(Nn)}.
$$
Note those $\tilde{\sigma}\in\mathfrak{S}_{Nn}$ not corresponding to dimer covers $\tilde{m}\in\tilde{\Omega}_1$ are automatically zero in this formula, by definition of $\tilde{K}\in \Mat_{Nn}(\mathbb{A})$.  

Note also that when $q=1$ and when $\Phi_q=\Phi_1$ is a classical connection, then the above formula reduces to the classical Kasteleyn determinant, $\mathrm{Kdet}_1(\Phi_1)=\Det(\tilde{K}(\Phi_1))$, which is valid for nonsimple graphs as well.  
\end{rema}

\begin{theo}\label{theo:quantumkasteleyndet}
For any quantum connection $\Phi_q$ one has $\pm \Kdetq(\Phi_q)=\Zqc(\Phi_q)$.
\end{theo}

\begin{proof}  
This is more or less immediate from the proof of the $q=1$ case (which can be used to keep track of signs), Proposition \ref{prop:comparisonofqtraces}, and a simple bookkeeping of powers of $q$.  Details are provided in Section \ref{sssec:proofofqkastdet}.
\end{proof}

\begin{rema}
\item  For any choice of quantum identity connection $\Phi_q=I_q$ one has $\pm \Kdetq(I_q)=\Zq$, while for any positive ciliation $L^+$ one has $\pm \Kdetq(I_q)= \Zq^+$.
\end{rema}

\begin{figure}[t]
\centering
\begin{tikzpicture}[scale=1.5]
\tikzset{snake it/.style={decorate, decoration=snake}}

    \draw (3.7,.5)--(-.6,.5)--(-1.9,-.6)--(2.4,-.6)--cycle;
    
    \node[draw=black,fill=white,shape=circle,label={[label distance=10pt]west:\tiny $w$}] (A) at (0,0) {};
    \node[draw=black,fill=black,shape=circle,label={[label distance=10pt]east:\tiny $b$}] (B) at (2,0) {};
    \draw[line width=.6] (A.-14)--(B.194);
    \draw[line width=.6] (A) -- (B);
    \draw[line width=.6] (A.14)--(B.166);
    

    \node[scale=.5] at (.45,.2) {$\{1,4,5\}$};
    \node[scale=.5] at (1.55,.2) {$\{2,3,5\}$};
    \node[scale=.5] at (-.25,.35) {$\{2\}$};
    \node[scale=.5] at (-.25,.-.35) {$\{3\}$};
    \node[scale=.5] at (2.25,.25) {$\{1,4\}$};

    \foreach \x in {1,...,5}
    {\tikzmath{\y=int(5-\x+1);}
        \node[draw=black,fill=white,shape=circle,scale=.5,label={[label distance=10pt]west:\tiny $w_\y$}] (A\x) at (0,.5*\x +.5) {};
        \node[draw=black,fill=black,shape=circle,scale=.5,label={[label distance=10pt]east:\tiny $b_\y$}] (B\x) at (2,.5*\x +.5) {};
    }

    \draw[] (A1)--(B4);
    \draw[] (A2)--(B1);
    \draw[] (A5)--(B3);

    \draw[dashed,line width=.6] (A)--(-1.1,-.6);
    \draw[dashed,line width=.6] (A)--(-.7,.4);
    \draw[dashed,line width=.6] (B.-5)--(3.4,.5);
    \draw[dashed,line width=.6] (B.5)--(3.4,.54);

    \draw[dashed] (A3)--(-1,1.5);
    \draw[dashed] (A4)--(-1,3);   
    \draw[dashed] (B2)--(3.5,2);
    \draw[dashed] (B5)--(3.5,3.5);
\end{tikzpicture}
\caption{An example of an edge permutation $\sigma_e=\sigma_{bw}\in\mathfrak{S}_{m_e}$ corresponding to the edge $e=bw$ ($G$ simple). Shown is a local picture of the graph $G$ and its blowup $\tilde{G}$. The permutation $\sigma_{bw}\in \mathfrak{S}_3$, corresponding to sending $w_1\mapsto b_3$, $w_4\mapsto b_5$, and $w_5\mapsto b_2$, is $231\in\mathfrak{S}_3$ (namely, $1\mapsto2, 2\mapsto3, 3\mapsto1$) in the natural order on $S_e=\{1,4,5\}$ and $T_e=\{2,3,5\}$. }
\label{fig:edgePerm}
\end{figure}
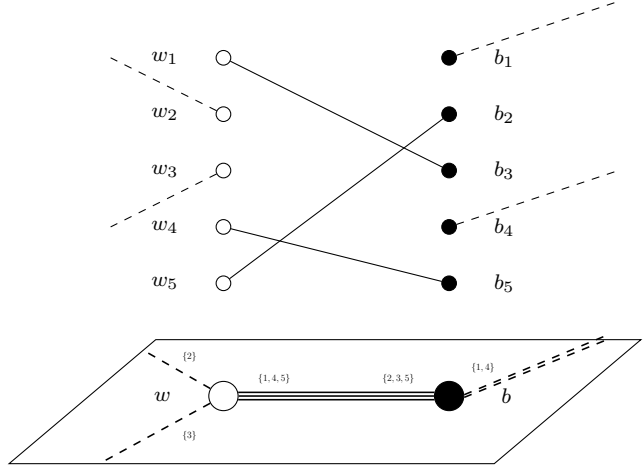

\subsubsection{Proof of Theorem \ref{theo:quantumkasteleyndet}}\label{sssec:proofofqkastdet}  

Recall the notion of an edge permutation $\sigma_e\in\mathfrak{S}_e$ from Definition \ref{ssec:qkastdet}.  See Figure \ref{fig:edgePerm}.

It is also possible to view, see below, a `local' edge permutation $\sigma_e \in \mathfrak{S}_{m_e}$, depending on a coloring $c$, as a `global' edge permutation $\tilde{\sigma}_e\in\mathfrak{S}_{Nn}$    (thought of as a permutation from whites to whites)   by choosing an indexing of  the white  vertices of $G$, which induces a (lexicographic) total order on all the white vertices of $\tilde{G}$.    (It is clear that as permutations in $\mathfrak{S}_{Nn},$ two edge permutations corresponding to different edges commute.)    However, when writing $\ell(\sigma_e)$, then $\sigma_e$ is always being thought of as a permutation of $m_e$ elements.  

The global edge permutations $\tilde{\sigma}_e=\tilde{\sigma}_{e,c}\in\mathfrak{S}_{Nn}$ subordinate to $c$ are essentially defined by the following property.    Also depending on the coloring $c$ of $m$, let $\tilde{\sigma}_{0,c} \in \mathfrak{S}_{Nn}$ be the permutation    (thought of as a permutation from whites to blacks)   such that all of its local edge permutations $\sigma_{0,e,c}\in\mathfrak{S}_{m_e}$ are the identity (compare Figure \ref{fig:edgePerm} where the identity edge permutation would match $w_1\mapsto b_2$, $w_4\mapsto b_3$, and $w_5\mapsto b_5$).  Here the black vertices are also totally ordered.  Then any other permutation $\tilde{\sigma}=\tilde{\sigma}_c\in\mathfrak{S}_{Nn}$ (from whites to blacks) corresponding to $c$ can be written as  
$$
\tilde{\sigma} =  \tilde{\sigma}_{0,c}\circ\left(\prod_e \tilde{\sigma}_e\right)\in\mathfrak{S}_{Nn}
$$
(so, white to black equals first white to white followed by white to black).

From the discussion above, one can write 
\begin{gather*}
\Kdetq(\Phi_q) =\\ q^{-N\binom{n}{2}}\sum_{m \in \Omega_n} \sum_{\tilde{m} \in m} (-1)^{\tilde{\sigma}}\left(\prod_v q^{\ell(\sigma_v)}\right)\left(\prod_e q^{\binom{m_e}{2}}q^{\ell(\sigma_e)}\right) \left(\prod_{\tilde{i}=1}^{Nn}\epsilon(e_{\tilde{i}\tilde{\sigma}(\tilde{i})})\right)
\tilde{\Phi}_q(\tilde{e}_{1\tilde{\sigma}(1)})\dots\tilde{\Phi}_q(\tilde{e}_{Nn\tilde{\sigma}(Nn)})
\end{gather*}
where $\tilde{m}\in m$ means that $\tilde{m}$ projects to $m$.  Note $\prod_{\tilde{i}=1}^{Nn}\epsilon(e_{\tilde{i}\tilde{\sigma}(\tilde{i})})=\prod_{e}\epsilon(e)^{m_e}$ in fact only depends on $m$, independent of $\tilde{m}\in m$.

The goal is to show that the interior sum is the quantum trace of $m$, up to sign, that is,  
\be\label{eqn:kDetThm}
s \trqc(\Phi_q,m) = q^{-N\binom{n}{2}}\left(\prod_{e}\epsilon(e)^{m_e}\right)\sum_{\tilde{m} \in m} (-1)^{\tilde{\sigma}}\left(\prod_v q^{\ell(\sigma_v)}\right)\left(\prod_e q^{\binom{m_e}{2}} q^{\ell(\sigma_e)}\right) \tilde{\Phi}_q(\tilde{e}_{1\tilde{\sigma}(1)})\dots\tilde{\Phi}_q(\tilde{e}_{Nn\tilde{\sigma}(Nn)})
\ee
for some constant sign $s$ independent of $m$. 

Now, for a fixed multiweb $m$, starting from the RHS of  (\ref{eqn:kDetThm}),
$$
= q^{-N\binom{n}{2}}\left(\prod_{e}\epsilon(e)^{m_e}\right)\sum_{c}\sum_{\tilde{m} \in c} (-1)^{\tilde{\sigma}}\left(\prod_v q^{\ell(\sigma_v)}\right)\left(\prod_e q^{\binom{m_e}{2}}q^{\ell(\sigma_e)}\right) \tilde{\Phi}_q(\tilde{e}_{1\tilde{\sigma}(1)})\dots\tilde{\Phi}_q(\tilde{e}_{Nn\tilde{\sigma}(Nn)})  \label{eqn:step1}
$$
where the sum over $c$ is the sum over half-edge colorings of $m$.
$$
= q^{-N\binom{n}{2}}\left(\prod_{e}\epsilon(e)^{m_e}\right)\sum_{c} \sum_{\tilde{\sigma} \in c} (-1)^{\tilde{\sigma}_{0,c}}\left(\prod_v q^{\ell(\sigma_v)}\right)\left(\prod_e (-1)^{\tilde{\sigma}_e}q^{\binom{m_e}{2}}q^{\ell(\sigma_e)}\right) \tilde{\Phi}_q(\tilde{e}_{1\tilde{\sigma}(1)})\dots\tilde{\Phi}_q(\tilde{e}_{Nn\tilde{\sigma}(Nn)}) \label{eqn:step2}
$$
recalling that signs of permutations are multiplicative.  Here also `$\tilde{\sigma}\in c$' has been inserted in place of `$\tilde{m}\in c$' as the map $\tilde{m}\mapsto\tilde{\sigma}$ becomes injective over fixed colored multiwebs $(m,c)$.  (It need not be injective over just multiwebs $m$.)  Note that $(-1)^{\tilde{\sigma}_e}=(-1)^{\sigma_e}$.  (This is because $\tilde{\sigma}_e$ is `conjugate' to $\sigma_e$ in a sense; note their lengths need not agree.)
$$
= q^{-N\binom{n}{2}}\left( \prod_{e}\epsilon(e)^{m_e} \right) \sum_{c} (-1)^{\tilde{\sigma}_{0,c}}\left(\prod_v q^{\ell(\sigma_v)}\right) \sum_{\tilde{\sigma} \in c} \left(\prod_e q^{\binom{m_e}{2}}(-q)^{\ell(\sigma_e)}\prod_{s\in S_e}\Phi_q(e)_{s,h_e(s)} \right)\label{eqn:step3}
$$
where $h_e$ from $S_e$ to $T_e$ is defined in (\ref{enum:item:definitionofh}) in Definition \ref{ssec:qkastdet}. Also, the powers of $q$ associated to the $\sigma_v$  have been factored out (since these depend only on $c$ and not $\tilde{\sigma}\in c$), and the commutativity of the $\Phi_q$ on different edges has been used.  Lastly, as it stands, the rightmost product over $S_e$ is not defined, due to the noncommutativity of variables over a single edge:  the convention is used that the terms are ordered from left to right according to the natural (increasing) order of $S_e$ (namely, the only convention such that the previous equation is valid).
$$
= q^{-N\binom{n}{2}}\left( \prod_{e}\epsilon(e)^{m_e} \right) \sum_{c} (-1)^{\tilde{\sigma}_{0,c}}\left(\prod_v q^{\ell(\sigma_v)}\right) \left(\prod_e q^{\binom{m_e}{2}} \sum_{\sigma_e\in c|_e} (-q)^{\ell(\sigma_e)}\prod_{s\in S_e}\Phi_q(e)_{s,h_e(s)}\right)   \label{eqn:step4}
$$
after exchanging the sum and the product (again using the commutativity of variables across different edges).
\be\label{eqn:step5}
= q^{-N\binom{n}{2}} \left( \prod_{e}\epsilon(e)^{m_e} \right) \sum_{c} (-1)^{\tilde{\sigma}_{0,c}}\left(\prod_v q^{\ell(\sigma_v)}\right) \left(\prod_e q^{\binom{m_e}{2}} \det_{q,S_e,T_e}(\Phi_q(e))\right) 
\ee
using the definition of the quantum minor.

By comparing (\ref{eqn:step5}) and (\ref{alttrace}), it is clear there is a correspondence between their terms respecting powers of $q$ (omitting the leading $q^{-N\binom{n}{2}}$ term), and thus all that remains is to compare their signs.  These signs are, however, identical to the classical ($q=1$) case, and the equivalence of signs is thus demonstrated in \cite{DKS,KenyonOvenhouse}; see Theorem \ref{theo:classicalkasteleyntheorem}.  This completes the proof.

\section{The circulation of a multiweb}\label{ssec:randomvariable} 

Let $G$ be planar, with positive ciliation $L=L^+$, so that $\Zq=\Zq^+$.  For $m\in\Omega_n$ put 
\be\label{eq:normalizedqtrace}
\htrq(m):=q^{-N\binom{n}{2}}\trq(m)
\ee 
which is a (nonzero) symmetric Laurent polynomial in $q$ with natural number coefficients, by Theorem \ref{theo:symmetricZq}.  Note 
\be\label{eq:firstderivativeiszero}
\left.\frac{d}{dq}\right|_{q=1}\htrq(m)=0
\ee 
by symmetry, and by definition 
$$
\Zq^+=\sum_{m\in\Omega_n}\htrq(m).
$$

Define a random variable (i.e. function) 
$$
X_n : \Omega_n \to \mathbb{Q}_{\geq0}
$$ 
by 
$$
X_n(m):=\left.\frac{d^2}{dq^2}\right|_{q=1}\log\htrq(m)=\frac{1}{\mathbf{tr}_1(m)}\left.\frac{d^2}{dq^2}\right|_{q=1}\htrq(m)
$$ 
where $\mathbf{tr}_1(m)$ is the classical trace (Definition \ref{defi:classicaltrace}).    Note $X_n$ is indeed valued in $\mathbb{Q}_{\geq0}$ by the positivity of $\htrq(m)$.    We call $X_n(m)$ the \emph{circulation} of the multiweb $m$. 

\begin{rema}\label{rema:rvindependentofcilia}
Because changing the ciliation $L$ only changes the quantum trace $\htrq(m)$ by a sign, it follows that $X_n\geq0$ is actually independent of the choice of cilia (including nonpositive cilia).  
\end{rema}

Two probability measures on $\Omega_n$ are now described.   The \emph{natural measure} $P$ is defined by 
$$
P(m):=\frac{\mathbf{tr}_1(m)}{\mathbf{Z}_1^+}
$$ 
where $\mathbf{Z}_1^+$ is the classical $n$-dimer partition function (Remark \ref{rema:classicalndimerpartitionfunction}).  The \emph{uniform measure} $P^u$ is defined by 
$$
P^u(m):=\frac{1}{|\Omega_n|}.
$$  
Of interest are the two expectation values 
$$
\E(X_n):=\sum_{m\in\Omega_n}X_n(m)P(m)
$$ 
and, to a lesser extent (at least in this paper), 
$$
\E^u(X_n):=\sum_{m\in\Omega_n}X_n(m)P^u(m).
$$ 

\begin{prop}\label{secondderiv}
The random variable $X_n$ satisfies the following enjoyable properties.
\begin{enumerate}
\item  
The expectation value with respect to the natural measure $P$ can be computed as 
\be\label{eq:expectationvalueofXn}
\E(X_n)=\left.\frac{d^2}{dq^2}\right|_{q=1}\log\Zq^+=\frac{1}{\mathbf{Z}_1^+}\left.\frac{d^2}{dq^2}\right|_{q=1}\Zq^+.
\ee
\item  
$X_n(m)$ depends only on the isotopy class of $m$ (respecting edge multiplicities).  
\item  
$X_n(m)$ is additive on the connected components of $m$.  
\item  
$X_n(m)=0$ on trivial components (edges of multiplicity $n$).  
\end{enumerate}
\end{prop}

\begin{proof}
The first item is a simple calculation from the definitions, requiring the positivity.  The second item follows from the isotopy invariance of the quantum trace.  The third item is an elementary computation using (\ref{eq:firstderivativeiszero}) together with the multiplicative property of the quantum trace (Proposition \ref{prop:multiplicative}), which also holds for the normalized version (\ref{eq:normalizedqtrace}).  The last item is because $\htrq(m)=1$ for trivial $m$.  
\end{proof}

\begin{rema}
By Theorem \ref{theo:dependenceonplanarembedding}, the random variable $X_n$ depends only on the combinatorial structure of $G$, not on any particular planar embedding or ciliation (see also Remark \ref{rema:rvindependentofcilia}).  This is evident in the case $n=2$ by Proposition \ref{prop:n=2rv}.  
\end{rema}

\subsection{Related (local) random variable}\label{sssec:localrv}

In practice, since $X(m):=X_n(m)\geq0$ is independent of the choice of any, not necessarily positive, cilia (Remark \ref{rema:rvindependentofcilia}), it is advantageous to do computations with respect to a trivial ciliation $L$ (constructed, for instance, as in the proof of Proposition \ref{id} or as in Remark \ref{3conn}).  Then, by (\ref{eq:formulaoftracefordiagonalconnection}) (see also \ref{eq:trivialciliaformula}),
$$
X(m)
=\frac{1}{\mathbf{tr}_1(m)}\left.\frac{d^2}{dq^2}\right|_{q=1}\htrq(m)
=\frac{1}{|\mathbf{tr}_1(m)|}\left.\frac{d^2}{dq^2}\right|_{q=1}\sum_c\prod_{e}q^{\binom{m_e}{2}}\prod_v q^{\hat{\ell}(\sigma_{cv}^L)}
=\frac{1}{|\mathbf{tr}_1(m)|}\left.\frac{d^2}{dq^2}\right|_{q=1}\sum_c q^{Y_c^L}
$$
where (1) $\sigma_{cv}^L\in\mathfrak{S}_n$ is the local permutation at $v$ corresponding to the edge coloring $c$ relative to the trivial ciliation $L$; (2) $\hat{\ell}(\sigma_{cv}^L):=\ell(\sigma_{cv}^L)-\frac{1}{2}\binom{n}{2}$ is the \emph{centered length}; and, (3) $Y_c^L:=\sum_v Y_{cv}^L:=\sum_v(\hat{\ell}(\sigma_{cv}^L)+\sum_{e\sim v}\frac{1}{2}\binom{m_e}{2})$.  (Recall also that $|\mathbf{tr}_1(m)|$ is the number of edge colorings $c$ of $m$.)  Evaluating the second derivative, 
$$
X(m)
=\frac{1}{|\mathbf{tr}_1(m)|}\sum_c Y_c^L(Y_c^L-1)
=\frac{1}{|\mathbf{tr}_1(m)|}\sum_c (Y_c^L)^2
$$
where the last equality follows by the symmetry of the quantum trace, that is, by $\frac{d}{dq}|_{q=1}\htrq(m)=\sum_c Y_c^L$ together with (\ref{eq:firstderivativeiszero}) (which is valid for nonpositive cilia, as well).  

One can think of $Y^L(c):=\sum_v Y^L_v(c):=Y_c^L=\sum_v Y_{cv}^L$ for trivial ciliation $L$ as a `local' random variable (i.e. depending only on the behavior at the vertices $v$), called the \emph{total centered length (relative to $L$ and adjusted for edge multiplicities $m_e$)}, as follows.  The domain of $Y^L$ is the Cartesian product $(\Omega_1)^n$, which can be thought of as the set of colored multiwebs $(m,c)$ (the $i$-th dimer cover corresponds to the $i$-th color of $m$).  For the corresponding uniform probability measure, where $c$ gets probability $\frac{1}{|\Omega_1|^n}=\frac{1}{\mathbf{Z}_1^+}$, the expectation value of $Y^L$ is $0$ by symmetry, as just explained.  Its variance is related to $X$, from (\ref{eq:expectationvalueofXn}) and the above calculation, by
$$
\E(X)
=\frac{1}{\mathbf{Z}_1^+}\sum_m\left.\frac{d^2}{dq^2}\right|_{q=1}\sum_c q^{Y_c^L}
=\sum_m\sum_c \frac{(Y_c^L)^2}{\mathbf{Z}_1^+}
=\mathrm{Var}(Y^L).
$$

\begin{rema}
\begin{enumerate}
\item  
Unlike $X$, the local random variable $Y^L$ does indeed depend on the choice of trivial ciliation $L$.
\item  
This `localization' of $X_n$ is taken advantage of in Section \ref{loopdensity}, when $n=2$, to compute the expected density of loops in the infinite honeycomb and square grid graphs.  
\end{enumerate}
\end{rema}

\subsection{The \texorpdfstring{$n=2$}{n=2} case}\label{sssec:probabilitiesn=2}

When $n=2$, there are two types of connected components of multiwebs $m\in\Omega_2$:  loops and doubled edges.  It is not hard to see, from multiple perspectives, that $\htrq(\text{loop})=[2]$ and $\htrq(\text{doubled edge})=1$.  By the multiplicative property of quantum traces, one gathers 
\be\label{eq:n=2partitionfunction}
\Zq^+=\sum_{m\in\Omega_2}[2]^{L(m)}
\ee 
where $L:\Omega_2\to\Z_{\geq0}$ is the random variable defined by 
$$
L(m):=\text{the number of loop components of }m.
$$  

(In the general $n=n$ case, one can similarly compute the quantum trace in terms of the winding numbers of colored loops; see \cite{moy} as explained in \cite[Section 1.7]{sikora1}.)  

The following statement further motivates the higher rank random variables $X_n$.  

\begin{prop}\label{prop:n=2rv}
The two random variables $X_2=L$ coincide.  
\end{prop}

\begin{proof}
This is a simple calculation from (\ref{eq:n=2partitionfunction}).  
\end{proof}

Due to the explicit form (\ref{eq:n=2partitionfunction}) of the $n=2$ quantum partition function $\Zq^+$, as $[2]=q+q^{-1}=1$ when $q=e^{i\frac{\pi}{3}}$ one can also easily compute the expectation value of $L$ with respect to the uniform measure $P^u$ as a first derivative.

\begin{prop}\label{prop:uniform_measure}
One has
$$
 \E^u(L) = \frac{1}{\sqrt{3}  e^{i\frac{\pi}{6}}}  \left. \frac{d}{dq} \right|_{q=e^{i\frac{\pi}{3}}} \log\Zq^+.  
$$
\qed\end{prop}

\section{Examples}
\label{sec:examples}

\subsection {Cycle graph}

Let $C_N$ be the cycle graph in $\R^2$ with $2N$ vertices ($N$ white and $N$ black) and $2N$ edges.  The $n$-multiwebs in $C_N$ (of which there are only $n+1$) can be described as follows. Pick an edge such that going from black-to-white goes in the counter-clockwise direction around the cycle. For each $k \in \{0,1,\dots,n\}$, there is a multiweb $m$ whose edge multiplicities alternate $k,n-k,k,n-k,\dots$ around the cycle, starting with the given edge.  We will compute the traces of these multiwebs via (\ref{eq:formulaoftracefordiagonalconnection}) taking $\Phi=I_q$ to be a quantum identity connection. To do so, we will choose the cilia all pointing outside of the cycle (which is a positive ciliation $L=L^+$), and we will represent the connection by putting the matrix $Q^{N-1}$ on the distinguished edge.  (See Section \ref{ssec:quantumidentityconnection} for the definition of $Q$.)

For a fixed $k$, let $m_k$ be the multiweb described above. There are precisely $\binom{n}{k}$ edge colorings, indexed by the different subsets $I \in \binom{\{1,2,\dots,n\}}{k}$. Let $I = \{i_1, i_2,\dots,i_k\}$ be written in order ($i_1 < i_2 < \dots < i_k$) and let $J = \{1,2,\dots,n\} \setminus I = \{j_1,j_2,\dots,j_{n-k}\}$ be the complement, also indexed in increasing order. In (\ref{eq:formulaoftracefordiagonalconnection}), the permutation $\sigma_v$ is the same at all vertices, and is given by 
$$ 
\sigma = i_1 i_2 \dots i_k j_1 j_2 \dots j_{n-k}. 
$$
Therefore, the factor of $\prod_v (-q)^{\ell(\sigma_v)}$ in (\ref{eq:formulaoftracefordiagonalconnection}) will be $q^{2N \ell(\sigma)}$. Half the edges have multiplicity $k$, the other half have multiplicity $n-k$, so the product of $q^{\binom{m_e}{2}}$ factors becomes  $q^{N( \binom{k}{2} + \binom{n-k}{2})} = q^{N ( \binom{n}{2} - k(n-k) )}$. Also, the matrices $\phi_{bw}$ are the identity matrix at all but one edge, where it is $\phi = Q^{N-1}$. So the $\prod_e \prod_{i \in S_e} (\phi_e)_{ii}$ product is $q^{( N - 1 )\sum_{i \in I} (n+1-2i)}$. All together, the contribution from (\ref{eq:formulaoftracefordiagonalconnection}) corresponding to a given coloring $I$ becomes
$$ 
q^{2N  \ell(\sigma)} q^{N\binom{n}{2}-Nk(n-k)} q^{( N-1 ) \left(\sum_{i \in I} n+1 - 2i\right)} 
= q^{N\binom{n}{2}+N(k^2+k) - k(n+1)} q^{2N  \ell(\sigma) - (2N-2)\left(\sum_{i \in I} i\right)}.  
$$
Note that the first factor on the right hand side ($q$ to the power $N\binom{n}{2}+N(k^2+k) - k(n+1)$) depends only on $k$, and not on the particular coloring $I$.

The permutation $\sigma$ is a Grassmannian permutation (i.e. it has at most 1 descent). There is a well-known bijection between such permutations and partitions whose Young diagram fits inside a $k \times (n-k)$ rectangle.   Under this bijection, the length of the permutation $\sigma$ corresponding to $I$ is equal to $|\lambda|$ (the size of the partition), and also
\be\label{eq:grassmannpermutation} 
\sum_{i \in I} i = |\lambda| + \binom{k+1}{2}.
\ee 

Putting this all together, and making the substitutions $\ell(\sigma) = |\lambda|$ and $(\ref{eq:grassmannpermutation})$, we get that the contribution from the coloring $I$ (except for the $q^{N\binom{n}{2}}$ term) is 
$$
q^{N(k^2+k) - k(n+1)+2N  \ell(\sigma) - (2N-2)\left(\sum_{i \in I} i\right)}=q^{2|\lambda| - k(n-k)}.
$$ 
Summing over all possible $I \in \binom{\{1,2,\dots,n\}}{k}$ (equivalently, summing over all corresponding partitions), we get
$$ 
\trq(m_k) = q^{N\binom{n}{2}-k(n-k)} \sum_\lambda q^{2{|\lambda|}}. 
$$
It is well-known that this sum is related to the Gaussian $q$-binomial coefficients. Specifically, 
$$
\sum_\lambda q^{2|\lambda|} = q^{k(n-k)} \begin{bmatrix} n \\ k \end{bmatrix}.
$$ 
The extra $q^{k(n-k)}$ factor cancels, and at last we see that 
$$
\trq(m_k) = q^{N\binom{n}{2}}\begin{bmatrix} n \\ k \end{bmatrix}. 
$$ 
The total partition function $\Zq^+$ from (\ref{Zqdef}) is the sum over all $k$, and by a version of the $q$-binomial theorem, this factors as
\be\label{eq:Zqcycle} 
\Zq^+ = \sum_{k=0}^n \begin{bmatrix} n \\ k \end{bmatrix}  = \prod_{i=1}^n (1 + q^{n+1-2i}) = \Det(I+Q). 
\ee
Therefore, this quantum partition function $\Zq^+$ is particularly simple (it can be computed in polynomial time).  It is not expected for general $\Zq^+$ to have such determinantal formulas.

\begin{exam}
Let us compute the random variable $X_n$ from Section \ref{ssec:randomvariable} for this cycle graph.  
As a warm up, for $n=3$ and the multiweb $m_1$, the normalized trace $\htrq(m_1)=[3]$ and one computes 
$$
X_3(m_1)=\frac{1}{\mathbf{tr}_1(m_1)}\left.\frac{d^2}{dq^2}\right|_{q=1}(q^{-2}+1+q^2)=\frac{8}{3}.
$$  
For general $n$, by a simple computation the expectation value of $X_n$ is, by (\ref{eq:expectationvalueofXn}) plus (\ref{eq:Zqcycle}), 
$$
\E(X_n)=\frac{n^3-n}{12} = \frac{1}{2} \binom{n+1}{3}.
$$
\end{exam}

\subsection{Snake graphs}

We will compute the $n=2$ quantum partition function $\Zq^+(G)$ for two families of so-called \emph{snake graphs}, consisting of a single sequence, or `snake', of boxes placed horizontally and vertically. They are certain types of skew Young diagrams, also called `border strips' or `ribbon shapes', which are indexed by rational numbers $\frac{r}{s} \in \mathbb{Q}$ (using continued fractions).  Certain versions of the classical (not quantum) dimer partition functions of these graphs  appear as cluster variables in cluster algebras of type $A$ \cite{msw_13}, and also can be used to compute the $q$-deformed rational numbers $\left[ \frac{r}{s} \right]_q$ of Morier-Genoud and Ovsienko \cite{mgo_20}.  See \cite{mosz_23} for a study of classical higher rank dimers on snake graphs, building on the $n=1$ case \cite{cs_18}.  

\subsubsection {Horizontal snake}

Consider the graph $G_m$, which is the $2 \times m$ grid (formed by $m-1$ squares attached left-to-right).  It is well-known (and an easy exercise to show) that the number of dimer covers of this graph is $f_m$, the Fibonacci number.  Let $\Zq^+(G_m)$ be its quantum partition function for $n=2$. As a convention, we define $\Zq^+(G_0) = 1$. 

\begin{prop}
The generating function $F(x)$ for the sequence $z_m = \Zq^+(G_m)$ is given by
$$ 
F(x) = \sum_{m=0}^\infty z_m x^m = \frac{1-x}{1 - 2x - [2] x^2 + x^3}. 
$$
\end{prop}

\begin {proof}
By the well-known relationship between rational generating functions and linear recursive sequences, the result can be obtained by a simple calculation   once we know that the coefficients satisfy the linear recurrence
$$ 
z_m = 2 z_{m-1} + [2] z_{m-2} - z_{m-3}
$$
together with the initial conditions $z_0 = z_1 = 1$ and $z_2 = 2 + [2]$.   We will obtain this recurrence as a consequence of another recurrence, which is more apparent but less elegant.   Consider the top-right vertex of $G_m$. Any double dimer cover must use exactly 2 of the edges adjacent to this vertex (possibly with multiplicity).   There are 3 cases, displayed in Figure \ref{fig:Recurrence}.    In the first case, the double edges contribute a factor of $1$, and what remains is a double dimer cover of $G_{m-1}$.    In the second case, the double edges again contribute a factor of $1$, and the remaining part is a double dimer cover of $G_{m-2}$.   In the third case, there is a cycle surrounding at least the last square. This divides into $m-1$ cases, since this cycle could surround   the last $k$ squares for any $k=1,2,\dots,m-1$. As seen before, a cycle contributes a factor of $[2]$ to the trace,  and the remaining part is a   double dimer cover of $G_{m-k-1}$. We therefore have the recurrence
$$ 
z_m = z_{m-1} + z_{m-2} + [2] \sum_{k=1}^{m-1} z_{m-k-1} = z_{m-1} + z_{m-2} + [2] \sum_{k=0}^{m-2} z_k. 
$$
Taking the difference of two instances of this equation (for $m$ and $m-1$), we get
$$ 
z_m - z_{m-1} = z_{m-1} + [2] z_{m-2} - z_{m-3}. 
$$
Rearranging gives the desired recurrence.
\end {proof}

\begin {rema} \label{rmk:uniform_gf}
When $q = e^{i\frac{\pi}{3}}$, then $[2] = 1$, and we get $\frac{1-x}{1-2x-x^2+x^3}$. This expression appeared in \cite{mosz_23} as the generating function  for the number of double dimer covers of the graphs $G_m$.
\end {rema}

\begin {exam}
Using the results of Section \ref{sssec:probabilitiesn=2}, we can compute the average number of loops of double dimer covers on $G_m$. The second derivative of this generating function with respect to $q$, at $q=1$, is
$$ 
\left. \frac{d^2}{dq^2} \right|_{q=1} F(x) = \frac{2x^2(1-x)}{(1-2x-2x^2+x^3)^2}. 
$$
If we expand this rational expression as a series $\sum_m c_m x^m$, then the expected number of loops in a double dimer cover of $G_m$   (with respect to the natural measure)   is equal to $\frac{c_m}{f_m^2}$, where $f_m$ is the $m$th Fibonacci number. One can analyze the asymptotics (for example by taking the partial fraction expansion of this rational function) to see that the expected number of loops grows linearly with $m$, and the growth rate is
$$ 
\lim_{m \to \infty} \frac{1}{m} \E(L) = \lim_{m \to \infty} \frac{1}{m}\frac{c_m}{f_m^2} = \frac{\sqrt{5}-1}{5} = \frac{2}{5 \varphi} \approx 0.2472 
$$
where $\varphi$ is the golden ratio.
\end {exam}

\begin {exam}
One can similarly use Proposition \ref{prop:uniform_measure} to say something about the distribution of the number of loops  with respect to the uniform measure on $\Omega_2(G_m)$. Here, we get
$$ 
\frac{1}{\sqrt{3}e^{i\frac{\pi}{6}}}\left. \frac{d}{dq} \right|_{q=e^{i\frac{\pi}{3}}} F(x) = \frac{x^2(1-x)}{(1-2x-x^2+x^3)^2}. 
$$
Let $a_m$ be the coefficients of the series expansion of this rational expression, and let $b_m$ be the coefficients of $\frac{1-x}{1-2x-x^2+x^3}$ (see Remark \ref{rmk:uniform_gf} above). Then $\frac{a_m}{b_m} = \E^u(L)$ is the expected number of loops in a double dimer cover of $G_m$, with respect to the uniform measure. Again, it grows linearly with $m$, with growth rate
$$ 
\lim_{m \to \infty} \frac{1}{m} \E^u(L) = \lim_{m \to \infty} \frac{1}{m}\frac{a_m}{ b_m} = \frac{1}{7} ( 1 + 2\rho - \rho^2 ) \approx 0.1938 
$$
where $\rho = 2\cos(\frac{\pi}{7})$ is the root of $1-2x-x^2+x^3$ of largest magnitude.
\end {exam}

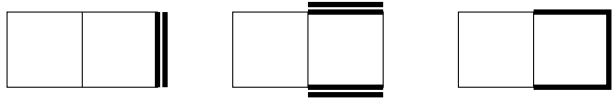
\begin{figure}[t]
    \centering
    \begin{tikzpicture}
        \foreach \x in {0, 3, 6} {
            \draw (\x,0) -- (\x+2,0) -- (\x+2,1) -- (\x,1) -- cycle;
            \draw (\x+1,0) -- (\x+1,1);
        }

        \draw[black, line width=2] (2,1) -- (2,0);
        \draw[black, line width=2] (2.1,1) -- (2.1,0);

        \begin{scope}[shift={(3,0)}]
            \draw[black, line width=2] (2,1) -- (1,1);
            \draw[black, line width=2] (2,1.1) -- (1,1.1);
            \draw[black, line width=2] (2,0) -- (1,0);
            \draw[black, line width=2] (2,-0.1) -- (1,-0.1);
        \end{scope}

        \begin{scope}[shift={(6,0)}]
            \draw[black, line width=2] (1,1) -- (2,1) -- (2,0) -- (1,0);
        \end{scope}
    \end{tikzpicture}
    \caption{Recurrence for double dimer covers on the $2\times m$ grid.}
    \label{fig:Recurrence}
\end{figure}

\subsubsection{Zig zag snake}

Now, we consider the other extreme case of $G_{m}$ where the graph is a `zig zag' or `staircase' with $m-1$ boxes,     where boxes alternately go up, right, up, right, etc.

It is not hard to see that (1) there are $m$ dimer covers; and, (2) any double dimer cover can have at most $1$ cycle, and, after choosing a cycle, there is a unique way to complete it to a valid double dimer cover using only doubled edges.  Therefore,     
$$ 
\Zq^+(G_m) = m + \binom{m}{2} [2]. 
$$ 
From this, by direct calculation (and putting $\Zq^+(G_0)=1$), the generating function is
$$ 
F(x) = \sum_{m=0}^\infty \Zq^+(G_m) x^m = 1+\frac{x}{(1-x)^2}+\frac{[2]x^2}{(1-x)^3}.
$$

\subsection{Density of loops for the infinite honeycomb}\label{loopdensity}

Let $\{H_m\}_{m\ge1}$ be a sequence of finite graphs converging as $m\to\infty$ to the infinite honeycomb graph $H$,   and such that the uniform measures on $\Omega_1(H_m)$ converge (weakly) to the unique maximal entropy measure on $\Omega_1(H)$   (see \cite{Kenyon.localstats}; for example $H_m$ can be the honeycomb graph on an $m\times m$ torus).     This implies that the natural measures on $\Omega_2(H_m)$ converge weakly  to that on $\Omega_2(H)$.   On $H_m$ the expected number of loops in a random $2$-multiweb is proportional to its number of vertices, with constant of proportionality $\rho_m$.   The sequence $\rho_m$ converges to a limit  $\rho$ which is the asymptotic density of loops for the (maximal entropy measure on the) honeycomb graph.  We show how to compute it here.

We make use of the following Morse-theoretic lemma, easily proved by deformation (or consider the Morse function $y$ along the loop).

\begin{lemma}\label{morse} 
Let $\gamma$ be a closed polygonal loop in the plane, with no horizontal edges, oriented CCW. Then the number of right-to-left local maxima, minus the left-to-right local maxima, plus left-to-right local minima, minus right-to-left local minima, is $2$. For a CW oriented loop, this sum is $-2$. 
\qed\end{lemma} 

See Figure \ref{morseloop} for an example.

Let $G=(V,E)$ be a subgraph of the honeycomb, with all faces hexagons.  Rotate $G$ so that one of the edge directions is vertical with black vertices at the lower end of the vertical edges.

We have shown that the partition function for double dimers is  $\Zq^+ = \sum_{m\in\Omega_2}(q+\frac1q)^{L}$ where $L=L(m)$ is the number of loops in $m$. By Proposition \ref{secondderiv} together with Proposition \ref{prop:n=2rv}, setting $q=e^\epsilon$,  
$$
\frac{\Zq^+}{\mathbf{Z}_1^+} =1+\frac{\epsilon^2}2\E(L)+O(\epsilon^3)
$$ 
where $\E(L)$ is the expected number of loops.

The local computation of $\Zq^+$ from Theorem \ref{theo:quantumkasteleyndet} leads us to the following method to  compute $\E(L)$ using a classical (that is, not quantum) computation. Suppose a double dimer cover $m$ is obtained by superimposing two independent single dimer covers, one colored with color $1$ and one colored with color $2$. Each loop of $m$ then has a $1-2$ coloring.   We assign a weight $q^\frac{r}{2}$ to such a colored double dimer configuration, where $r$ counts the signed number of local maxima, plus the signed number of local minima, with signs as shown in Figure \ref{maxmin}  (orienting from black to white along edges colored $1$).

By Lemma \ref{morse}, summing over both colorings of each loop indeed gives weighting $q+\frac{1}{q}$ per loop.    We can thus write 
\be\label{Zq1}
\Zq^+ = \sum_{m\in\Omega_2}\sum_c q^{\frac12\sum_vX_v}
\ee
where the sum over $c$ is the sum over colorings of the edges of $m$ with colors $\{1,2\}$, and $X_v$ is $1,-1$ or $0$ according to Figure \ref{maxmin}, that is, if either edge is vertical at $v$, or both edges are the same,  then $X_v=0$; otherwise $X_v=\pm1$ as indicated.  (Alternatively, one can use, instead of $\frac12\sum_vX_v$, the local random variable $Y^L=\sum_v Y^L_v$ from Section \ref{sssec:localrv} for the trivial ciliation $L$ where the cilia all point to the left, say, the result being that vertex contributions at vertices which are not local maxima or minima cancel out, while doubled edges contribute zero.)

Expanding (\ref{Zq1}) with $q=e^\epsilon$ we have
\begin{align*}
\Zq^+ &= \sum_{m\in\Omega_2}\sum_c e^{\frac12\epsilon\sum_v X_v}\\
&=\sum_{m\in\Omega_2}\sum_c \left(1+\frac12\epsilon\sum_v X_v+\frac{\epsilon^2}8\left(\sum X_v\right)^2+O(\epsilon^3)\right)\\
&=\sum_{m\in\Omega_2}\sum_c \left(1+\frac{\epsilon^2}8\left(\sum X_v\right)^2\right)
\end{align*}
ignoring higher order terms, where we used that $X_v$ has expectation $0$ (since exchanging colors negates $X_v$). This gives
$$
\frac{\Zq^+}{\mathbf{Z}_1^+}=1+\frac{\epsilon^2}8\E\left(\left(\sum X_v\right)^2\right)
$$
where the expectation is with respect to the uniform measure on colored double dimer covers, or equivalently the  natural measure on double dimer configurations.    Comparing with the above we conclude that 
$$
\E(L) = \frac14\left(\sum_u\E(X_u^2)+\sum_{u\ne v}\E(X_uX_v)\right).
$$

Consider the whole-plane honeycomb $H$. We parameterize vertices in $H$ with $\Z^2$, as shown in Figure \ref{Hcartesian}.   Let $H_n$ be the $n\times n$ honeycomb on the torus, the quotient of $H$ by translations $(n,0)$ and $n(\frac12,\frac{\sqrt{3}}2)$. As $n\to\infty$ the measures $\mu_n$ on $\Omega_2(H_n)$ converge to the unique maximal entropy measure $\mu$ on $\Omega_2(H)$, see \cite{Kenyon.localstats}.

By translation invariance of $H_n$ and $H$,   the expected number of loops `per vertex' $\rho_n:=\frac{\E(L)}{|V_n|}$ on $H_n$ converges  as $n\to\infty$ to 
$$
\rho = \frac14\left(\E(X_u^2)+\sum_{v:~v\ne u}\E(X_uX_v)\right)
$$
where $u$ is fixed to be the vertex at the origin, and the expectations are with respect to the measure $\mu$. We have $\E(X_u^2)=\frac29$ since the probability of each edge in a single dimer cover is $\frac{1}{3}$.

Suppose without loss of generality that $u$ is white: $u=w_{0,0}$. Let $e_1,e_2$ be the edges right and left from $u$.

First suppose $v$ is also white, with edges $f_1,f_2$ to its right and left. We have 
\begin{align}\nonumber
\E(X_uX_v)&=P^1(e_1,f_1)P^2(e_2,f_2)-P^1(e_1,f_2)P^2(e_2,f_1)-P^1(e_2,f_1)P^2(e_1,f_2)+P^1(e_2,f_2)P^2(e_1,f_1)\\
&=2P(e_1,f_1)P(e_2,f_2)-2P(e_1,f_2)P(e_2,f_1)\label{PP}
\end{align}
where in the first line $P^1$ refers to the first dimer cover and $P^2$ to the second. In the second line we used equivalence of $P^1$ and $P^2$.

Recall \cite{Kenyon.localstats} that for the single dimer model, for two distinct edges $e=wb, e'=w'b'$ their joint probability is  
$$
P(e,e')=\Det\begin{pmatrix}K^{-1}(b,w)&K^{-1}(b',w)\\K^{-1}(b,w')&K^{-1}(b',w')\end{pmatrix}=\frac19-K^{-1}(b',w)K^{-1}(b,w')
$$ 
(and if the edges are the same their probability is $\frac13$). Here $K^{-1}$ is the limiting inverse Kasteleyn matrix on $H$. 

By translation invariance, $K^{-1}(b,w)$ only depends on $b-w$.  Let $B_{x,y} := K^{-1}(b_{x,y},w_{0,0}).$ If $v=w_{x,y}$ is the white vertex at $(x,y)$, the expression (\ref{PP}) is
\begin{align}\nonumber
\E(X_uX_v)&=2(\frac19-B_{x,y}B_{-x,-y})(\frac19-B_{x-1,y}B_{-1-x,-y})-2(\frac19-B_{x-1,y}B_{-x,-y})(\frac19-B_{x,y}B_{-1-x,-y})\\
\label{WW}&=-\frac29(B_{-1-x,-y}-B_{-x,-y})(B_{x-1,y}-B_{x,y})
\end{align}
unless $(x,y)=(0,0)$ (in which case it is $\frac29$). 

If $v$ is the black vertex $b_{x,y}$ at $(x,y)$ with edges $f_1$ to its right and $f_2$ to its left, the expression (\ref{PP}) gives
\begin{align}
&2(\frac19-B_{x,y}B_{-x-1,-y})(\frac19-B_{x,y}B_{-x-1,-y})-2(\frac19-B_{x,y}B_{-x,-y})(\frac19-B_{x,y}B_{-x-2,-y})\nonumber\\
&=\frac29B_{x,y}(B_{-x,-y}-2B_{-1-x,-y}+B_{-x-2,-y}+9B_{-1-x,-y}^2B_{x,y}-9B_{-2-x,-y}B_{-x,-y}B_{x,y})\label{WB}
\end{align}
unless $v=b_{0,0}$ or $v=b_{-1,0}$ in which case there is a slightly different expression, due to either $e_1=f_2$ or $e_2=f_1$.

Notice that when summing over $x,y\in\Z^2$, the quadratic terms in (\ref{WB}) exactly cancel the terms of (\ref{WW}): the term $B_{-1-x,-y}B_{x-1,-y}$ cancels the term $B_{x,y}B_{-x-2,-y}$ when shifting $x$ by $1$, and $B_{-x,-y}B_{x-1,y}$ cancels $B_{x,y}B_{-1-x,-y}$ upon changing the sign of $x$ and $y$.  Only the quartic terms of (\ref{WB}) remain (and the contributions from the exceptions near the origin, yielding (using $B_{-1,0}=B_{0,0}=\frac{1}{3}$ and the formula for $B_{-2,0}$ and $B_{1,0}$ given in Section \ref{sssec:integrals}):
\be\label{rhosum}
\rho = -\frac1{54}+\frac1{6\sqrt{3}\pi}+\frac12\sum_{x,y\in\Z}B_{x,y}^2(B_{-1-x,-y}^2-B_{-2-x,-y}B_{-x,-y}).
\ee

\begin{figure}[t]
\centering
\includegraphics[width=1.1in]{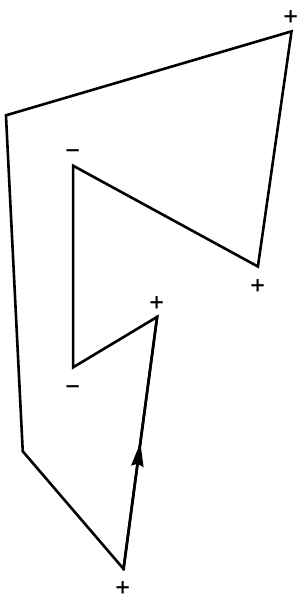}
\caption{\label{morseloop}Contribution from a polygonal loop.}
\end{figure}

\begin{figure}[t]
\centering
\includegraphics[width=6in]{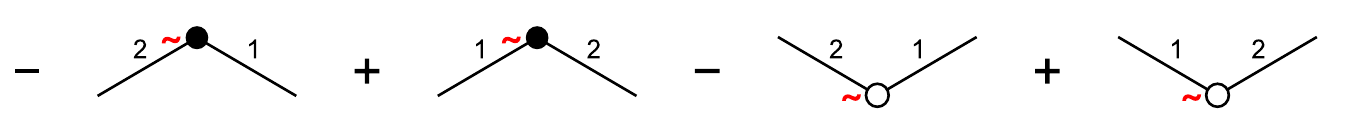}
\caption{\label{maxmin}Local max and min contributions required for the honeycomb.}
\end{figure}

\begin{figure}[t]
\centering
\includegraphics[width=3in]{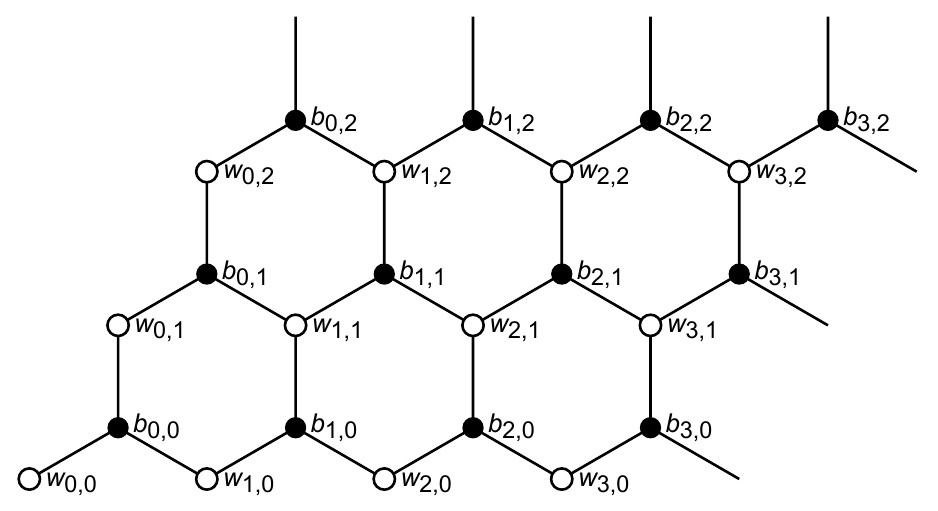}
\caption{\label{Hcartesian}Parameterizing vertices of the honeycomb.}
\end{figure}
 
\subsubsection{Integrals}\label{sssec:integrals}

From \cite{Kenyon.localstats} we have  
$$
B_{x,y} = \frac1{(2\pi i)^2}\iint_{\T^2}\frac{z^{-x} w^{-y}}{1+z+w}\frac{dz}{z}\frac{dw}w
$$ 
that is, $B_{x,y}$ are Fourier coefficients of $\frac1{1+z+w}$.

We get the square of the Fourier coefficients by convolution: 
$$
Q(a,b) := \sum_{x,y\in\Z}B_{x,y}^2a^xb^y= \frac1{(2\pi i)^2}\iint_{\T^2}\frac1{(1+z+w)(1+\frac{a}{z}+\frac{b}{w})}\frac{dz}{z}\frac{dw}{w}.
$$ 
Likewise  
$$
R(a,b) := \sum_{x,y\in\Z}B_{x,y}B_{x+2,y}a^xb^y= \frac1{(2\pi i)^2}\iint_{\T^2}\frac{z^2}{(1+z+w)(1+\frac{a}{z}+\frac{b}{w})}\frac{dz}{z}\frac{dw}{w}.
$$ 
Thus  
$$
\frac12\sum_{x,y\in\Z}B_{x,y}^2(B_{-1-x,-y}^2-B_{-2-x,-y}B_{-x,-y})=\frac1{2(2\pi i)^2}\iint_{\T^2} (aQ(a,b)^2-Q(a,b)R(a,b))\frac{da}{a}\frac{db}{b}.
$$

Summing numerically (\ref{rhosum}) over all $x,y$ with  $|x|,|y|\leq300$ gives 
$$ 
\frac{1}{27.000058...}
$$ 
indicating a probable value of $\frac{1}{27}$. However we were not able to evaluate the above integral exactly.

An analogous computation for the expected number of loops of the double dimer model on the  square grid $\Z^2$ yields a similar integral with numerical value very close to $\frac{1}{16}$. We conjecture that $\frac{1}{27}$ and $\frac{1}{16}$ are the exact values.
\appendix

\section{Grassmann variables}\label{sec:qgrassmann}  

\subsection{Classical Grassmann variables}\label{ssec:reviewofgrassmannvariables}

Grassmann integration is a compact way to encode antisymmetric structures, such as determinants. We review it here.   (See e.g. \cite{bigYellowCFT1997}.)

The \emph{Grassmann algebra} over the set $\{1,\dots,n\}$ is generated by a set of anticommuting variables $\psi_1, \psi_2, \dots, \psi_n$ over $\mathbb{C}$. This means
$$
\{ \psi_i, \psi_j \}
:=
\psi_i \psi_j + \psi_j \psi_i = 0
$$ 
for all  $i,j$.  Note  $\psi_i^2 = 0$ for all  $i$. Note also that a general element $f$ of the Grassmann algebra is represented by a polynomial with terms of order $\leq 1$ in each variable, that is
$$
f=\sum_{k=0}^{n} \sum_{1 \leq i_1 < \dots < i_k \leq n} c_{i_1,\dots,i_k} \psi_{i_1} \dots \psi_{i_k}.
$$
An element is \emph{Grassmann even (resp. odd)} if all monomials consist of an even (resp. odd) number of variables of degree $1$.  Note that even elements commute with all elements, and odd elements anticommute with odd elements.  

Introduce more anticommuting variables $d\psi_1,\dots,d\psi_n$ satisfying
$$
\{ d\psi_i, \psi_j \}
= 0
$$
for all {$i,j$}, which play the role of differentials for integration. 

Consider a nonzero monomial $d\psi_{j_1} \dots d\psi_{j_l} \psi_{i_1} \dots \psi_{i_k}$ in these variables (so $j_1,\dots,j_l$ are distinct and $i_1,\dots,i_k$ are distinct). For each $i\notin\{i_1,\dots,i_k\}\cup\{j_1,\dots,j_{l}\}$ define
\begin{gather*}
\int_i d\psi_{j_1} \dots d\psi_{j_{l}} d\psi_{i} \psi_{i} \psi_{i_1} \dots \psi_{i_k} 
:=
d\psi_{j_1} \dots d\psi_{j_{l}} \psi_{i_1} \dots \psi_{i_k},
\\
\int_i d\psi_{j_1} \dots d\psi_{j_l} \psi_{i_1} \dots \psi_{i_k} 
:= \int_i d\psi_{j_1} \dots d\psi_{j_l} d\psi_{i} \psi_{i_1} \dots \psi_{i_k} 
:= \int_i d\psi_{j_1} \dots d\psi_{j_l} \psi_{i} \psi_{i_1} \dots \psi_{i_k} 
:= 0.  
\end{gather*}
These assignments extend linearly to define a polynomial $\int_i p$ in the $\psi_j$ and $d\psi_j$, called the \emph{integral}, for every polynomial $p$ in the $\psi_j$ and $d\psi_j$.  Henceforth, we will denote by $\int d\psi_{j_1} \dots d\psi_{j_{l}} f$ the integral $\int_{j_1} \dots \int_{j_{l}} d\psi_{j_1} \dots  d\psi_{j_{l}} f$ for $f$ an element of the Grassmann algebra.  

As examples, we have
\begin{gather*}
\int d\psi_i = 0, 
\int d\psi_1 d\psi_2  \psi_2 \psi_1 = \int d\psi_1  \psi_1 = 1,
\int d\psi_1  \psi_2 \psi_1 = -\psi_2,
\int d\psi_i \psi_i = 1 = -\int \psi_i d\psi_i,\\
\int d\psi_1 d\psi_2  \psi_1 \psi_2 = -\int d\psi_2 d\psi_1  \psi_1 \psi_2 = -1,
\int d\psi_1 d\psi_2 d\psi_3 d\psi_4  \psi_2 \psi_1 = 0.
\end{gather*}

Now, let $\psi_1,\dots,\psi_n$, $\bar{\psi}_1,\dots,\bar{\psi}_n$ be two independent sets of Grassmann variables and $d\psi_1,\dots,d\psi_n$, $d\bar{\psi}_1,\dots,d\bar{\psi}_n$ their corresponding differentials such that the bar variables anticommute with the non-bar variables. Then, we have the following fact. 

\begin{prop} 
For any matrix $M \in \Mat_{n}$
$$
\int 
\left( \prod_{i=1}^{n} d\bar{\psi}_i d{\psi}_i \right) 
\exp\left(- \sum_{i,j} \bar{\psi}_i M_{ij} \psi_j\right)
= 
\Det(M).
$$
\end{prop}

\begin{proof}
First, note that the order of the differentials $d\bar{\psi}_i d{\psi}_i$ does not matter since these are even.  Next, note that since the terms of $-\sum_{i,j} \bar{\psi}_i M_{ij} \psi_j$ all commute, we can write
$$
\exp\left(- \sum_{i,j} \bar{\psi}_i M_{ij} \psi_j\right)
=
\prod_{i,j} \exp( - \bar{\psi}_i M_{ij} \psi_j )
=
\prod_{i,j} ( 1 - \bar{\psi}_i M_{ij} \psi_j )
$$
where the last equality follows from noting that $(\bar{\psi}_i \psi_j )^2 = 0$. Now after expanding, we note that the only terms contributing to the integral will have exactly one of each $\psi_1,\dots,\psi_n$ and $\bar{\psi}_1,\dots,\bar{\psi}_n$. This means we can write 
\begin{gather*}
\int 
\left( \prod_{i=1}^{n} d\bar{\psi}_i d{\psi}_i \right) 
\exp\left(- \sum_{i,j} \bar{\psi}_i M_{ij} \psi_j\right)
\\
= 
\int 
\left( \prod_{i=1}^{n} d\bar{\psi}_i d{\psi}_i \right) 
\sum_{\sigma \in \mathfrak{S}_n}
( - \bar{\psi}_1 M_{1 \sigma(1)} \psi_{\sigma(1)} ) 
( - \bar{\psi}_2 M_{2 \sigma(2)} \psi_{\sigma(2)} ) 
\dots
( - \bar{\psi}_n M_{n \sigma(n)} \psi_{\sigma(n)} )
\\
= 
\sum_{\sigma \in \mathfrak{S}_n} 
\int 
( d\bar{\psi}_n d{\psi}_n \dots  d\bar{\psi}_1 d{\psi}_1 ) 
( \psi_{\sigma(1)} \bar{\psi}_1  \dots \psi_{\sigma(n)} \bar{\psi}_n )
M_{1 \sigma(1)} \dots M_{n \sigma(n)}
\\
= 
\sum_{\sigma \in \mathfrak{S}_n} (-1)^{\sigma}
\int 
( d\bar{\psi}_n d{\psi}_n \dots  d\bar{\psi}_1 d{\psi}_1 ) 
( \psi_{1} \bar{\psi}_1  \dots \psi_{n} \bar{\psi}_n )
M_{1 \sigma(1)} \dots M_{n \sigma(n)}
\\
= 
\sum_{\sigma \in \mathfrak{S}_n} (-1)^{\sigma}
M_{1 \sigma(1)} \dots M_{n \sigma(n)}
= \Det(M).
\end{gather*}
\end{proof}

\subsection{Quantum Grassmann variables}

We discuss a quantization of the classical Grassmann variables story (Section \ref{ssec:reviewofgrassmannvariables}) and relate it to the quantum determinant and trace.

Consider an $n \times n$ quantum matrix $M$, defined as in Definition \ref{defi:quantummatrix}, whose entries we may view as belonging to some noncommutative algebra over $\mathbb{C}$, and let $q \in \mathbb{C}\setminus\{0\}$.

For the purposes of this subsection, we will actually only need the relation
\be \label{eq:mainRelation}
M_{i j} M_{k l} - q M_{i l} M_{k j}
= 
M_{k l} M_{i j} - q^{-1} M_{k j} M_{i l}
\ee 
for $1\leq i<k\leq n$ and $1\leq j<l\leq n$, which follows readily from the definition of the quantum matrix $M$.  

The \textit{quantum Grassmann algebra} is generated by variables $\psi_1, \psi_2, \dots, \psi_n$ subject to the relations  
$$
\psi_i \psi_j = -q^{-1} \psi_j \psi_i  \text{ for } i < j, 
\psi_i^2 = 0  \text{ for all } i.
$$
Letting $\bar{\psi}_i$ be the generators of another copy of the quantum Grassmann algebra, we also assume the antisymmetric cross relations
$$
{\psi}_i \bar{\psi}_j = - \bar{\psi}_j {\psi}_i  \text{ for all } i,j.
$$
Henceforth, we also assume that all quantum matrix elements $M_{ij}$ commute with all quantum Grassmann variables $\psi_i$, $\bar{\psi}_i$.

The relations $\psi_i^2 = \bar{\psi}_i^2 = 0$ imply 
$$
\left(\sum_{i,j} \bar{\psi}_i  M_{i j}\psi_j\right)^n
=
\sum_{\pi \in \mathfrak{S}_n} \sum_{\sigma \in \mathfrak{S}_n}
\bar{\psi}_{\pi(1)} \psi_{\sigma(1)} \dots \bar{\psi}_{\pi(n)} \psi_{\sigma(n)}
M_{\pi(1) \sigma(1)} \dots M_{\pi(n) \sigma(n)}.
$$
In particular, all terms in the double sum are proportional to $\bar{\psi}_{1}{\psi}_{1} \dots \bar{\psi}_{n}{\psi}_{n}$.  

To find the coefficients, the strategy is to reorder all the terms in the first index to be in order using the relation \eqref{eq:mainRelation} as follows. There will generally be many pairs of compatible terms in the expansion whose first indices of various factors are out of order. Precisely, for $1 \leq i < k \leq n$, $1 \leq j < l \leq n$, consider pairs of the form
$$
\dots 
\bar{\psi}_k \psi_{l} \bar{\psi}_i \psi_j   M_{k l} M_{i j}
\dots
,\quad
\dots 
\bar{\psi}_k \psi_{j} \bar{\psi}_i \psi_{l}   M_{k j} M_{i l}
\dots
$$
where the $\dots$ on both sides match up. All such terms whose first indices are out of order come in pairs like this.  We can add these up to get
\begin{gather*}
\dots 
( 
\bar{\psi}_k \psi_{l} \bar{\psi}_i \psi_j   M_{k l} M_{i j} 
+
\bar{\psi}_k \psi_{j} \bar{\psi}_i \psi_{l}   M_{k j} M_{i l}
)
\dots
=
\dots 
( 
\bar{\psi}_k \psi_{l} \bar{\psi}_i \psi_j   M_{k l} M_{i j} 
- 
q^{-1}
\bar{\psi}_k \psi_{l} \bar{\psi}_i \psi_j   M_{k j} M_{i l}
)
\dots
\\
=
\dots 
\bar{\psi}_k \psi_{l} \bar{\psi}_i \psi_j
( 
M_{k l} M_{i j} - q^{-1} M_{k j} M_{i l}
)
\dots
=
q^{2} 
\dots 
\bar{\psi}_i \psi_j \bar{\psi}_k \psi_{l}
( 
M_{i j} M_{k l} - q M_{i l} M_{k j}
)
\dots
\\
=
q^{2} 
\dots
( 
\bar{\psi}_i \psi_j \bar{\psi}_k \psi_{l}
M_{i j} M_{k l} + 
\bar{\psi}_i \psi_l \bar{\psi}_k \psi_{j} M_{i l} M_{k j}
)
\dots.
\end{gather*}
The first equality used the Grassmann commutation relations to rewrite  $\bar{\psi}_k \psi_{j} \bar{\psi}_i \psi_{l} = -q^{-1} \bar{\psi}_k \psi_{l} \bar{\psi}_i \psi_j$. The third equality used Equation \eqref{eq:mainRelation} as well as  $\bar{\psi}_k \psi_{l} \bar{\psi}_i \psi_j = q^{2} \bar{\psi}_i \psi_j \bar{\psi}_k \psi_{l}$. The last equality follows from $\bar{\psi}_i \psi_j \bar{\psi}_k \psi_{l} = -q^{-1}\bar{\psi}_i \psi_l \bar{\psi}_k \psi_{j}$.

As such, we have the freedom to reorder all terms (in pairs) to put the first indices in order. However, doing this process leaves a factor of $q^{2 \ell(\pi)}$ for each permutation $\pi$.  Doing this procedure removes the sum over $\pi$ and gives a factor of (compare the proof of Proposition \ref{prop:comparisonofqtraces})
$$
\sum_{\pi \in \mathfrak{S}_n} q^{2\ell(\pi)} = q^{\binom{n}{2}} [n]!.
$$

Gathering, we have
\begin{gather*}
\left(\sum_{i,j} \bar{\psi}_i  M_{i j}\psi_j\right)^n
=
q^{\binom{n}{2}} [n]!
\sum_{\sigma \in \mathfrak{S}_n}
\bar{\psi}_{1} \psi_{\sigma(1)} \dots \bar{\psi}_{n} \psi_{\sigma(n)}
M_{1 \sigma(1)} \dots M_{n \sigma(n)}
\\
=
q^{\binom{n}{2}} [n]!
\sum_{\sigma \in \mathfrak{S}_n}
(-q)^{\ell(\sigma)}
\bar{\psi}_{1} \psi_{1} \dots \bar{\psi}_{n} \psi_{n}
M_{1 \sigma(1)} \dots M_{n \sigma(n)}.
\end{gather*}

The above calculations can be summarized by the following identity.

\begin{prop}
Define 
$$
\exp_{q}( x ) := \sum_{k=0}^{\infty} \frac{x^k}{q^{\binom{k}{2}} [k]!}.
$$ 
Then,
$$
\exp_q\left(-\sum_{i,j} \bar{\psi}_i  M_{i j}\psi_j\right)
=
\sum_{k=0}^{n} (-1)^k 
\sum_{\substack{\{i_1 < \dots < i_k\} \subset \{1,\dots,n\} \\ \{j_1 < \dots < j_k\} \subset \{1,\dots,n\}}} 
\bar{\psi}_{i_1} \psi_{j_1} \dots \bar{\psi}_{i_k} \psi_{j_k}
\det_q(M,i_1 \dots i_k,j_1 \dots j_k)
$$
where
$$
\det_q(M,i_1 \dots i_k,j_1 \dots j_k):=\sum_{\sigma \in \mathfrak{S}_k}
(-q)^{\ell(\sigma)}
M_{i_1 j_{\sigma(1)}} \dots M_{i_k j_{\sigma(k)}}.
$$
\qed\end{prop}

We can similarly define the Grassmann integral, so that
$$
\int d\bar{\psi}_n d\psi_n \dots d\bar{\psi}_1 d\psi_1
\bar{\psi}_{i_1} \psi_{j_1} \dots \bar{\psi}_{i_k} \psi_{j_k}
\exp_q \left(-\sum_{i,j} \bar{\psi}_i  M_{i j} \psi_j \right)
=
 \pm q^\alpha\det_q(M,{\hat{i}}_1 \dots {\hat{i}}_{n-k},{\hat{j}}_1 \dots {\hat{j}}_{n-k})
$$
for some sign $\pm$ and some power $\alpha$ of $q$, and where the sign is $+$ and $\alpha=0$ when $k=0$.  

\section{Code}\label{app:code}

Mathematica code computing many of the quantities discussed in this paper can be found at the website of the fourth author; see \url{https://sites.google.com/yale.edu/samuelpanitch/research}
\bibliographystyle{plain}
\bibliography{references.bib}
\end{document}